\documentclass[10pt,reqno]{amsart}
\usepackage{amsmath,amsopn,amssymb,amsthm}
\usepackage[cp1251]{inputenc}
\usepackage[english]{babel}
\usepackage[pagebackref,breaklinks=true,colorlinks=true,linkcolor=blue,citecolor=blue,urlcolor=blue]{hyperref}
\usepackage{graphicx}%
\usepackage{color}

\newtheorem{theorem}{Theorem}[section]

\newtheorem{corollary}[theorem]{Corollary}

\newtheorem{lemma}[theorem]{Lemma}

\newtheorem{question}[theorem]{Question}

\newtheorem{remark}[theorem]{Remark}

\renewenvironment{proof}[1][Proof]{\noindent\textbf{#1.} }{\ \rule{0.5em}{0.5em}}

\begin{document}
\title[Submersion and homogeneous spray geometry]{Submersion and homogeneous spray geometry}
\author{Ming Xu}

\address{Ming Xu \newline
School of Mathematical Sciences,
Capital Normal University,
Beijing 100048,
P.R. China}
\email{mgmgmgxu@163.com}
%
%
\date{}

\begin{abstract}
We introduce the submersion between two spray structures and propose the submersion technique in spray geometry. Using this technique, as well as global invariant frames on a Lie group, we setup the general theoretical framework for homogeneous spray geometry. We define the spray vector field $\eta$ and the connection operator $N$ for a homogeneous spray manifold $(G/H,\mathbf{G})$ with a linear decomposition $\mathfrak{g}=\mathfrak{h}+\mathfrak{m}$. These notions generalize their counter parts in homogeneous Finsler geometry. We prove the correspondence between $\mathbf{G}$ and $\eta$ when the given decomposition is reductive, and that between geodesics on $(G/H,\mathbf{G})$ and integral curves of $-\eta$. We find the ordinary differential equations on $\mathfrak{m}$ describing parallel translations on $(G/H,\mathbf{G})$, and we calculate the S-curvature and Riemann curvature of $(G/H,\mathbf{G})$, generalizing L. Huang's curvature formulae for homogeneous Finsler manifolds.

\textbf{Mathematics Subject Classification (2010)}:
53B40, 53C30, 53C60

\textbf{Key words}: geodesic, homogeneous Finsler metric, homogeneous spray manifold,  invariant frame, parallel translation, Riemann curvature, S-curvature, submersion, spray structure
\end{abstract}
\maketitle

\section{Introduction}
\subsection{Background knowledge}
In spray geometry, we concern the {\it spray structure} $\mathbf{G}$ on a smooth manifold $M$, which is a smooth tangent vector field on $TM\backslash0$ with  the standard local coordinate representation $\mathbf{G}=y^i\partial_{x^i}-2\mathbf{G}^i\partial_{y^i}$,
where $\mathbf{G}^i=\mathbf{G}^i(x,y)$ is
positively 2-homogeneous for its $y$-entry \cite{Sh2001-1}.
Many geometric notions, geodesic, (linearly and nonlinearly) parallel translations, Riemann curvature, S-curvature, etc., can be defined for a spray structure \cite{Be1926,Sh2001-1}.

Spray geometry generalizes Finsler geometry. For a Finsler metric $F:TM\rightarrow[0,\infty)$ (see \cite{BCS2000,Sh2001-2} or Section \ref{subsection-2-4}), the induced spray structure, i.e., the {\it geodesic spray} of $F$, has the coefficients
\begin{equation}\label{300}
\mathbf{G}^i=\tfrac14g^{il}([F^2]_{x^ky^l}y^k
-[F^2]_{x^l}).
\end{equation}
The inverse problem concerns those spray structures which can not be induced by Finsler metrics, because they may exhibit interesting new geometric or dynamic phenomena. Recently, many important examples have been found for this project \cite{LM2021,LMY2019,LS2018,Ya2011}.

In this paper, Lie method is added to the research of spray geometry. Our goal is to setup a general theoretical framework for {\it homogeneous spray geometry}. A spray structure $\mathbf{G}$ on a smooth manifold $M$ is called {\it homogeneous} if $M$ admits the smooth transitive action of a Lie group $G$ which preserves $\mathbf{G}$, i.e., all $G$-actions map  geodesics to geodesics (see Section \ref{subsection-5-1} for more details).  Homogeneous spray geometry generalizes {\it homogeneous Finsler geometry} \cite{De2012}, because the geodesic spray for a homogeneous Finsler metric is automatically homogeneous.

Naturally a homogeneous spray manifold can be presented as $(G/H,\mathbf{G})$, in which $H$ is the isotropy subgroup at the origin $o=eH\in G/H$. We choose a linear decomposition $\mathfrak{g}=\mathfrak{h}+\mathfrak{m}$ for $G/H$, i.e., we have $\mathfrak{g}=\mathrm{Lie}(G)$ and $\mathfrak{h}=\mathrm{Lie}(H)$, and the tangent space $T_o(G/H)$ at $o$ is identified with $\mathfrak{m}$.
\subsection{Motivation and Hints}
The general philosophy of homogeneous geometry is that all  notions and properties can be reduced to $T_o(G/H)=\mathfrak{m}$ by $G$-actions. Guided by this thought, we ask

\begin{question}\label{main-question}
Given the decomposition $\mathfrak{g}=\mathfrak{h}+\mathfrak{m}$, how can we use $\mathfrak{m}$ to describe the geodesic, parallel translations and curvatures for a
homogeneous spray manifold $(G/H,\mathbf{G})$?
\end{question}

The hints for answering Question \ref{main-question} arise
from two sources.

One hint is from L. Huang's curvature formulae in homogeneous Finsler geometry \cite{Hu2015-1,Hu2015-2,Hu2017}, which need
 the spray vector field $\eta:\mathfrak{m}\backslash\{0\}\rightarrow\mathfrak{m}$
and the connection operator $N:\mathfrak{m}\backslash\{0\}\times\mathfrak{m}\rightarrow
\mathfrak{m}$ that he defined for a homogeneous Finsler manifold $(G/H,F)$ with a {\it reductive}
decomposition (i.e., an $\mathrm{Ad}(H)$-invariant decomposition) $\mathfrak{g}=\mathfrak{h}+\mathfrak{m}$.
So we see that
$$\mbox{{\it spray vector field and connection operator play the main roles.}}$$
If these notions can be generalized to homogeneous spray geometry, then some of L. Huang's homogeneous curvature formulae  follow naturally.

The other hint is from the recent progress in the study of
a left invariant spray structure $\mathbf{G}$ on a Lie group $G$ \cite{Xu2021-1,Xu2021-2}. Here the Lie group $G$ is viewed as the coset space $G/H=G/\{e\}$ with the unique linear decomposition $\mathfrak{g}=\mathfrak{h}+\mathfrak{m}=0+\mathfrak{g}$, which is obviously reductive.
In this context,
the left invariant frame $\{\widetilde{U}_i,\partial_{u^i},\forall i\}$
and the right invariant frame $\{\widetilde{V}_i,\partial_{v^i},\forall i\}$ can be defined  on $TG$ (see Section 2.1 in \cite{Xu2021-1} or Section \ref{subsection-4-1}).
Using these frames, the spray vector field $\eta$ and connection operator $N$ can be defined for $(G,\mathbf{G})$, and
the S-curvature formula, the Riemann curvature formula, etc., for $(G,\mathbf{G})$, can be globally presented. If we use $\eta$ and $N$ to translate these curvature formulae, we see again L. Huang's formulae in \cite{Hu2015-1}. By left translations, geodesics and parallel translations along smooth curves on $(G,\mathbf{G})$ are corresponded to integral curves of $-\eta$ and some ordinary differential equations (i.e., ODEs in short) on $T_eG=\mathfrak{g}$ respectively. See \cite{Xu2021-1,Xu2021-2} or Section 4 for the precise statements. To summarize,
Question \ref{main-question} can be answered for a left invariant spray structure, using the technique of global invariant frames.
Moreover, it has the byproduct of exhibiting many intrinsic differences between left invariant spray geometry and left invariant Finsler geometry, which were independently captured in \cite{HM2021} and \cite{Xu2021-1}. So we see that,
$$\mbox{
{\it the global  invariant frames on $TG$ provide a useful tool.}}$$
\subsection{Submersion in spray geometry}

Consider Question \ref{main-question} for a homogeneous spray manifold $(G/H$, $\mathbf{G})$ with nontrival $H$, we are faced with several obstacles.

Firstly, generally speaking, there are no global
invariant frames on $G/H$ or $T(G/H)$.
In \cite{Hu2015-1}, L. Huang has used a local invariant frame and the Nomizu connection \cite{KN1969,No1954}. So
his calculation can only provide precise information at the origin $o=eH\in G/H$. Though his method is enough for calculating the homogeneous curvatures,  it is not sufficient for the tasks of describing geodesics and parallel translations.

Secondly, the methods in \cite{Hu2015-1}, including the formulae for the spray vector field and the connection operator (see (\ref{301}) and (\ref{302}) in Remark \ref{remark-5-4}), and the usage of the Nomizu connection, heavily rely on a reductive decomposition. In homogeneous Finsler geometry, without loss of generality, we may assume the $G$-action on the homogeneous Finsler manifold $(G/H,F)$ is effect, then the existence of reductive decompositions is a well known fact (see Lemma 2.2 in \cite{XN2021}, which is  valid in homogeneous Finsler geometry). However, we do not know a similar result in homogeneous spray geometry.

To overcome these difficulties, we let the submersion technique
step in. Submersion technique in Riemannian geometry has a relatively long history \cite{Kl1982,Ne1966} and plays an important role in the study of sectional curvature \cite{Ch1973,GM1974}. Finslerian  Submersion was introduced and studied by J.C.\'{A}. Paiva and C.E. Dur\'{a}n \cite{PD2001} and applied to homogeneous Finsler geometry for defining the normal homogeneity \cite{XD2017-2} and proving a homogeneous flag curvature formula \cite{XDHH}.

In this paper, we define the submersion in spray geometry as a pair $(\pi,\mathcal{L})$, in which $\pi:(\overline{M},\overline{\mathbf{G}})\rightarrow (M,\mathbf{G})$ is a smooth submersion between two spray manifolds, and $\mathcal{L}$, which is called the {\it horizonal bundle} for this submersion, is a distribution on $\overline{M}$ (i.e., a smooth linear sub-bundle of $T\overline{M}$), such that $\overline{\mathbf{G}}$ is tangent to $\mathcal{L}\backslash0\subset T\overline{M}\backslash0$, and geodesics $\overline{c}(t)$ on $(\overline{M},\overline{\mathbf{G}})$ which are tangent to $\mathcal{L}$ are related to geodesics on $(M,\mathbf{G})$ by $\pi$. See Section \ref{subsection-3-1} for the precise definition.

Though the requirements for a submersion between two spray structures are relatively weak, it is enough for us to lift the geometric notions on $(M,\mathbf{G})$, for example, geodesics, parallel translations, Riemann curvature, etc,
to $\mathcal{L}$ and describe them on $\overline{M}$. See Lemma \ref{lemma-3-5}, Lemma \ref{lemma-3-6} and Lemma \ref{lemma-3-7} for the precise statements. These lemmas provide the submersion technique in spray geometry. To apply this technique to the study of a homogeneous spray manifold
$(G/H,\mathbf{G})$, we only need to find the suitable submersion $(\pi,\mathcal{L})$ between a left invariant spray structure $\overline{\mathbf{G}}$ on $G$ and the homogeneous spray structure $\mathbf{G}$, so that we can lift everything from $(G/H,\mathbf{G})$ to $\mathcal{L}$, and then use the global invariant frames on $TG$ to study them.

The submersion pair $(\pi,\mathcal{L})$ generalizes Riemannian submersion, but not all Finslerian submersion. The reason is that, for a Finslerian submersion
$\pi:(\overline{M},\overline{F})\rightarrow (M,F)$, the fiber $\mathcal{L}_{\overline{p}}$ with $\pi(\overline{p})=p$,
of the horizonal bundle $\mathcal{L}$, i.e.,
$$\mathcal{L}_{\overline{p}}=\{\overline{v}\ |\
\overline{F}(\overline{p},\overline{v})=F(p,
\pi_*(\overline{v}))\}\subset T_{\overline{p}}\overline{M},$$
may not be a linear subspace. We may alter the definition in Section \ref{subsection-3-1}, permitting the sub-bundle $\mathcal{L}$ to be nonlinear, so that all Finslerian submersions can be included. However, this generalization seems not practically useful because (2) in Lemma \ref{lemma-3-4} (as well as Lemma \ref{lemma-3-6} and Lemma \ref{lemma-3-7}) may lose their validity.

The submersion pair $(\pi,\mathcal{L})$ can also produce a
Ehresmann connection which is useful for studying foliations \cite{Eh1950,Zh2016}. In later discussion, we avoid this terminology and concentrate on spray geometry.
\subsection{Main results}
\label{subsection-1-4}

Firstly, we define and study the spray vector field $\eta$ and the connection operator $N$, and construct the wanted submersion for
a homogeneous spray manifold $(G/H,\mathbf{G})$ with a linear decomposition $\mathfrak{g}=\mathfrak{h}+\mathfrak{m}$.

The following observation is fundamental and useful. Any smooth curve $c(t)$ on $G/H$ with $c(0)=g\cdot o$ for some $g\in G$ and nowhere-vanishing $\dot{c}(t)$ can be uniquely lifted to a smooth curve $\overline{c}(t)$ on $G$, which is defined at least around $t=0$, satisfying $\overline{c}(0)=g$, $c(t)=\overline{c}(t)\cdot o$ and $(L_{\overline{c}(t)^{-1}})_*(\dot{\overline{c}}(t))=
(\overline{c}(t)^{-1})_*(\dot{c}(t))
\in\mathfrak{m}\backslash0$ everywhere, and this $\overline{c}(t)$ can exist more globally,
when the decomposition $\mathfrak{g}=\mathfrak{h}+\mathfrak{m}$ is reductive or $H$ is compact  (see Lemma \ref{lemma-5-1}).

Let $c(t)$ be a geodesic on $(G/H,\mathbf{G})$ with $c(0)=o$ and $\dot{c}(0)=y\in \mathfrak{m}\backslash\{0\}$,
$\overline{c}(t)$ the smooth curve on $G$ provided by Lemma \ref{lemma-5-1} for $c(t)$, satisfying $\overline{c}(0)=e$, $c(t)=\overline{c}(t)\cdot o$ and $y(t)=(L_{\overline{c}(t)^{-1}})_*(\dot{\overline{c}}(t))\in
\mathfrak{m}\backslash\{0\}$ everywhere.  Then
the {\it spray vector field}
is  $\eta(y)=-\tfrac{{\rm d}}{{\rm d}t}|_{t=0}y(t)$.
It is a
positively 2-homogeneous smooth map from $\mathfrak{m}\backslash\{0\}$ to $\mathfrak{m}$ (see Lemma \ref{subsection-5-2}), so we can further define the {\it connection operator} $N:\mathfrak{m}\backslash\{0\}\times\mathfrak{m}
\rightarrow\mathfrak{m}$ by $N(y,u)=\tfrac12D\eta(y,u)-\tfrac12[y,u]_\mathfrak{m}$,
in which $D\eta(y,u)=\tfrac{{\rm d}}{{\rm d}t}|_{t=0}\eta(y+tu)$.

The following correspondence between geodesics on $(G/H,\mathbf{G})$ and integral curves of $-\eta$ on $\mathfrak{m}\backslash\{0\}$ follows naturally (see Theorem \ref{theorem-5-3}).
\medskip
\\
\noindent
{\bf Theorem A}\quad
{\it Let $(G/H,\mathbf{G})$ be a homogeneous spray manifold with a linear decomposition $\mathfrak{g}=
\mathfrak{h}+\mathfrak{m}$ and $\eta$ the spray vector field. Then we have a one-to-one correspondence between the following two sets:
\begin{enumerate}
\item the set of all geodesics $c(t)$ on $(G/H,\mathbf{G})$, which is defined around $t=0$ and satisfies $c(0)=o$;
\item the set of all integral curves $y(t)$ of $-\eta$ on $\mathfrak{m}\backslash\{0\}$, which is defined around $t=0$.
\end{enumerate}
The correspondence is from $c(t)$ to $y(t)=
(L_{\overline{c}(t)^{-1}})_*(\dot{\overline{c}}(t))$, in which $\overline{c}(t)$ is the smooth curve on $G$ satisfying $\overline{c}(0)=e$, $c(t)=\overline{c}(t)\cdot o$ and $(L_{\overline{c}(t)^{-1}})_*(\dot{\overline{c}}(t))
\in\mathfrak{m}\backslash\{0\}$ for all possible values of $t$.
The range for the parameter $t$ in this correspondence can be changed to an arbitrary interval $(a,b)$ with $a<0<b$ when $\mathfrak{g}=\mathfrak{h}+\mathfrak{m}$ is reductive
or $H$ is compact.
}
\medskip

The submersion for a homogeneous spray structure can be constructed as following.
\medskip
\\
\noindent
{\bf Theorem B}\quad
{\it Let $(G/H,\mathbf{G})$ be a homogeneous spray manifold with a linear decomposition $\mathfrak{g}=\mathfrak{h}+\mathfrak{m}$ and $\eta:\mathfrak{m}\backslash\{0\}\rightarrow\mathfrak{m}$ the spray vector field.
Let $\overline{\mathbf{G}}$ be a left invariant spray structure on $G$ such that its spray vector field $\overline{\eta}:\mathfrak{g}\backslash\{0\}
\rightarrow\mathfrak{g}$ satisfies $\eta=\overline{\eta}|_{\mathfrak{m}\backslash\{0\}}$. Denote $\pi: G\rightarrow G/H$ the smooth map $\pi(g)=g\cdot o$ for all $ g\in G$, and $\mathcal{L}=\cup_{g\in G}(L_g)_*(\mathfrak{m})\subset TG$ a distribution on $G$.
Then
$(\pi,\mathcal{L})$ is a submersion from $(G,\overline{\mathbf{G}})$ to $(G/H,\mathbf{G})$.}
\medskip

Generally speaking the relation between a homogeneous spray structure and its spray vector field might be very complicated.
However, when we have a reductive decomposition,
it can be easily understood.

Using the submersion technique, we prove a one-to-one correspondence between the following two sets, with respect
to a given reductive decomposition $\mathfrak{g}=\mathfrak{h}+\mathfrak{m}$ for $G/H$:
\begin{enumerate}
\item the set of all $G$-invariant spray structures $\mathbf{G}$ on $G/H$;
\item the set of all $\mathrm{Ad}(H)$-invariant smooth maps
$\eta:\mathfrak{m}\backslash\{0\}\rightarrow\mathfrak{m}$.
\end{enumerate}
The correspondence is from $\mathbf{G}$ to its spray vector field (see Theorem \ref{theorem-5-5}).

As a byproduct, we see that any homogeneous spray structure
$\mathbf{G}$ on $G/H$ with a reductive decomposition $\mathfrak{g}=\mathfrak{h}+\mathfrak{m}$ can be uniquely presented as $\mathbf{G}=\mathbf{G}_0-\mathbf{H}$,
in which
$\mathbf{G}_0$ is the spray structure for the Nomizu connection, with the spray vector field $\eta=0$,
and $\mathbf{H}$ is the $G$-invariant smooth tangent vector field
on $T(G/H)\backslash\{0\}$ which is tangent to each $T_x(G/H)$ and $\mathbf{H}|_{T_o(G/H)}$ provides the spray vector field.

Further more, we assume that $\mathbf{G}$ is the geodesic spray of a homogeneous Finsler metric on $G/H$. Using the submersion technique again, we prove the spray vector field and the connection operator of $\mathbf{G}$, with respect to the given
reductive decomposition, coincide with those in \cite{Hu2015-1} defined for $F$ respectively (see Theorem \ref{theorem-5-7}).

Nextly, we study the parallel translations on $(G/H,\mathbf{G})$.

Because of Theorem B, the submersion technique can be applied. Practically,
We fix the $\overline{\mathbf{G}}$ in Theorem B, such that
its spray vector field $\overline{\eta}$ satisfies $\overline{\eta}(y)=\eta(y_\mathfrak{m})$ in a conic open neighborhood of $\mathfrak{m}\backslash\{0\}$. This technical assumption helps simplify some calculation.

We use Lemma \ref{lemma-3-5} and Lemma \ref{lemma-3-6} to lift parallel translations  to $\mathcal{L}$, and use the global invariant frame on $TG$ to find the ODEs on $\mathfrak{m}$. For the linearly and nonlinearly parallel translations, we have the following two theorems respectively (see Theorem \ref{theorem-6-1} and Theorem \ref{theorem-6-2}).
\medskip

\noindent
{\bf Theorem C}\quad
{\it
Let $(G/H,\mathbf{G})$ be a homogeneous spray manifold with a linear decomposition $\mathfrak{g}=\mathfrak{h}+\mathfrak{m}$,
$c(t)$ any smooth curve on $G/H$ with $c(0)=o$ and nowhere-vanishing $\dot{c}(t)$, and $\overline{c}(t)$ the smooth curve on $G$ satisfying $\overline{c}(0)=e$, $c(t)=\overline{c}(t)\cdot o$ and $y(t)=(L_{\overline{c}(t)^{-1}})_*
(\overline{c}(t))\in\mathfrak{m}\backslash\{0\}$ for each $t$.
Then for any smooth vector field $W(t)=(\overline{c}(t))_*(w(t))$ along $c(t)$,
we have
\begin{equation*}\label{069}
D_{\dot{c}(t)}W(t)=(\overline{c}(t))_*
(\tfrac{{\rm d}}{{\rm d}t}w(t)+N(y(t),w(t))+[y(t),w(t)]_\mathfrak{m}).
\end{equation*}
In particular, $W(t)$ is linearly parallel along $c(t)$ if and only if
$w(t)$ satisfies
$$\tfrac{{\rm d}}{{\rm d}t}w(t)+N(y(t),w(t))+[y(t),w(t)]_\mathfrak{m}=0$$
everywhere.}\medskip

\noindent
{\bf Theorem D}\quad
{\it
Let $(G/H,\mathbf{G})$ be a homogeneous spray manifold with a linear decomposition $\mathfrak{g}=\mathfrak{h}+\mathfrak{m}$,
$c(t)$
any smooth curve on $G/H$ with $c(0)=o$ and nowhere-vanishing $\dot{c}(t)$, and $\overline{c}(t)$ the smooth curve on $G$ satisfying $\overline{c}(0)=e$, $c(t)=\overline{c}(t)\cdot o$ and $w(t)=(L_{\overline{c}(t)^{-1}})_*
(\overline{c}(t))\in\mathfrak{m}\backslash\{0\}$ for each $t$.
Suppose $Y(t)=(\overline{c}(t))_*(y(t))$ is a nowhere-vanishing smooth vector field along $c(t)$. Then $Y(t)$ is nonlinearly parallel
along $c(t)$ if and only if $y(t)$ satisfies
\begin{equation*}
\tfrac{{\rm d}}{{\rm d}t}y(t)+N(y(t),w(t))=0
\end{equation*}
everywhere.}\medskip

Theorem C is applied later to calculating or interpreting homogeneous curvature formulae. Theorem D is useful for studying the restricted holonomy group \cite{HMM2021,Ko1957} of a homogeneous spray manifold and the Landsberg problem for a homogeneous Finsler manifold \cite{XD2021}.


Finally, we show that homogeneous curvature formulae
can be derived using the techniques and results previously mentioned, including the submersion technique, the global invariant frame, descriptions for the parallel translations, etc.
Here we take the S-curvature, Landsberg curvature and Riemann curvature for example. See \cite{HM2019,HS2019,HYS2016,Ma1996,Sh1997,XD2018,XDHH}
for some usage of these curvatures in general and homogeneous Finsler geometry.

The homogeneous S-curvature and Landsberg curvature formulae are immediate corollary of Theorem C. For the S-curvature of
a homogeneous spray manifold, we have (see Theorem \ref{theorem-6-3}).
\medskip

\noindent
{\bf Theorem E}\quad
{\it Let $(G/H,\mathbf{G})$ be a homogeneous spray structure with a linear decomposition $\mathfrak{g}=\mathfrak{h}+\mathfrak{m}$. Suppose that the $\mathrm{Ad}(H)$-action on $\mathfrak{g}/\mathfrak{h}$ is unimodular. Then for $\mathbf{G}$ and any $G$-invariant smooth measure $d\mu$ on $G/H$, the S-curvature satisfies
\begin{equation*}
\mathbf{S}(o,y)=\mathrm{Tr}_\mathbb{R}(N(y,\cdot)+
\mathrm{ad}_\mathfrak{m}(y)),
\end{equation*}
for any $y\in\mathfrak{m}\backslash\{0\}=
T_o(G/H)\backslash\{0\}$. Here $\mathrm{ad}_\mathfrak{m}(y):\mathfrak{m}\rightarrow\mathfrak{m}$ is the linear map $w\mapsto[y,w]_\mathfrak{m}$.
}\medskip

For the Landsberg curvature of a homogeneous Finsler manifold, we re-prove L. Huang's formula in \cite{Hu2015-1}
(see Theorem \ref{theorem-6-4}), with the slight refinement that the reductive decomposition is no longer needed.

The calculation for the Riemann curvature of a homogeneous spray manifold $(G/H,\mathbf{G})$ is much harder.
However, after lifted to $\mathcal{L}\subset TG$ where the global invariant frames are available, by Lemma \ref{lemma-3-7}, the calculation is almost the same as that in \cite{Xu2021-1}, proving Theorem C for the Riemann curvature
of a left invariant spray structure.

We not only generalize L. Huang's Riemann curvature formula \cite{Hu2015-1,Hu2015-2,Hu2017} to homogeneous spray geometry, but also provide an interesting new interpretation for
it.
%
%
Summarizing
Theorem \ref{theorem-6-6} and Corollary \ref{corollary-6-7},
we have
\medskip

\noindent
{\bf Theorem F}\quad
{\it
Let $(G/H,\mathbf{G})$ be a homogeneous spray manifold with a linear decomposition $\mathfrak{g}=\mathfrak{h}+\mathfrak{m}$. Then for any $y\in\mathfrak{m}\backslash\{0\}=T_o(G/H)$, the Riemann curvature operator $\mathbf{R}_y:\mathfrak{m}\rightarrow\mathfrak{m}$ satisfies
\begin{eqnarray*}
\mathbf{R}_y(w)=[y,[w,y]_\mathfrak{h}]_\mathfrak{m}
+DN(\eta(y),y,w)-N(y,N(y,w))+N(y,[y,w]_\mathfrak{m})
-[y,N(y,w)]_\mathfrak{m},
\end{eqnarray*}
in which $DN(\eta(y),y,w)=\tfrac{{\rm d}}{{\rm d}s}|_{s=0}N(y+s\eta(y),w)$.

Let $c(t)$ be a geodesic on $(G/H,\mathbf{G})$,
which is defined around $t=0$ and satisfies $c(0)=o$,
$\overline{c}(t)$ the smooth curve on $G$ provided by Lemma \ref{lemma-5-1} for $c(t)$, satisfying $\overline{c}(0)=e$, $c(t)=\overline{c}(t)\cdot o$ and
$y(t)=(L_{c(t)^{-1}})_*(\dot{c}(t))\in\mathfrak{m}\backslash\{0\}$
everywhere, and
$W(t)=(\overline{c}(t))_*(w(t))$
a linearly parallel vector field along $c(t)$, where $w(t)=w^i(t)e_i$ is viewed as a smooth vector field along
the smooth curve $y(t)$ in $\mathfrak{m}\backslash\{0\}$. Then $N(t)=N(y(t),w(t))$
 and $R(t)=\mathbf{R}_{y(t)}(w(t))$ satisfy
 \begin{eqnarray*}
 N(t)=[w(t),\eta]\quad\mbox{and}\quad R(t)=[y(t),[w(t),y(t)]_\mathfrak{h}]_\mathfrak{m}+
 [\eta,N(t)]
 \end{eqnarray*}
 when they are viewed as smooth vector fields along $y(t)$.
}\medskip

Notice that the second statement in Theorem F was firstly proved in \cite{Xu2021-2} for a left invariant spray structure,
where a slight generalization for
the bracket between smooth vector fields is needed (see  Remark 1.3 in \cite{Xu2021-2} or Remark \ref{remark-6-10}).
%

\subsection{Conclusion remarks}

By the theorems listed in Section \ref{subsection-1-4},
we have setup a general theoretical framework for homogeneous spray geometry, and answered Question \ref{main-question} simultaneously. In one hand, this project can be continued by
generalizing more notions,  Clifford-Wolf homogeneity \cite{BN2009,XD2014}, weak symmetry \cite{DH2010,Se1956,WC2021}, homogeneous geodesic and geodesic orbit property \cite{CWZ2021,KV1991,YD2014}, etc, to homogeneous spray geometry. In the other hand, it can be applied to some special cases.

For example, those theorems in Section \ref{subsection-1-4} can be applied to homogeneous Finsler geometry. Some known results or formulae can be re-proved without using a reductive decomposition, and more interesting new results are expected. Besides the Lie algebraic information from $G/H$ and  $\mathfrak{g}=\mathfrak{h}+\mathfrak{m}$, and the dynamic system induced by the spray vector field $\eta$ and the connection operator $N$ on $\mathfrak{m}\backslash\{0\}$, we have one more ingredient, the Hessian geometry \cite{Sh2007,Xu2021-0} for the $\mathrm{Ad}(H)$-invariant Minkowski norm $F$ on $\mathfrak{m}$ which determines a homogeneous Finsler metric. The interaction among these three
will make a good story in homogeneous Finsler geometry.

Another application is homogeneous pseudo-Finsler geometry, because the geodesic spray of a pseudo-Finsler metric can be defined by the same coefficient formula (\ref{300}). Usually, a pseudo-Finsler metric $F$ on $M$ is {\it conic}, i.e., it is only defined on a conic open subset $\mathcal{A}$ in $TM\backslash0$, i.e., each $\mathcal{A}_x=\mathcal{A}\cap T_xM$ is conic in $T_xM\backslash\{0\}$ \cite{JS2017,JV2018}. Based on this observation, we can define a {\it conic spray structure} $\mathbf{G}$ on $M$, which is then a smooth tangent vector field on the conic open subset $\mathcal{A}\subset TM\backslash0$ satisfying similar requirements as in
Section \ref{subsection-2-1}. The $G$-invariancy for a conic spray structure can also be defined. All discussions and results in this paper, after some minor changes, are valid for conic spray structures, so they can be applied to homogeneous pseudo-Finsler geometry.

This paper is organized as following. In Section 2, we summarize some basic knowledge in spray geometry and Finsler geometry. In Section 3, we define the submersion between two spray structures and prove some lemmas for the submersion technique. In Section 4, we introduce the tool of global invariant frames on a Lie group and its tangent bundle, and recall some known results for a left invariant spray structure on a Lie group. In Section 5, we generalize the spray vector field and the connection operator to homogeneous spray geometry, and construct the wanted submersion for a homogeneous spray manifold. Moreover, we discuss the correspondence between a homogeneous spray structure $\mathbf{G}$ and its spray vector field $\eta$, and that between geodesics on $(G/H,\mathbf{G})$ and integral curves of $-\eta$. In Section 6, we use the submersion technique to describe the parallel translations and calculate some curvature formulae in homogeneous spray and Finsler geometries.

\section{Preliminaries in spray geometry}
%
%
%
%

In this section, we summarize some basic knowledge on spray and Finsler geometries. See \cite{BCS2000,Sh2001-1,Sh2001-2}
for more details.

\subsection{Spray structure and geodesic}
\label{subsection-2-1}

Let $M^n$ be an $n$-dimensional smooth manifold. Then a {\it spray structure} on $M$ is a smooth tangent vector field $\mathbf{G}$ on the slit tangent bundle $TM\backslash0$, which can be locally
presented as
\begin{equation}\label{001}
\mathbf{G}=y^i\partial_{x^i}-2\mathbf{G}^i\partial_{y^i}
\end{equation}
for any {\it standard local coordinate} $(x^i,y^j)$ (i.e., $x=(x^i)\in M$ and
$y=y^j\partial_{x^j}\in T_xM$), where
$\mathbf{G}^i=\mathbf{G}^i(x,y)$ is positively 2-homogeneous
for its $y$-entry (i.e., $\mathbf{G}^i(x,\lambda y)=\lambda^2
\mathbf{G}^i(x,y)$, $\forall \lambda>0$). A spray structure $\mathbf{G}$ is called {\it affine}, if every coefficient $\mathbf{G}^i$ is quadratic for its $y$-entry.

In one hand, geodesics can be defined for a spray structure.
A smooth curve $c(t)$ on $(M,\mathbf{G})$ with nowhere-vanishing $\dot{c}(t)$ is called a {\it geodesic}
if its lifting $(c(t),\dot{c}(t))$ in $TM\backslash0$ is an integral curve of $\mathbf{G}$. Locally, a geodesic $c(t)=(c^i(t))$
is characterized by the following ODE for any standard local coordinate $(x^i,y^j)$,
$$\ddot{c}(t)+2\mathbf{G}^i(c(t),\dot{c}(t))=0,\quad\forall i.$$

In the other hand, the set of all geodesics on $(M,\mathbf{G})$
completely determines $\mathbf{G}$.
%

\subsection{Covariant derivative and parallel translation}
\label{subsection-2-2}

Following (\ref{001}) for the spray structure $\mathbf{G}$ on $M$,
we denote

\begin{equation}
\mathbf{N}^i_j=\tfrac{\partial}{\partial y^j}\mathbf{G}^i\quad\mbox{and}\quad
\delta_{x^i}=\partial_{x^i}-\mathbf{N}_i^j\partial_{y^j},
\end{equation}
which are functions and tangent vector fields locally defined on $M$ and $TM\backslash0$ respectively.

The {\it linearly covariant derivative} along a smooth curve $c(t)$ on $M$ with nowhere-vanishing $\dot{c}(t)$ can be locally presented as
$$
D_{\dot{c}(t)}X(t)=(\tfrac{{\rm d}}{{\rm d}t}X^i(t)+
\mathbf{N}^i_j(c(t),\dot{c}(t))X^j(t))\partial_{x^i}|_{c(t)},$$
in which $X(t)=X^i(t)\partial_{x^i}|_{c(t)}$ is a smooth vector field along $c(t)$. A smooth vector field $X(t)$ is called {\it linearly parallel} along $c(t)$ if $D_{\dot{c}(t)}X(t)=0$. Existence and uniqueness theory for the solution of ODE provides the {\it linearly parallel translation} along the curve $c(t)$. For example, $\mathrm{P}^{\mathrm{l}}_{c;a,b}:T_pM\rightarrow T_qM$ from $p=c(a)$ to $q=c(b)$ is the linear isomorphism defined by $\mathrm{P}^{\mathrm{l}}_{c}(v)=X(b)$, in which $X(t)$ is the unique linearly parallel vector field along $c(t)$ satisfying $X(a)=v$.

There is also a {\it nonlinearly covariant derivative} along $c(t)$, i.e.,
$$\widetilde{D}_{\dot{c}(t)}X(t)=(\tfrac{{\rm d}}{{\rm d}t}X^i(t)+
\dot{c}^j(t)\mathbf{N}^i_j(c(t),X(t)))\partial_{x^i}|_{c(t)},$$
for any smooth vector field
$X(t)=X^i(t)\partial_{x^i}|_{c(t)}$ along $c(t)$ which is  nowhere-vanishing. Using $\widetilde{D}$ instead of $D$, {\it nonlinearly parallel vector field} and {\it nonlinearly parallel translation} can be similarly defined.

Another equivalent description for the nonlinearly parallel translation is the following. Suppose $c(t)=(c^i(t))$ is a smooth curve in a standard local coordinate. Denote $S=\cup_t (T_{c(t)}M\backslash\{0\})$ a smooth submanifold in $TM\backslash0$. The tangent vector field $\dot{c}(t)$ of $c(t)$ can be lifted to
the smooth vector field
\begin{equation}
\label{033}
\widetilde{\dot{c}(t)}^\mathcal{H}=
\dot{c}^i(t)\delta_{x^i}|_{T_{c(t)}M\backslash\{0\}}
\end{equation}
on $S$.
Then nonlinearly parallel vector fields along $c(t)$ are  integral curves of $\widetilde{\dot{c}(t)}^\mathcal{H}$. A nonlinearly parallel translation is denoted as $\mathrm{P}^{\mathrm{nl}}_{c;a,b}: T_pM\backslash\{0\}\rightarrow
T_qM\backslash\{0\}$ where $c(t)$ is a smooth  curve on $M$ with $c(a)=p$ and $c(b)=q$.

\subsection{Riemann curvature and S-curvature}
\label{subsection-2-3}
The {\it Riemann curvature} for a spray structure $\mathbf{G}$ is the linear map $\mathbf{R}_y:T_xM\rightarrow T_xM$ for any $y\in T_xM\backslash\{0\}$, which naturally appears in the Jacobi equation for a smooth variation of geodesics. In any standard local coordinate $(x^i,y^j)$, it can be presented as $\mathbf{R}_y=\mathbf{R}^i_k \partial_{x^i}\otimes dx^k$, with the Riemann curvature coefficients
$$\mathbf{R}^i_k=2\tfrac{\partial \mathbf{G}^i}{\partial x^k}
-y^j\tfrac{\partial^2 \mathbf{G}^i}{\partial x^j\partial y^k}
+2\mathbf{G}^j\tfrac{\partial^2\mathbf{G}^i}{\partial y^j\partial y^k}-\tfrac{\partial\mathbf{G}^i}{\partial y^j}\tfrac{\partial \mathbf{G}^j}{\partial y^k}.$$

Riemann curvature can be alternatively interpreted as following. We first decompose the tangent bundle $T(TM\backslash0)$ as the direct sum of two linear sub-bundles, the {\it horizonal distribution} $\mathcal{H}$ and the {\it vertical distribution}
$\mathcal{V}$, which are linearly spanned by all $\delta_{x^i}$ and by all $\partial_{y^j}$ respectively. For any fixed $x\in M$ and $y\in T_xM\backslash\{0\}$, a tangent vector
$v=a^i\partial_{x^i}\in T_{x}M$ can be one-to-one corresponded to its horizonal lifting $\widetilde{v}^\mathcal{H}=a^i\delta_{x^i}\in \mathcal{H}_{(x,y)}$, and its vertical lifting $\widetilde{v}^\mathcal{V}=a^i\partial_{y^i}\in \mathcal{V}_{(x,y)}$ respectively. The notions $\mathcal{H}$, $\mathcal{V}$, $\widetilde{\cdot}^\mathcal{H}$ and $\widetilde{\cdot}^\mathcal{V}$ are all irrelevant to the choice of standard local coordinate.

Direct calculation shows that
$$[\mathbf{G},\delta_{x^k}]\equiv \mathbf{R}_k^i\partial_{y^i}\quad \mbox{(mod }\mathcal{H}\mbox{)},$$
and more generally,
\begin{equation}\label{048}
[\mathbf{G},a^k(x)\delta_{x^k}]\equiv a^k(x)\mathbf{R}_k^i\partial_{y^i}\quad \mbox{(mod }\mathcal{H}\mbox{)}.
\end{equation}
Using the horizonal and vertical liftings, (\ref{048}) can be interpreted as the following lemma.

\begin{lemma}\label{lemma-2-1}
For any smooth vector field $X$ on $M$, we have
$$[\mathbf{G},\widetilde{X}^{\mathcal{H}}]|_{(x,y)}
=\widetilde{\mathbf{R}_y(X(x))}^{\mathcal{V}}
\quad\mbox{(mod }\mathcal{H}\mbox{)},$$
at each point $(x,y)\in T_xM\backslash\{0\}$.
\end{lemma}

To define the S-curvature, we need to specify a smooth measure ${\rm d}\mu$ on $M$. Suppose ${\rm d}\mu$ can be presented as
${\rm d}\mu=\sigma(x){\rm d}x^1\cdots {\rm d}x^n$
in a standard local coordinate, in which $\sigma(x)$ is a nowhere-vanishing smooth function $\sigma(x)$.
Then the
S-curvature for $\mathbf{G}$ and ${\rm d}\mu$ is a smooth function $\mathbf{S}:TM\backslash0\rightarrow\mathbb{R}$, satisfying
\begin{equation}\label{020}
\mathbf{S}(x,y)=\mathbf{N}^i_i(x,y)-\tfrac{y^i}{\sigma(x)}
\tfrac{\partial}{\partial x^i}\sigma(x),
\end{equation}
for any standard local coordinate. The equality (\ref{020})
can be interpreted as following.

\begin{lemma}\label{lemma-2-2}
Let $c(t)$ be a geodesic on $(M,\mathbf{G})$ with $c(0)=x$ and $\dot{c}(0)=y\in T_xM\backslash\{0\}$, $\{E_1(t),\cdots,E_n(t)\}$ a smooth frame along $c(t)$ such that ${\rm d}\mu(E_1(t),\cdots,E_n(t))$ is a nonzero constant function. Denote $A$ the linear map from any $w=(v^1,\cdots,v^n)\in\mathbb{R}^n$ to
$-\tfrac{{\rm d}}{{\rm d}t}|_{t=0}v(t)$, where $v(t)=(v^1(t),\cdot,v^n(t))$ is determined by the linearly parallel vector field $V(t)=v^i(t)E_i(t)$ along $c(t)$ with $V(0)=v^iE_i(0)$. Then for $\mathbf{G}$ and ${\rm d}\mu$, the value of the S-curvature at $(x,y)$  coincides with $\mathrm{Tr}_\mathbb{R}A$.
\end{lemma}
\begin{proof}
Choose any standard local coordinate $(x^i,y^j)$ around $x$. Denote $E_i(t)=A_i^j(t)\partial_{x^j}$ and $\partial_{x^j}=B_j^i(t) E_i(t)$ (so we have $(B_i^j(t))=(A_i^j(t))^{-1}$, i.e., $A_i^j(t)B^k_j(t)=\delta^k_i$). Since we have assumed that ${\rm d}\mu(E_1(t),\cdots,E_n(t))=\sigma(c(t))\det(A_i^j(t))$
is a nonzero constant function, the term $\tfrac{y^i}{\sigma(x)}\tfrac{\partial}{\partial x^i}\sigma(x)$
in (\ref{020}) equals
\begin{eqnarray*}
\tfrac{y^i}{\sigma(x)}\tfrac{\partial}{\partial x^i}\sigma(x)
=\det(A^i_j(0))\tfrac{{\rm d}}{{\rm d}t}|_{t=0}\det(B^i_j(t))=
A_i^j(0)\tfrac{{\rm d}}{{\rm d}t}B_j^i(0).
\end{eqnarray*}
For any linearly parallel vector field $W(t)=w^i(t)\partial_{x^i}|_{c(t)}=v^i(t)E_i(t)$ along $c(t)$,
we have $v^i(t)=B^i_j(t)w^j(t)$, so
\begin{eqnarray*}
\tfrac{{\rm d}}{{\rm d}t}v^i(t)&=&w^j(t)\tfrac{{\rm d}}{{\rm d}t}B^i_j(t) +B^i_j(t)\tfrac{{\rm d}}{{\rm d}t}w^j(t)\\
&=&w^j(t)\tfrac{{\rm d}}{{\rm d}t}B^i_j(t)-
B^i_j(t)\mathbf{N}^j_k(c(t),\dot{c}(t)) w^k(t)\\&=&(A^j_l(t)\tfrac{{\rm d}}{{\rm d}t}B^i_j(t)-
B^i_j(t)\mathbf{N}^j_k(c(t),\dot{c}(t)) A^k_l(t))v^l(t).
\end{eqnarray*}
So we have
\begin{eqnarray*}
\mathrm{Tr}_\mathbb{R}A&=&-A^j_i(0)\tfrac{{\rm d}}{{\rm d}t}B^i_j(0)+
B^i_j(0)\mathbf{N}^j_k(c(0),\dot{c}(0))A^k_i(0)\\
&=&\mathbf{N}^i_i(x,y)-\tfrac{y^i}{\sigma(x)}\tfrac{\partial}{\partial x^i}\sigma(x)=S(x,y),
\end{eqnarray*}
which ends the proof.
\end{proof}
\subsection{Finsler metric and its curvatures}
\label{subsection-2-4}
A {\it Finsler metric} on a smooth manifold $M^n$ is a continuous Function $F:TM\rightarrow[0,\infty)$ satisfying the following properties:
\begin{enumerate}
\item positiveness and smoothness, i.e., the restriction of $F$ to the slit tangent bundle $TM\backslash0$ is a positive smooth function;
\item positive 1-homogeneity, i.e., $F(x,\lambda y)=\lambda F(x,y)$ for any $x\in M$, $y\in T_xM$ and $\lambda\geq0$;
\item strong convexity, i.e., with respect to any standard local coordinate, the Hessian matrices $(g_{ij}(x,y))=([\tfrac12F^2]_{y^iy^j})$ is positive definite
    for any $x\in M$ and $y\in T_xM\backslash\{0\}$.
\end{enumerate}

The restriction of a Finsler metric $F$ to each tangent space is a {\it Minkowski norm}. Minkowski norm can be abstractly defined
on any finite dimensional real vector space, using similar requirements as (1)-(3) above.
The Hessian matrix $(g_{ij})$ in (3) is called the fundamental tensor for a Finsler metric or a Minkowski norm. We use it and its inverse matrix $(g^{ij})$ to move indices up and down. The fundamental tensor $(g_{ij})=(g_{ij}(x,y))$ induces inner products $g_y(\cdot,\cdot)$ on $T_xM$ parametrized by $y\in T_xM\backslash\{0\}$, i.e.,
\begin{equation}
g_y(u,v)=\tfrac12\tfrac{\partial^2}{\partial s\partial t}|_{s=t=0}F^2(y+su+tv)=u^iv^j g_{ij}(x,y),\label{060}
\end{equation}
for any $u=u^i\partial_{x^i}$ and $v=v^j\partial_{x^j}$ in $T_xM$.

The Finsler metric $F$ induces a spray structure, called the {\it geodesic spray} of $F$, with the spray coefficients $\mathbf{G}^i=\tfrac14g^{il}([F^2]_{x^ky^l}y^k-[F^2]_{x^l})$
in any standard local coordinate.
So all geometric notions in spray geometry, parallel translations, S-curvature, Riemann curvature, etc, are naturally inherited by Finsler geometry.

The Finsler metric $F$ can provide more curvature notions than its geodesic spray $\mathbf{G}$. For example, for any $y\in T_xM\backslash\{0\}$,
the Cartan tensor $\mathbf{C}_y:T_xM\times T_xM\times T_xM\rightarrow\mathbb{R}$ for $F$ is
$$\mathbf{C}_y(u,v,w)=\mathbf{C}_{ijk}u^i v^j w^k,\quad  \forall u=u^i\partial_{x^i}, v=v^j\partial_{x^j}, w=w^k\partial_{x^k}\in T_xM,$$
in which $\mathbf{C}_{ijk}=[\tfrac14F^2(x,y)]_{y^iy^jy^k}$, and
the Landsberg curvature
$\mathbf{L}_y:T_xM\times T_xM\times T_xM\rightarrow\mathbb{R}$ can be determined by
\begin{equation}\label{071}
\mathbf{L}_{\dot{c}(t)}(U(t),V(t),W(t))=
\tfrac{{\rm d}}{{\rm d}t}\mathbf{C}_{\dot{c}(t)}(U(t),V(t),W(t)),
\end{equation}
in which $U(t)$, $V(t)$ and $W(t)$ are linearly parallel vector fields along the geodesic $c(t)$.

See \cite{Sh2001-2} for more Riemannian and non-Riemannian curvatures in Finsler geometry.
\section{Submersion and submersion technique in spray geometry}

In this section, we define a submersion between two spray structures and prove some lemmas for the submersion technique in spray geometry.

\subsection{Submersion between two spray structure}
\label{subsection-3-1}

Let $\pi:\overline{M}^m\rightarrow M^n$ be a smooth submersion between two smooth manifolds, with $\dim\overline{M}=m\geq\dim M=n$. It can be locally presented as $\pi(x,z)=x$ by choosing suitable local coordinates \begin{equation}\label{030}
x=(x^1,\cdots,x^n)\in M\quad \mbox{and}\quad
(x,z)=(x^1,\cdots,x^n,z^{n+1},\cdots,z^m)\in\overline{M}.
\end{equation}

Suppose $\mathcal{L}=\cup_{(x,z)\in\overline{M}}\mathcal{L}_{(x,z)}$ is a distribution on $\overline{M}$, such that each fiber
$\mathcal{L}_{(x,z)}$ is a linear complement of
$\mathrm{ker}(\pi_*|_{(x,z)})$ in $ T_{(x,z)}\overline{M}$. Then $\mathcal{L}$ and $\mathcal{L}\backslash0=\cup_{(x,z)\in \overline{M}}(\mathcal{L}_{(x,z)}\backslash\{0\})$ are closed submanifolds of $T\overline{M}$ and $T\overline{M}\backslash0$ respectively.

Consider two spray structures, $\overline{\mathbf{G}}$ on $\overline{M}$ and $\mathbf{G}$ on $M$ respectively. We
say $(\pi,\mathcal{L})$ is a {\it submersion between the spray structures} $\overline{\mathbf{G}}$ and $\mathbf{G}$, if the following two conditions are satisfied:
\begin{enumerate}
\item $\overline{\mathbf{G}}|_{\mathcal{L}\backslash0}$  is  tangent to $\mathcal{L}\backslash0$ everywhere;
\item for any geodesic $\overline{c}(t)$ on $(\overline{M},\overline{\mathbf{G}})$, which lifting $(\overline{c}(t),\dot{\overline{c}}(t))$ is contained in $\mathcal{L}\backslash0$, $c(t)=\pi(\overline{c}(t))$ is
    a geodesic on $(M,\mathbf{G})$.
\end{enumerate}

\subsection{Standard local coordinate representations}
\label{subsection-3-2}

In this section, the following convention for indices is applied,
\begin{equation}\label{067}
1\leq i,j,k,l,p,q,r\leq n,\quad
n+1\leq \alpha,\beta,\gamma\leq m.
\end{equation}
For a submersion $(\pi,\mathcal{L})$ between $(\overline{M},\overline{\mathbf{G}})$ and $(M,\mathbf{G})$, we use bars to distinguish notations for $M$ (or $TM$) and those for $\overline{M}$ (or $T\overline{M}$ respectively). For example, the tangent vector fields in the frame for
a standard local coordinate are denoted as $\overline{\partial}_{x^i}$, $\overline{\partial}_{z^\alpha}$, etc, on $\overline{M}$ or $T\overline{M}$, and denoted as $\partial_{x^i}$ and $\partial_{y^i}$ on $M$ or $TM$.
For a spray structure $\overline{\mathbf{G}}$ on $\overline{M}$, we have
$\overline{\mathbf{G}}^i$ and $\overline{\mathbf{G}}^\alpha$ for its coefficients,
$\overline{\mathbf{N}}^i_j=\tfrac{\partial}{\partial y^j}\overline{\mathbf{G}}^i$, $\cdots$, $\overline{\mathbf{N}}^\alpha_\beta=\tfrac{\partial}{\partial w^\beta}\overline{\mathbf{G}}^\alpha$, $\overline{\mathcal{H}}$ and $\overline{\mathcal{V}}$ for its horizonal and vertical distributions respectively, $\widetilde{\overline{v}}^{\overline{\mathcal{H}}}$ and
$\widetilde{\overline{v}}^{\overline{\mathcal{V}}}$ for
the horizonal and vertical liftings of $\overline{v}\in T\overline{M}$ respectively,
etc.

We expand the local coordinates in (\ref{030}) to
standard local coordinates $(x,y)=(x^i,y^j)\in TM$ and
$(x,z,y,w)=(x^i,z^\alpha,y^j,w^\beta)\in T\overline{M}$ respectively, i.e.,
\begin{eqnarray*}
y=y^i\partial_{x^i}\in T_xM\ \ \mbox{or}\ \ y=y^i\overline{\partial}_{x^i}\in T_{(x,z)}\overline{M},\quad\mbox{and}\quad
w=w^\alpha\overline{\partial}_{z^\alpha}\in T_{(x,z)}\overline{M}.
\end{eqnarray*}

Using above standard local coordinates,
the submersion $(\pi,\mathcal{L})$ from $(\overline{M},\overline{\mathbf{G}})$ to $(M,\mathbf{G})$
can be presented as following.

Firstly, we have $\pi(x,z)=x$. So its tangent map $\pi_*:T\overline{M}\rightarrow TM$ can be presented as $\pi_*(x,z,y,w)=(x,y)$. The tangent map of $\pi_*$, i.e.,
$(\pi_*)_*:T(T\overline{M})\rightarrow T(TM)$, satisfies
$$(\pi_*)_*(\overline{\partial}_{x^i})=\partial_{x^i},\quad
(\pi_*)_*(\overline{\partial}_{y^j})=\partial_{y^j},\quad
\mbox{and}\quad
(\pi_*)_*(\overline{\partial}_{z^\alpha})=
(\pi_*)_*(\overline{\partial}_{w^\beta})=0.$$

Nextly,
the spray structures $\mathbf{G}$ and $\overline{\mathbf{G}}$ have the representations
\begin{eqnarray}
\mathbf{G}&=&y^i\partial_{x^i}-2\mathbf{G}^i\partial_{y^i}
\quad\mbox{and}\label{041}\\
\overline{\mathbf{G}}&=&
y^i\overline{\partial}_{x^i}+w^\alpha\overline{\partial}_{z^\alpha}
-2\overline{\mathbf{G}}^i\overline{\partial}_{y^i}
-2\overline{\mathbf{G}}^\alpha\overline{\partial}_{w^\alpha}\label{042}
\end{eqnarray}
respectively.

Finally, $\mathcal{L}$ can be presented as
the graph of $w=l(x,z,y)$, in which
$w^\beta=l^\beta(x,z,y)$
for each $\beta$ is a smooth function and it is homogeneously linear for its $y$-entry.

\begin{lemma}\label{lemma-3-1}
For the submersion $(\pi,\mathcal{L})$ between
$(\overline{M},\overline{\mathbf{G}})$ and $(M,\mathbf{G})$,
we have
\begin{enumerate}
\item $\overline{\mathbf{G}}^j(x,z,y,l(x,z,y))
=\mathbf{G}^j(x,y)$, $\forall j$, and
\item $\overline{\mathbf{N}}^j_i(x,z,y,l(x,z,y))+
\tfrac{\partial}{\partial y^i}l^\alpha(x,z)\overline{\mathbf{N}}^j_\alpha(x,z,y,l(x,z,y))
=\mathbf{N}^j_i(x,y)$, $\forall i,j$,
\end{enumerate}
at each $(x,z,y,l(x,z,y))\in \mathcal{L}\backslash0$ (i.e.,  $y\neq0$).
\end{lemma}
\begin{proof}
Let $\overline{c}(t)=(x(t),z(t))=(x^i(t),z^\alpha(t))$ with $t\in(-\epsilon,\epsilon)$
be a geodesic on
$(\overline{M},\overline{\mathbf{G}})$, such that its lifting
$(x(t),z(t),\dot{x}(t),\dot{z}(t))$ is contained in $\mathcal{L}$. That means
\begin{eqnarray}
& &\dot{z}^\alpha(t)=l^\alpha(x(t),z(t),\dot{x}(t)),\quad\forall\alpha,\label{002}
\\
& &\ddot{x}^j(t)+
2\overline{\mathbf{G}}^j(x(t),z(t),\dot{x}(t),\dot{z}(t))=0,
\quad\forall i,\label{003}\\
& &\ddot{z}^\alpha(t)+2
\overline{\mathbf{G}}^\alpha(x(t),z(t),\dot{x}(t),\dot{z}(t))=0,
\quad\forall\alpha.\label{004}
\end{eqnarray}

Since $c(t)=\pi(\overline{c}(t))=x(t)$ is a geodesic on $(M,\mathbf{G})$,
we have
\begin{equation}
\ddot{x}^j(t)+2\mathbf{G}^j(x(t),\dot{x}(t))=0.\label{005}
\end{equation}
Comparing (\ref{003}) and (\ref{005}), we get
\begin{equation}\label{006}
\overline{\mathbf{G}}^j(x(t),z(t),\dot{x}(t),l(x(t),z(t),\dot{x}(t)))
=\mathbf{G}^j(x(t),\dot{x}(t)).
\end{equation}

Since $x=x(0)$, $z=z(0)$ and $y=\dot{x}(0)$ can be chosen arbitrarily, (1)
in Lemma \ref{lemma-3-1} follows immediately after (\ref{006}).
Differentiate the equality in (1) of  Lemma \ref{lemma-3-1} with respect to $y^i$, then we get (2) in Lemma \ref{lemma-3-1}.
\end{proof}

\begin{remark}\label{remark-3-2}
Because
each $l^\alpha(x,z,y)$ depends linearly on its $y$-entry,  $\tfrac{\partial}{\partial y^i}l^\alpha$ only depends on $x$ and $z$. So we may denote $\tfrac{\partial}{\partial y^i}l^\alpha=
\tfrac{\partial}{\partial y^i}l^\alpha(x,z)$.
As a byproduct of the proof for Lemma \ref{lemma-3-1}, we can also get
\begin{eqnarray}
\overline{\mathbf{G}}^\alpha(x,z,y,l(x,z,y))&=&
-\tfrac12
y^i\tfrac{\partial}{\partial x^i}l^\alpha(x,z,y)
-\tfrac12l^\beta(x ,z ,y)\tfrac{\partial}{\partial z^\beta}l^\alpha(x ,z ,y))\nonumber\\
& &+\tfrac{\partial}{\partial y^i}
l^\alpha(x,z )\mathbf{G}^i(x ,y),\quad\forall \alpha,\label{007}
\end{eqnarray}
by differentiating (\ref{002}) with respect to $t$ and then Plugging
(\ref{002})-(\ref{005}) into it.
\end{remark}

\subsection{Some bundle maps induced by $\mathcal{L}$}
\label{subsection-3-3}

For the submersion $(\pi,\mathcal{L}):(\overline{M},\overline{\mathbf{G}})\rightarrow
(M,\mathbf{G})$, we have the following
smooth bundle maps, which are distinguished by their  superscripts
and subscripts, marking the domain and target bundles respectively:
\begin{enumerate}
\item $\Phi^{T\overline{M}}_{\mathcal{L}}:
    T\overline{M}\rightarrow\mathcal{L}$ with
    $(x,z,y,w)\mapsto (x,z,y,l(x,z,y))$,
\item $\Phi^{\mathcal{L}}_{TM}=
    \pi_*|_{\mathcal{L}}:\mathcal{L}\rightarrow TM$ with $(x,z,y,l(x,z,y))\mapsto (x,y)$, and
\item $\Phi^{T(T\overline{M})|_{\mathcal{L}}}_{T\mathcal{L}}:
T(T\overline{M})|_{\mathcal{L}}\rightarrow T\mathcal{L}$
is the restriction of the tangent map $(\Phi^{T\overline{M}}_\mathcal{L})_*: T(T\overline{M})\rightarrow T\mathcal{L}$ to $T(T\overline{M})|_{\mathcal{L}}$.
\end{enumerate}

Standard local coordinates are applied here to present these bundle maps, for the convenience in calculation. But we need to notice that
these bundle maps are in fact canonically defined from the submersion, i.e., local coordinate representations are inessential. This observation enable us to re-interpret these bundle maps in Section 6 by
global invariant frame.

For example, the tangent map of $\Phi^{T\overline{M}}_{\mathcal{L}}(x,z,y,w)=(x,z,y,l(x,z,y))$
provides
\begin{eqnarray}
&&\Phi^{T(T\overline{M})|_{\mathcal{L}}}_{T\mathcal{L}} (a^i\overline{\partial}_{x^i}+b^j\overline{\partial}_{y^j}
+c^\alpha
\overline{\partial}_{z^\alpha}+
d^\beta\overline{\partial}_{w^\beta})\nonumber\\
&=&a^i(\overline{\partial}_{x^i}+
\tfrac{\partial}{\partial x^i}l^\beta(x,z,y)\overline{\partial}_{w^\beta})+b^j
(\overline{\partial}_{y^j}+
\tfrac{\partial}{\partial y^j}l^\beta(x,z)\overline{\partial}_{w^\beta}+c^\alpha
(\overline{\partial}_{z^\alpha}+
\tfrac{\partial}{\partial z^\alpha}l^\beta(x,z,y)\overline{\partial}_{w^\beta})\nonumber\\
&=&
a^i\overline{\partial}_{x^i}+b^j\overline{\partial}_{y^j}
+c^\alpha\overline{\partial}_{z^\alpha}\quad
\mbox{(mod }\overline{\partial}_{w^\beta},\forall\beta\mbox{)}.\label{043}
\end{eqnarray}
at each $\overline{p}=(x,z,y,l(x,z,y))\in\mathcal{L}$. On the other hand, both $\Phi^{T\overline{M}}_{\mathcal{L}}$ and
$\Phi^{T(T\overline{M})}_{T\mathcal{L}}$ can be described by the following easy lemma, without using the local coordinates.
 \begin{lemma}\label{lemma-3-3}
 Denote $\psi:T\overline{M}\rightarrow\overline{M}$
the bundle map for $T\overline{M}$ over $\overline{M}$, i.e., $\psi(x,z,y,w)=(x,z)$. Then we have
\begin{enumerate}
\item at each $(x,z)\in\overline{M}$, $\Phi^{T\overline{M}}_{\mathcal{L}}|_{(x,z)}:
    T_{(x,z)}\overline{M}\rightarrow\mathcal{L}_{(x,z)}$
    is the linear projection with the kernel $\mathrm{span}=\{\overline{\partial}_{z^\alpha},
    \forall\alpha\}=\ker\pi_*\cap
    T_{(x,z)}\overline{M}$;
    \item at each $\overline{p}\in\mathcal{L}$,
$\Phi^{T(T\overline{M})|_{\mathcal{L}}}_{T\mathcal{L}}
|_{\overline{p}}:T_{\overline{p}}(T\overline{M})\rightarrow
T_{\overline{p}}\mathcal{L}$ is the linear projection with the kernel
$\mathrm{span}\{\overline{\partial}_{w^\beta},
\forall\beta\}=\ker(\pi_*)_*\cap
\ker\psi_*$.
\end{enumerate}
\end{lemma}

In later discussion, we will also use $\Phi^{\mathcal{L}\backslash0}_{TM\backslash0}=
\Phi^{\mathcal{L}}_{TM}|_{\mathcal{L}\backslash0}$, with the same local representation $(x,z,y,l(x,z,y))\mapsto(x,y)$ as $ \Phi^{\mathcal{L}}_{TM}$ and
the extra requirement that $y\neq0$.

\begin{lemma} \label{lemma-3-4}
(1) Denote $(\Phi^{\mathcal{L}\backslash0}_{TM\backslash0})_*
:T(\mathcal{L}\backslash0)\rightarrow T(TM\backslash0)$  the tangent map of $\Phi^{\mathcal{L}\backslash0}_{TM\backslash0}$, then
$(\Phi^{\mathcal{L}\backslash0}_{TM\backslash0})_*
(\overline{\mathbf{G}}|_{\mathcal{L}\backslash0})
=\mathbf{G}$. That means
$(\Phi^{\mathcal{L}\backslash0}_{TM\backslash0})_*
(\overline{\mathbf{G}}(x,z,y,l(x,z,y)))=\mathbf{G}(x,y)$
at any point $\overline{p}=(x,z,y,l(x,z,y))\in\mathcal{L}\backslash0$.

(2) At each point $\overline{p}=(x,z,y,l(x,z,y))\in\mathcal{L}\backslash0$,
$$\overline{\delta}_{x^i}+
\tfrac{\partial}{\partial y^i}l^\alpha(x,z)
\overline{\delta}_{z^\alpha}\quad\mbox{and}\quad
\overline{\partial}_{x^i}-\mathbf{N}^j_i(x,y)
\overline{\partial}_{y^j}+
\tfrac{\partial}{\partial y^i}l^\alpha(x,z)
\overline{\partial}_{z^\alpha}$$
have the same $\Phi^{T(T\overline{M})|_{\mathcal{L}}}_{T\mathcal{L}}$-image,
which is mapped by
$(\Phi^{\mathcal{L}\backslash0}_{M\backslash0})_*$
to $\delta_{x^i}=\partial_{x^i}-\mathbf{N}^j_i\partial_{y^j}$ at $(x,y)$.
\end{lemma}

\begin{proof} (1) The tangent map
$(\Phi^{\mathcal{L}\backslash0}_{TM\backslash0})_*$
maps any tangent vector $a^i\overline{\partial}_{x^i}+b^j\overline{\partial}_{y^j}
+c^\alpha\overline{\partial}_{z^\alpha}+
d^\beta\overline{\partial}_{w^\beta}$ of $\mathcal{L}\backslash0$ at $\overline{p}=(x,z,y,l(x,z,y))$
to the tangent vector
$a^i{\partial}_{x^i}+b^j{\partial}_{y^j}$
of $TM\backslash0$ at $(x,y)$. So (1) in Lemma \ref{lemma-3-4} follows immediately after the local representations (\ref{041}) and (\ref{042}), and
(1) in Lemma \ref{lemma-3-1}.

(2) Recall that
\begin{eqnarray*}
\overline{\delta}_{x^i}=\overline{\partial}_{x^i}
-\overline{\mathbf{N}}_i^j\overline{\partial}_{y^j}
-\overline{\mathbf{N}}_i^\beta\overline{\partial}_{w^\beta}
\quad\mbox{and}\quad
\overline{\delta}_{z^\alpha}=\overline{\partial}_{z^\alpha}
-\overline{\mathbf{N}}_{\alpha}^j\overline{\partial}_{y^j}
-\overline{\mathbf{N}}_\alpha^\beta
\overline{\partial}_{w^\beta}
\end{eqnarray*}
are the horizonal liftings for $\overline{\partial}_{x^i}$ and
$\overline{\partial}_{z^\alpha}$ on $\overline{M}$ respectively. By (2) in Lemma \ref{lemma-3-1}, at any  $\overline{p}=(x,z,y,l(x,z,y))\in\mathcal{L}\backslash0$, we have
\begin{eqnarray*}
& &\overline{\delta}_{x^i}+\tfrac{\partial}{\partial y^i}l^\alpha(x,z)\overline{\delta}_{z^\alpha}\\
&=&
\overline{\partial}_{x^i}
-(\overline{\mathbf{N}}_i^j+\tfrac{\partial}{\partial y^i}l^\alpha(x,z)\overline{\mathbf{N}}_\alpha^j)
\overline{\partial}_{y^j}
+\tfrac{\partial}{\partial y^i}l^\alpha(x,z)\overline{\partial}_{z^\alpha}
-(\overline{\mathbf{N}}^\beta_i+
\tfrac{\partial}{\partial y^i}l^\alpha(x,z)
\overline{\mathbf{N}}_\alpha^\beta)
\overline{\partial}_{w^\beta}\\
&=&\overline{\partial}_{x^i}-
\mathbf{N}^j_i(x,y)\overline{\partial}_{y^j}+
\tfrac{\partial}{\partial y^i}l^\alpha(x,z)
\overline{\partial}_{z^\alpha},\quad\mbox{(mod }\overline{\partial}_{w^\beta},\forall\beta{)}.
\end{eqnarray*}
The first statement in (2) of Lemma \ref{lemma-3-4} follows (\ref{043}) immediately, and the second is obvious by the previous description for $(\Phi^{\mathcal{L}\backslash0}_{TM\backslash0})_*$.
\end{proof}\medskip

\subsection{Liftings from $M$ to $\mathcal{L}$}
\label{subsection-3-4}

Besides the liftings provided by the horizonal and vertical distributions for $\mathbf{G}$ and $\overline{\mathbf{G}}$
(see Section \ref{subsection-2-4}), there are more involving $\mathcal{L}$.

Firstly, we can use $\mathcal{L}$ to lift smooth tangent vector fields and smooth curves.

Let $X$ be a smooth tangent vector field on $M$. Then there exists a unique smooth tangent vector field $\overline{X}$ on $\overline{M}$, such that at each $(x,z)\in\overline{M}$, $\overline{X}(x,z)\in\mathcal{L}_{(x,z)}$ and $\pi_*(\overline{X}(x,z))=X(x)$. We call this $\overline{X}$
the {\it lifting of $X$ in $\mathcal{L}$}.

Denote $c(t)$ and $\overline{c}(t)$ the smooth curves on $M$ and $\overline{M}$ respectively, which is defined when $t$ is close to $0$. We call $\overline{c}(t)$
a {\it lifting of} $c(t)$ {\it tangent to} $\mathcal{L}$ if $\pi(\overline{c}(t))=c(t)$ and $\dot{\overline{c}}(t)\in\mathcal{L}_{\overline{c}(t)}$ for all possible values of $t$. Given any smooth curve $c(t)=(x^i(t))$, its lifting $\overline{c}(t)=(x^i(t),z^\alpha(t))$ tangent to $\mathcal{L}$ is locally a solution of
the ODE
$$\dot{z}^\alpha(t)=l^\alpha(x(t),z(t),\dot{x}(t))
=\dot{x}^j(t)\cdot\tfrac{\partial}{\partial y^j}l^\alpha(x(t),z(t)),
\quad\forall\alpha.$$
So for any $(x_0,z_0)\in\pi^{-1}(c(0))\subset \overline{M}$,
there exists a unique lifting $\overline{c}(t)$ of $c(t)$ which is defined around $t=0$, tangent to $\mathcal{L}$ and satisfies $\overline{c}(0)=(x_0,z_0)$.
By definition, $c(t)$ is a geodesic on $(M,\mathbf{G})$ if and only if any lifting $\overline{c}(t)$
of $c(t)$ tangent to $\mathcal{L}$ is a geodesic on $(\overline{M},\overline{\mathbf{G}})$.

Nextly, we can use $\mathcal{L}$ to lift a smooth vector field along a curve.

Let $c(t)$ be a smooth curve on $M$, $X(t)$ a smooth vector field along $c(t)$,
$\overline{c}(t)$ the smooth curve on $\overline{M}$ which is the lifting of $c(t)$ tangent to $\mathcal{L}$. Then we call the smooth vector field  $\overline{X}(t)$ along $\overline{c}(t)$ the {\it lifting of $X(t)$ along $\overline{c}(t)$ in $\mathcal{L}$} if
$\overline{X}(t)\in\mathcal{L}_{\overline{c}(t)}$ and $\Phi_{TM}^\mathcal{L}(\overline{X}(t))
=\pi_*(\overline{X}(t))=X(t)$
 for all values of $t$. Locally, for $X(t)=X^i(t)\partial_{x^i}$
along $c(t)$, its lifting along $\overline{c}(t)$ in $\mathcal{L}$ can be presented as
$\overline{X}(t)=X^i(t)(\overline{\partial}_{x^i}
+\tfrac{\partial}{\partial y^i}l^\alpha(\overline{c}(t))
\overline{\partial}_{z^\alpha})$. In particular, $\dot{\overline{c}}(t)=\dot{c}^i(t)(\overline{\partial}_{x^i}
+\tfrac{\partial}{\partial y^i}l^\alpha(\overline{c}(t)))
\overline{\partial}_{z^\alpha}$ along $\overline{c}(t)$
is the lifting in $\mathcal{L}$ of $\dot{c}(t)$ along $c(t)$.

\begin{lemma}\label{lemma-3-5}
Let $\overline{c}(t)$ and $\overline{X}(t)$ be the liftings of $ {c}(t)$ and $ {X}(t)$ respectively, as mentioned above. Suppose $\dot{c}(t)$ is nowhere-vanishing, then
\begin{equation}\label{044}
\Phi^{\mathcal{L}}_{TM}
(\Phi^{T\overline{M}}_{\mathcal{L}}
(\overline{D}_{\dot{\overline{c}}(t)}\overline{X}(t)))
=D_{\dot{c}(t)}X(t),
\end{equation}
in which $D$ and $\overline{D}$ are
the linearly covariant derivatives for $(M,\mathbf{G})$ and
$(\overline{M},\overline{\mathbf{G}})$. That means, for the smooth vector field $D_{\dot{c}(t)}X(t)$ along $c(t)$,
$\Phi^{T\overline{M}}_{\mathcal{L}}
(\overline{D}_{\dot{\overline{c}}(t)}\overline{X}(t))$ is its lifting along $\overline{c}(t)$ in $\mathcal{L}$.
In particular, $X(t)$ is linearly parallel along $c(t)$ if and only if $\Phi^{T\overline{M}}_{\mathcal{L}}
(\overline{D}_{\dot{\overline{c}}(t)}\overline{X}(t))=0$.
\end{lemma}
\begin{proof} Firstly, we prove (\ref{044}).
It is easy to see the following facts:
\begin{enumerate}
\item both sides of (\ref{044}) are $\mathbb{R}$-linear;
\item for any smooth real function $f(t)$, the lifting of $f(t)X(t)$ in $\mathcal{L}$ along $\overline{c}(t)$ coincides with $f(t)\overline{X}(t)$;
\item the Lebniz Rule provides $D_{\dot{c}(t)}(f(t)X(t))=(\tfrac{{\rm d}}{{\rm d}t}f(t)) X(t)+
f(t)D_{\dot{c}}X(t)$ and
\begin{eqnarray*}
\Phi^{\mathcal{L}}_{TM}
(\Phi^{T\overline{M}}_{\mathcal{L}}
(\overline{D}_{\dot{\overline{c}}(t)}(f(t)\overline{X}(t))))=
(\tfrac{{\rm d}}{{\rm d}t}f(t))X(t)+f(t)\Phi^{\mathcal{L}}_{TM}
(\Phi^{T\overline{M}}_{\mathcal{L}}
(\overline{D}_{\dot{\overline{c}}(t)}\overline{X}(t))).
\end{eqnarray*}
So when we replace $X(t)$ with $f(t)X(t)$, the same extra term appears in both sides of (\ref{044}).
\end{enumerate}
To summarize, we only need to prove (\ref{044}) for $X(t)=\partial_{x^i}$,
and without loss of generality, we may assume $X(t)=\partial_{x^1}$.

Locally we denote $c(t)=x(t)=(x^i(t))$ and $\overline{c}(t)=(x(t),z(t))=(x^i(t),z^\alpha(t))$, with
$\dot{z}^\alpha(t)=l^\alpha(x(t),z(t),\dot{x}(t))$ for all $\alpha$.
Then $\overline{X}(t)=\overline{\partial}_{x^1}
+\tfrac{\partial}{\partial y^1}l^\alpha(x(t),z(t))
\overline{\partial}_{z^\alpha}$.
Direct calculation shows
\begin{equation}\label{046}
D_{\dot{c}(t)}X(t)=
\mathbf{N}_1^j(x(t),\dot{x}(t))\partial_{x^j}|_{c(t)},
\end{equation}
and
\begin{eqnarray}
\overline{D}_{\dot{\overline{c}}(t)}\overline{X}(t)
&=&(\overline{\mathbf{N}}_1^j
(\overline{c}(t),\dot{\overline{c}}(t))\overline{\partial}_{x^j}
+\overline{\mathbf{N}}_1^\beta
(\overline{c}(t),\dot{\overline{c}}(t))
\overline{\partial}_{z^\beta}
+\tfrac{{\rm d}}{{\rm d}t}(\tfrac{\partial}{\partial y^1}l^\alpha(x(t),z(t)))\overline{\partial}_{z^\alpha}
\nonumber\\
& &+\tfrac{\partial}{\partial y^1}l^\alpha(x(t),z(t))
\overline{\mathbf{N}}_\alpha^j
(\overline{c}(t),\dot{\overline{c}}(t))\overline{\partial}_{x^j}
+\tfrac{\partial}{\partial y^1}l^\alpha(x(t),z(t))
\overline{\mathbf{N}}_\alpha^\beta
(\overline{c}(t),\dot{\overline{c}}(t))
\overline{\partial}_{z^\beta})|_{\overline{c}(t)}\nonumber\\
&=&(\overline{\mathbf{N}}_1^j
(\overline{c}(t),\dot{\overline{c}}(t))+
\tfrac{\partial}{\partial y^1}l^\alpha
(x(t),z(t))\overline{\mathbf{N}}_\alpha^j
(\overline{c}(t),\dot{\overline{c}}(t)))
\overline{\partial}_{x^j}|_{\overline{c}(t)}
\ \mbox{(mod }\overline{\partial}_{z^\alpha},\forall \alpha\mbox{)}.\label{045}
\end{eqnarray}
Apply (2) in Lemma \ref{lemma-3-1} at
$(\overline{c}(t),\dot{\overline{c}}(t))
=(x(t),z(t),\dot{x}(t),l(x(t),z(t),\dot{x}(t)))
\in\mathcal{L}\backslash0$,
we see from (\ref{045}) that
\begin{equation}\label{047}
\overline{D}_{\dot{\overline{c}}(t)}\overline{X}(t)
={\mathbf{N}}^j_1(x(t),\dot{x}(t))\overline{\partial}_{x^j}
|_{\overline{c}(t)},\quad\mbox{(mod }\overline{\partial}_{z^\alpha},\forall\alpha\mbox{)}.
\end{equation}

Comparing (\ref{046}) and (\ref{047}), we get
(\ref{044}) for $X=\partial_{x^1}$. Previous observations
end the proof of (\ref{044}) for other $X$.

The second statement in Lemma \ref{lemma-3-5} is just an interpretation for (\ref{044}).

The third statement follows after
the first immediately, because at each point $(x,z)\in\overline{M}$, $\Phi^{\mathcal{L}}_{TM}$ induces
a linear isomorphism from $\mathcal{L}_{(x,z)}$ to $T_xM$.
\end{proof}\medskip

Finally, we consider lifting  $\widetilde{\dot{c}(t)}^\mathcal{H}$ to $\mathcal{L}$.

Suppose $\overline{c}(t)$  is a lifting of $c(t)$ tangent to $\mathcal{L}$, with $t\in(a,b)$. Without loss of generality, we may assume the interval $(a,b)$ is sufficiently small and $\dot{c}(t)$ is nowhere-vanishing. Then
$N=\cup_{t\in(a,b)}(T_{c(t)}M\backslash\{0\})$,
 $\overline{N}=\cup_{t\in(a,b)}
 (T_{\overline{c}(t)}\overline{M}\backslash\{0\})$ and
$N_\mathcal{L}=\cup_{t\in(a,b)}
(\mathcal{L}_{(\overline{c}(t),
\dot{\overline{c}}(t))}\backslash\{0\})$
are imbedded submanifolds in $TM\backslash0$, $T\overline{M}\backslash0$ and $\overline{N}$
respectively, and
$\Phi^{N_\mathcal{L}}_N=
\Phi^{\mathcal{L}\backslash0}_{TM\backslash0}
|_{N_\mathcal{L}}:N_{\mathcal{L}}
\rightarrow N$ is a diffeomorphism.

If we locally denote $c(t)=x(t)=(x^i(t))$, then the horizonal lifting $$\widetilde{\dot{c}(t)}^\mathcal{H}=\dot{x}^i(t)\delta_{x^i}
=
\dot{x}^i(t)(\partial_{x^i}-\mathbf{N}^j_i(x(t),y)
\partial_{y^j})$$
at any $(x(t),y)\in T_{c(t)}M\backslash\{0\}$
is a smooth tangent vector field on $N$. We call the smooth tangent vector field
$((\Phi_{N}^{N_\mathcal{L}})^{-1})_*
(\widetilde{\dot{c}(t)}^\mathcal{H})$ on $N_\mathcal{L}$
 a {\it lifting of $\widetilde{\dot{c}(t)}^\mathcal{H}$ to $\mathcal{L}$}.

Meanwhile, if we denote $\overline{c}(t)=(x(t),z(t))=(x^i(t),z^\alpha(t))$, then
the horizonal lifting $\widetilde{\dot{\overline{c}}(t)}^{\mathcal{\overline{H}}}$
is a smooth tangent vector field on $\overline{N}$, determined by
\begin{eqnarray*}
\widetilde{\dot{\overline{c}}(t)}^{\mathcal{\overline{H}}}
&=&\dot{x}^i(t)\overline{\delta}_{x^i}+\dot{z}^\alpha(t)
\overline{\delta}_{z^\alpha}\\
&=&\dot{x}^i(t)(\overline{\partial}_{x^i}-
\overline{\mathbf{N}}^j_i(\overline{c}(t),y,w)
\overline{\partial}_{y^j}-
\overline{\mathbf{N}}^\beta_i(\overline{c}(t),y,w)
\overline{\partial}_{w^\beta})\\& &+
\dot{z}^\alpha(t)(\overline{\partial}_{z^\alpha}-
\overline{\mathbf{N}}^j_\alpha(\overline{c}(t),y,w)
\overline{\partial}_{y^j}-
\overline{\mathbf{N}}^\beta_\alpha(\overline{c}(t),y,w)
\overline{\partial}_{w^\beta})
\end{eqnarray*}
at each $\overline{p}=(\overline{c}(t),y,w)\in T_{c(t)}\overline{M}\backslash\{0\}$.

The relation between a lifting of $\widetilde{\dot{c}(t)}^\mathcal{H}$ to $\mathcal{L}$ and
$\widetilde{\dot{\overline{c}}(t)}^{\mathcal{\overline{H}}}$
is revealed  by the following lemma.

\begin{lemma}\label{lemma-3-6} $\Phi^{T(T\overline{M})|_{\mathcal{L}}}_{T\mathcal{L}}
(\widetilde{\dot{\overline{c}}(t)}^{\mathcal{\overline{H}}}
|_{N_\mathcal{L}})=((\Phi_{N}^{N_\mathcal{L}})^{-1})_*
(\widetilde{\dot{c}(t)}^\mathcal{H})$.
\end{lemma}

\begin{proof}
Since $\overline{c}(t)=(x(t),z(t))=(x^i(t),z^\alpha(t))$ is a lifting of $c(t)=x(t)=(x^i(t))$, which is tangent to $\mathcal{L}$, we have
$\dot{z}^\alpha(t)=l^\alpha(x(t),z(t),\dot{x}(t))
=\dot{x}^i(t)\tfrac{\partial}{\partial y^i}l^\alpha(x(t),z(t))$.
So $$\dot{\overline{c}}(t)=\dot{x}^i(t)\overline{\partial}_{x^i}
+\dot{z}^\alpha(t)\overline{\partial}_{z^\alpha}=
\dot{x}^i(t)(\overline{\partial}_{x^i}+
\tfrac{\partial}{\partial y^i}l^\alpha(x(t),z(t))\overline{\partial}_{z^\alpha})$$ at each $\overline{c}(t)$, and
$$\widetilde{\dot{\overline{c}}(t)}^{\overline{\mathcal{H}}}
=\dot{x}^i(t)(\overline{\delta}_{x^i}+
\tfrac{\partial}{\partial y^i}l^\alpha(x(t),z(t))\overline{\delta}_{z^\alpha})$$
at each $\overline{p}=(\overline{c}(t),y,w)\in T_{\overline{c}(t)}M\backslash
\{0\}$. When we have
$\overline{p}=(x(t),z(t),y,l(x(t),z(t),y))\in N_{\mathcal{L}}$,
$T_{\overline{p}}N_{\mathcal{L}}$ is linearly spanned by
$\Phi^{T(T\overline{M})|_{\mathcal{L}}}_{T\mathcal{L}}
(\dot{x}^i(t)(\overline{\partial}_{x^i}+
\tfrac{\partial}{\partial y^i}l^\alpha(x(t),z(t))\overline{\partial}_{z^\alpha}))$ and
$\Phi^{T(T\overline{M})
|_{\mathcal{L}}}_{T\mathcal{L}}(\overline{\partial}_{y^j})$
for all $j$. So the first statement in (2) of Lemma \ref{lemma-3-4} indicates
that $\Phi^{T(T\overline{M})}_{T\mathcal{L}}
(\widetilde{\dot{\overline{c}}(t)}^{\overline{\mathcal{H}}}
|_{N_\mathcal{L}})$ is tangent to $N_\mathcal{L}$. Then
the second statement in (2) of Lemma \ref{lemma-3-4} indicates
it coincides with the lifting of $\widetilde{\dot{c}(t)}^\mathcal{H}$ to $\mathcal{L}$.
\end{proof}

\subsection{Riemann curvature formula for a submersion}
\label{subsection-3-5}
Lemma \ref{lemma-3-5} and Lemma \ref{lemma-3-6} provide the submersion techniques for
parallel translations.
Now we consider the submersion technique for the Riemann curvature.

Denote $\mathcal{K}$ the linear sub-bundle of $T(T\overline{M})$ such that each point $\overline{p}=(x,z,y,w)\in T\overline{M}$, its fiber $\mathcal{K}_{\overline{p}}$ is the kernel of
$(\pi_*)_*:T(T\overline{M})\rightarrow T(TM)$, i.e.,
$\mathcal{K}_{\overline{p}}=
\mathrm{span}\{\overline{\partial}_{z^\alpha},
\overline{\partial}_{w^\alpha},\forall \alpha\}$.

The tangent bundle $T(\mathcal{L}\backslash\{0\})$
has two linear sub-bundles, $\mathcal{H}'$ and
$\mathcal{K}'$. At each point $\overline{p}=(x,z,y,l(x,z,y))\in\mathcal{L}\backslash0$, the fiber $\mathcal{H}'_{\overline{p}}$ consists of the  $\Phi^{T(T\overline{M})
|_{\mathcal{L}}}_{T\mathcal{L}}$-images
of the $\overline{\mathcal{H}}$-horizonal liftings of all vectors
in $\mathcal{L}_{(x,z)}$, i.e., it is linearly spanned by
 $$
 \Phi^{T(T\overline{M})|_{\mathcal{L}}}_{\mathcal{L}}
(\overline{\delta}_{x^i}+\tfrac{\partial}{\partial y^i}l^\alpha(x,z,y)
\overline{\delta}_{z^\alpha}), \quad\forall i,
$$
and the fiber $\mathcal{K}'_{\overline{p}}$  is $T_{\overline{p}}(\mathcal{L}\backslash0)\cap \mathcal{K}_{\overline{p}}=
\Phi^{T(T\overline{M})|_{\mathcal{L}}}_{T\mathcal{L}}
(\mathcal{K}_{\overline{p}})$, i.e., it
is linearly spanned by
$\overline{\partial}_{z^\alpha}+\tfrac{\partial}{\partial z^\alpha}l^\beta(x,z,y)\overline{\partial}_{w^\beta}$, $\forall \alpha.$

The following lemma interprets the Riemann curvature of $(M,\mathbf{G})$ by the submersion $(\pi,\mathcal{L})$.

\begin{lemma}\label{lemma-3-7}
Let $\overline{X}$ be a smooth section of $\mathcal{H}'$, which is defined around $\overline{p}=(x,z,y,l(x,z,y))$ in $\mathcal{L}\backslash0$ and satisfies
$\overline{X}(\overline{p})=
\Phi^{T(T\overline{M})|_{\mathcal{L}}}_{\mathcal{L}}
(a^i(\overline{\delta}_{x^i}+\tfrac{\partial}{\partial y^i}
l^\alpha(x,z,y)\overline{\delta}_{z^\alpha}))$,
then we have
\begin{eqnarray}\label{075}
[\overline{\mathbf{G}}|_{\mathcal{L}\backslash0},
\overline{X}]|_{\overline{p}}
= a^i\mathbf{R}^k_i(x,y)\overline{\partial}_{y^k}\quad
\mbox{(mod }\mathcal{H}'\oplus\mathcal{K}'\mbox{)}.
\end{eqnarray}
\end{lemma}
\begin{proof}
Firstly, we prove (\ref{075}) for $\overline{X}=\Phi^{T(T\overline{M})
|_{\mathcal{L}}}_{T\mathcal{L}}(\overline{\delta}_{x^i}+\tfrac{\partial}{\partial y^i}
l^\alpha(x,z,y)\overline{\delta}_{z^\alpha})$. By (2) in Lemma \ref{lemma-3-4}, $(\Phi^{\mathcal{L}\backslash0}_{TM\backslash0})_*
(\overline{X})
=\delta_{x^i}=\partial_{x^i}-\mathbf{N}^j_i\partial_{y^j}$.
By (1) in Lemma \ref{lemma-3-4}, $(\Phi^{\mathcal{L}\backslash0}_{TM\backslash0})_*
(\overline{\mathbf{G}}|_{\mathcal{L}\backslash0})
=\mathbf{G}$.
Using the property of tangent map and Lemma \ref{lemma-2-1}, we get at any
$\overline{p}=(x,z,y,l(x,z,y))\in\mathcal{L}\backslash0$ that
\begin{eqnarray}
(\Phi^{\mathcal{L}\backslash0}_{TM\backslash0})_*
([\overline{\mathbf{G}}
|_{\mathcal{L}\backslash0},\overline{X}]|_{\overline{p}})
&=&[(\Phi^{\mathcal{L}\backslash0}_{TM\backslash0})_*
(\overline{\mathbf{G}}|_{\mathcal{L}\backslash0}),
(\Phi^{\mathcal{L}\backslash0}_{TM\backslash0})_*
(\overline{X})]|_{(x,y)}\nonumber\\
&=&
[\mathbf{G},\delta_{x^i}]|_{(x,y)}= \mathbf{R}_i^k(x,y)\partial_{y^k}\quad\mbox{(mod }\mathcal{H}\mbox{)}.\label{076}
\end{eqnarray}
Then (\ref{076}) implies (\ref{075}) for that special
$\overline{X}$, using the following observations:
\begin{enumerate}
\item by (2) in Lemma \ref{lemma-3-4}, $(\Phi^{\mathcal{L}\backslash0}_{TM\backslash0})_*$
maps each fiber of $\mathcal{H}'$ isomorphically to that of
$\mathcal{H}=\mathrm{span}\{\delta_{x^i},\forall i\}$;
\item $(\Phi^{\mathcal{L}\backslash0}_{TM\backslash0})_*$
maps each fiber of $\mathcal{K}'$ to 0;
\item At $\overline{p}=(x,z,y,l(x,z,y))\in\mathcal{L}\backslash0$, any vector in the quotient space $T_{\overline{p}}
    (\mathcal{L}\backslash0)/(\mathcal{H}'_{\overline{p}}
    \oplus\mathcal{K}_{\overline{p}})$ can be uniquely represented by a vector in the subspace $$\mathrm{span}\{\overline{\partial}_{y^i}+
    \tfrac{\partial}{\partial y^i}l^\beta(x,z)\overline{\partial}_{w^\beta},\forall i\}\subset T_{\overline{p}}(\mathcal{L}\backslash0).$$
\end{enumerate}

Nextly, we consider any smooth section
$\overline{X}$, which is a $C^\infty$-linear combination of the sections
$\Phi^{T(T\overline{M})|_{\mathcal{L}}}_{\mathcal{L}}(\overline{\delta}_{x^i}+\tfrac{\partial}{\partial y^i}
l^\alpha(x,z,y)\overline{\delta}_{z^\alpha})$ of $\mathcal{H}'$. Since (\ref{075}) is mod $\mathcal{H}'\oplus\mathcal{K}'$, its validity is not affected by the smooth coefficients.
\end{proof}\medskip
\section{Left invariant spray geometry}

In this section, we recall the technique of global invariant frames on a Lie group and collect some known results for a left invariant spray structure from \cite{Xu2021-1,Xu2021-2}.

\subsection{Invariant frames on a Lie group and its tangent bundle}
\label{subsection-4-1}

Let $G$ be a Lie group. Denote $L_g(g')=gg'$ and $R_g(g')=g'g$ for all $g,g'\in G$ the left and right translations. Let $\mathfrak{g}=T_eG$ be the Lie algebra of $G$, with the bracket
$[\cdot,\cdot]_\mathfrak{g}$. Notice that in this paper, we reserve the more usual notation $[\cdot,\cdot]$ for the Lie bracket between
smooth vector fields.
We fix a basis $\{e_1,\cdots,e_m\}$ of $\mathfrak{g}$
and denote $c^k_{ij}$ the corresponding Lie bracket coefficients in $[e_i,e_j]_\mathfrak{g}=c^k_{ij}e_k$.

For each $i\in\{1,\cdots,m\}$, there exists a left invariant vector field $U_i$, and a right invariant vector field $V_i$, such that $U_i(e)=V_i(e)=e_i$. Then any left (or right) invariant vector field on $G$ is a $\mathbb{R}$-linear combination of $U_i$'s (or $V_i$'s respctively). Left and right invariant vector fields generate
right and left translations respectively on the Lie group.

We call $\{U_i,\forall i\}$
a {\it left invariant frame on $G$}. We expand it to a left invariant
frame on $TG$ as following.
Let $\widetilde{U}_i$ be
the complete lifting of $U_i$ (see (\ref{050}) and (\ref{051})
for its standard local coordinate representation).
Any $y\in T_gG$
can be uniquely presented as $y=u^i U_i(g)$. So we have
the frame $\{\partial_{u^i},\forall i\}$ in $T_gG$ corresponding to this linear coordinate. With $g$ exhausting all points of $G$, each $\partial_{u^i}$ is then a smooth tangent vector field on $TG$. Since all $\widetilde{U}_i$'s and
$\partial_{u^j}$'s are invariant with respect to the action of $((L_g)_*)_*$ for all $g\in G$, we call
$\{\widetilde{U}_i,\partial_{u^i},\forall i\}$ a
{\it left invariant frame on $TG$}.
Similarly, using $V_i$ and $\partial_{v^i}$ for $y=v^i V_i(g)\in T_gG$, we have a right invariant
frame $\{\widetilde{V}_i,\partial_{v^i},\forall i\}$ on $TG$.

Let $\{\partial_{x^i},\partial_{y^i},\forall i\}$ be the frame
for a standard local coordinate $(x^i,y^j)$ on $G$. Then
the transformation between
$\{U_i,\partial_{u^i},\forall i\}$ and $\{\partial_{x^i},\partial_{y^i},\forall i\}$
is the following \cite{Xu2021-2},
\begin{eqnarray}
& & U_i=A^j_i\partial_{x^i},\quad u^i=y^j B^i_j\quad\mbox{and}\quad \partial_{u^i}=A^j_i\partial_{u^i},\label{050}\\
& &\widetilde{U}_i=A^j_i\partial_{x^i}+
y^j\tfrac{\partial}{\partial x^j}A^k_i\partial_{y^i},\label{051}
\end{eqnarray}
where $(A^j_i)=(A^j_i(x))$ and $(B^j_i)=(B^j_i(x))=
(A^j_i(x))^{-1}$ (i.e., $A^j_iB^k_j=B^j_iA^k_j=\delta^k_i$) are matrix valued functions which only depend on the $x$-entry.
Notice that $\partial_{x^i}$'s in (\ref{050}) and
(\ref{051}) are local tangent vector fields on $G$ and their complete lifting to $TG$ respectively.

The invariancy of $\{\widetilde{U}_i,\partial_{u^i},\forall i\}$
and $\{\widetilde{V}_i,\partial_{v^i},\forall i\}$ implies the following obvious facts for every $i$ and $j$ (see (8) in \cite{Xu2021-1})
\begin{eqnarray}
& &[U_i,U_j]=c^k_{ij}U_k,\quad [V_i,V_j]=-c^k_{ij}V_k,\quad
[U_i,V_j]=0,\nonumber\\
& &[\widetilde{U}_i,\widetilde{U}_j]=c^k_{ij}\widetilde{U}_k,\quad [\widetilde{V}_i,\widetilde{V}_j]=-c^k_{ij}\widetilde{V}_k,\quad
[\widetilde{U}_i,\widetilde{V}_j]=0,\nonumber\\
& &\widetilde{U}_i v^j=0,\quad
[\widetilde{U}_i,\partial_{v^j}]=0,\quad
\widetilde{V}_i u^j=0,\quad
[\widetilde{V}_i,\partial_{u^j}]=0.\label{063}
\end{eqnarray}

To provide more relations between the left and right invariant frames on $TG$, we need to introduce the functions
$\phi_i^j$ and $\psi_i^j$ on $G$ such that $\mathrm{Ad}(g)e_i=\phi_i^j(g) e_j$ and $\mathrm{Ad}(g^{-1})e_i=\psi_i^j(g)e_j$. Then
we have
\begin{equation}
U_i=\phi_i^j V_j,\quad u^i=\psi_j^i v^j,\quad \partial_{u^i}=\phi_i^j\partial_{v^j},
\end{equation}
and (see (9) and Lemma 3.1 in \cite{Xu2021-1})

\begin{lemma} \label{lemma-4-1}
Keep above notations, then at each point of $G$, we have the following
$$
\mbox{(1) } \phi_l^j V_j\phi_i^k=c_{li}^j\phi_j^k,\
\mbox{(2) } \widetilde{U}_i=\phi_i^j\widetilde{V}_j+
c^q_{pi}u^p\partial_{u^q},\
\mbox{(3) }\widetilde{U}_i u^j=c^j_{li}u^l,\
\mbox{(4) } [\widetilde{U}_i,\partial_{u^l}]=c^p_{il}\partial_{u^p}.
$$
\end{lemma}
\subsection{Left invariant spray structure on a Lie group}

\label{subsection-4-2}
A spray structure $\mathbf{G}$ on the Lie group $G$ is called
{\it left invariant} (or {\it right invariant}), if for each $g\in G$, $((L_g)_*)_*\mathbf{G}=\mathbf{G}$ (or $((L_g)_*)_*\mathbf{G}=\mathbf{G}$ respectively). It is called {\it bi-invariant} if it is both left and right invariant.

There exists a canonical affine bi-invariant spray structure $\mathbf{G}_0$ on $G$ (see Theorem A in \cite{Xu2021-1}).

\begin{theorem}\label{theorem-4-2}
On any Lie group $G$, $\mathbf{G}_0=u^i\widetilde{U}_i=
v^i\widetilde{V}_i$ is an affine bi-invariant spray structure.
\end{theorem}

Using $\mathbf{G}_0$ as the origin, any left invariant spray structure $\mathbf{G}$ on $G$ can be presented as
$\mathbf{G}=\mathbf{G}_0-\mathbf{H}$. Here  $\mathbf{H}=\mathbf{H}^i\partial_{u^i}$ is a smooth vector field on $ TG\backslash0$, tangent to each $T_gG$, in which
each $\mathbf{H}^i=\mathbf{H}^i(x,y)$ is positively 2-homogeneous for its $y$-entry. The restriction $\eta=\mathbf{H}|_{T_eG\backslash\{0\}}$ is called the {\it spray vector field} for $\mathbf{G}$. We sometimes view $\eta$ as a positively 2-homogeneous smooth map
from $\mathfrak{g}\backslash\{0\}$ to $\mathfrak{g}$, i.e.,
$\eta(y)=\mathbf{H}^i(e,y)e_i$, $\forall y\in\mathfrak{g}\backslash\{0\}$. Obviously, the correspondence between the set of all left invariant spray structures on $G$ and the set of all positive 2-homogeneous smooth maps $\eta:
\mathfrak{g}\backslash\{0\}\rightarrow\mathfrak{g}$, mapping each $\mathbf{G}$ to its spray vector field, is one-to-one.

Based on the spray vector field $\eta:\mathfrak{g}\backslash\{0\}
\rightarrow\mathfrak{g}$, we further define the {\it connection operator} $N:\mathfrak{g}\backslash\{0\}\times\mathfrak{g}
\rightarrow\mathfrak{g}$ by
$N(y,v)=\tfrac12D\eta(y,v)-\tfrac12[y,v]_\mathfrak{g}$, in which $D\eta(y,v)=\tfrac{{\rm d}}{{\rm d}t}|_{t=0}\eta(y+tv)$.

Using left translations, the geometry of a left invariant spray structure $\mathbf{G}$ can be described by its spray vector field and connection operator. The following results are already known (see Theorem D in \cite{Xu2021-1}, and Theorem 1.1, Theorem 1.4 and Lemma 3.2 in \cite{Xu2021-2}).

\begin{theorem}\label{theorem-4-3}
Let $\mathbf{G}$ be a left invariant spray structure on a Lie group $G$ with the spray vector field $\eta$. Then for any open interval $(a,b)\subset\mathbb{R}$ containing $0$, there is a one-to-one correspondence between the following two sets:
\begin{enumerate}
\item the set of all $c(t)$ with $t\in(a,b)$ and $c(0)=e$, which are geodesics for $\mathbf{G}$;
\item the set of all $y(t)$ with $t\in(a,b)$, which are integral curves of $-\eta$.
\end{enumerate}
The correspondence is from a geodesic $c(t)$ on $(G,\mathbf{G})$ to the curve $y(t)=(L_{c(t)^{-1}})_*(\dot{c}(t))$ on $\mathfrak{g}\backslash\{0\}$.
\end{theorem}

\begin{theorem}\label{theorem-4-4}
Let $\mathbf{G}=\mathbf{G}_0-\mathbf{H}$ be a left invariant spray structure on a Lie group $G$ with the connection operator $N$, $c(t)$ a smooth curve on $G$ with nowhere-vanishing $\dot{c}(t)$, and
$W(t)$ a smooth vector field along $c(t)$. Denote $y(t)=(L_{c(t)^{-1}})_*(\dot{c}(t))
=u^i(t)e_i\in\mathfrak{g}\backslash\{0\}$ and $w(t)=(L_{c(t)^{-1}})_*(W(t))\in\mathfrak{g}$. Then
\begin{equation*}
D_{\dot{c}(t)}W(t)=(\tfrac{{\rm d}}{{\rm d}t}w^l(t)+\tfrac12
w^j(t)\tfrac{\partial}{\partial u^j}\mathbf{H}^l(c(t),\dot{c}(t))+\tfrac12 w^j(t) u^k(t)
c^l_{kj})U_l(c(t)).
\end{equation*}
In particular, $W(t)$ is linearly parallel along $c(t)$ if and only if
\begin{equation*}
\tfrac{{\rm d}}{{\rm d}t}w(t)+N(y(t),w(t))+[y(t),w(t)]_\mathfrak{g}=0
\end{equation*}
is satisfied everywhere.
\end{theorem}
\begin{remark} \label{remark-4-5}
Theorem 1.1 and Lemma 3.2 in \cite{Xu2021-2}
are only stated in the case that $c(t)$ is a geodesic on $(G,\mathbf{G})$, but their proofs are valid for any smooth curve $c(t)$ with nowhere-vanishing $\dot{c}(t)$. So they are summarized and reformulated as Theorem \ref{theorem-4-4} here.
\end{remark}
\begin{theorem}\label{theorem-4-6} Let $\mathbf{G}=\mathbf{G}_0-\mathbf{H}$ be a left invariant spray structure on a Lie group $G$ with the connection operator $N$, $c(t)$ a smooth curve on $G$ and $Y(t)$ a nowhere-vanishing vector field along $c(t)$. Denote $w(t)=
(L_{c(t)^{-1}})_*(\dot{c}(t))$ and $y(t)=(L_{c(t)^{-1}})_*(Y(t))$. Then $Y(t)$ is nonlinearly parallel along $c(t)$ if and only if
\begin{equation*}
\tfrac{{\rm d}}{{\rm d}t}y(t)+N(y(t),w(t))=0
\end{equation*}
is satisfied everywere.
\end{theorem}

Theorem \ref{theorem-4-6} is based on the following calculation results
which are useful in later discussion (see Lemma 3.2 in \cite{Xu2021-1} and Lemma 4.2 in \cite{Xu2021-2}).

\begin{lemma} \label{lemma-4-8}
(1) The horizonal lifting of $U_q$ is
$\widetilde{U}_q^\mathcal{H}=\widetilde{U}_q-
(\tfrac12\tfrac{\partial}{\partial u^q}\mathbf{H}^i-
\tfrac12u^j c^i_{qj})\partial_{u^i}$.

(2) Denote $S=\cup_{t}(T_{c(t)}G\backslash\{0\})$
a submanifold of $TG\backslash0$ in which $c(t)$ is
a smooth curve in $G$ with nowhere-vanishing $\dot{c}(t)=w^i(t)U_i(c(t))$, and $\{\partial_{t},\partial_{u^i},\cdots,\partial_{u^m}\}$ the
global frame corresponding to the coordinate $(t,u^1,\cdots,u^m)$ on $S$. Then we have
\begin{equation*}
\widetilde{\dot{c}(t)}^\mathcal{H}=\partial_{t}+
(-\tfrac12 w^i(t)\tfrac{\partial}{\partial u^i}\mathbf{H}^j
+\tfrac12 w^i(t)u^p c^j_{pi})\partial_{u^j},
\end{equation*}
at each $(c(t),y)\in S$ with $y=u^i U_i(c(t))\in T_{c(t)}G\backslash\{0\}$.
\end{lemma}

Using the left invariant frame $\{\widetilde{U}_i,\partial_{u^i},\forall i\}$ on $TG$, we calculate the S-curvature and Riemann curvature for a left invariant spray structure (see Theorem B and Theorem C in \cite{Xu2021-1}),
which can be translated to the following (see Corollary 4.1 in \cite{Xu2021-1})

\begin{theorem} Let $\mathbf{G}$ be a left invariant spray structure on the Lie group $G$ with the spray vector field $\eta:\mathfrak{g}\backslash\{0\}\rightarrow\mathfrak{g}$ and the connection operator $N:\mathfrak{g}\backslash\{0\}\times\mathfrak{g}
\rightarrow\mathfrak{g}$, then its S-curvature for a left invariant smooth measure and its Riemann curvature satisfy
\begin{eqnarray*}
\mathbf{S}(e,y)&=&\mathrm{Tr}_\mathbb{R}
(N(y,\cdot)+\mathrm{ad}(y)),\quad\mbox{and}\\
\mathbf{R}_y(v)&=&DN(y,v,\eta(y))-N(y,N(y,v))+
N(y,[y,v]_\mathfrak{g})
-[y,N(y,v)]_\mathfrak{g},
\end{eqnarray*}
respectively, for any $y\in\mathfrak{g}\backslash\{0\}$ and $v\in\mathfrak{g}$. Here $DN(y,v,\eta(y))=\tfrac{{\rm d}}{{\rm d}t}|_{t=0}N(y+t\eta(y),v)$.
\end{theorem}

In next two sections, we will generalize above theorems for any homogeneous spray manifold.

\section{Homogeneous spray structure and submersion}

In this section, we define the spray vector field and the connection operator and construct the submersion for
a homogeneous spray structure.

\subsection{Homogeneous spray and Finsler geometries}
\label{subsection-5-1}

Let $\mathbf{G}$ be a spray structure on the smooth manifold $M^n$. We call $\mathbf{G}$ a {\it homogeneous} spray structure on $M$, or call $(M,\mathbf{G})$ a {\it homogeneous spray manifold}, if $M$ admits the transitive smooth action of the Lie group $G$, such that $(g_*)_*\mathbf{G}=\mathbf{G}$, $\forall g\in G$, or equivalently, the action of every $g\in G$ on $M$ maps geodesics to geodesics.

Practically, we may identify the homogeneous spray manifold $M$ with a smooth coset space
$G/H$, in which $H$ is the isotropy subgroup at the origin $o=eH$.
Now suppose that $(G/H,\mathbf{G})$ is equipped with a linear decomposition
$
\mathfrak{g}=\mathfrak{h}+\mathfrak{m}
$.
Then the subspace $\mathfrak{m}$ is identified with $T_o(G/H)$. When this decomposition is {\it reductive}, i.e., it is $\mathrm{Ad}(H)$-invariant, the $\mathrm{Ad}(H)$-action on $\mathfrak{m}$ coincides
with the isotropic $H$-action on $T_o(G/H)$.

For
any vector $v\in \mathfrak{g}$, we denote $v_\mathfrak{h}=\mathrm{pr}_\mathfrak{h}(v)$ and
$v_\mathfrak{m}=\mathrm{pr}_\mathfrak{m}(v)$, in which
$\mathrm{pr}_\mathfrak{h}:\mathfrak{g}\rightarrow\mathfrak{h}$ and $\mathrm{pr}_{\mathfrak{m}}:\mathfrak{g}\rightarrow
\mathfrak{m}$ are the linear projections  according to the given decomposition, and we have the operations $[\cdot,\cdot]_\mathfrak{h}=
\mathrm{pr}_{\mathfrak{h}}\circ
[\cdot,\cdot]_\mathfrak{g}$ and
$[\cdot,\cdot]_\mathfrak{m}=
\mathrm{pr}_{\mathfrak{m}}\circ
[\cdot,\cdot]_\mathfrak{g}$. 

A Finsler metric $F$ on a smooth manifold $M$ is called {\it homogeneous} if its isometry group $I(M,F)$ acts transitively on $M$ \cite{De2012}. We also call $(M,F)$ a {\it homogeneous Finsler manifold}. Notice that $I(M,F)$ is a Lie transformation group \cite{DH2002}. The geodesic spray $\mathbf{G}$ of a homogeneous Finsler metric $F$ is also homogeneous. So homogeneous Finsler geometry is a special case of homogeneous spray geometry.

Let $G$ be any closed Lie subgroup of $I(M,F)$, which acts transitively on the homogeneous Finsler manifold $(M,F)$. Then we have the representation $M=G/H$. A reductive decomposition $\mathfrak{g}=\mathfrak{h}+\mathfrak{h}$ can be easily observed as following. Because $G$ is closed in $I(M,F)$, the isotropy subgroup $H$ is compact, so there exists an $\mathrm{Ad}(H)$-invariant inner product on $\mathfrak{g}$. Then the orthogonal decomposition $\mathfrak{g}=\mathfrak{h}+\mathfrak{m}$ with respect to this inner product is reductive. Generally speaking, the reductive decomposition for a homogeneous Finsler manifold $(G/H,F)$ is not unique. The assumption that $G$ is closed in $I(M,F)$ is not essential, but it helps us skip some minor technical chores without loss of many generalities.

On a smooth coset space $G/H$ with a reductive decomposition $\mathfrak{g}=\mathfrak{h}+\mathfrak{m}$, the homogeneous Finsler metric $F$ is one-to-one determined by
its restriction to $T_o(G/H)$, which is an arbitrary $\mathrm{Ad}(H)$-invariant Minkowski norm on $\mathfrak{m}$. For simplicity, we use the same $F$ to denote this Minkowski norm.

In \cite{Hu2015-1}, L. Huang defined the spray vector field
$\eta:\mathfrak{m}\backslash\{0\}\rightarrow\mathfrak{m}$ and the connection operator $N:\mathfrak{m}\backslash{0}\times\mathfrak{m}\rightarrow
\mathfrak{m}$ for a homogeneous Finsler manifold $(G/H,F)$
with a reductive decomposition $\mathfrak{g}=\mathfrak{h}+\mathfrak{m}$ by the following equalities,
\begin{eqnarray*}
g_y(\eta(y),u)&=&g_y(y,[u,y]_\mathfrak{m}),\quad\forall y\in\mathfrak{m}\backslash\{0\},\ u\in\mathfrak{m},\quad\mbox{and}\\
2g_y(N(y,v),u)&=& g_y([u,v]_\mathfrak{m},y)+
g_y([u,y]_\mathfrak{m},v)+g_y([v,y]_\mathfrak{m},y)\nonumber
\\ & &
-2\mathbf{C}_y(u,v,\eta(y)),\quad\forall y\in\mathfrak{m}\backslash\{0\},\ u,v\in\mathfrak{m}.
\end{eqnarray*}
Using these notions, L. Huang presented beautiful homogeneous curvature formulae. For example, his homogeneous S-curvature, Landsberg curvature and Riemann curvatures are the following (see Proposition 4.6, Theorem 4.8 in \cite{Hu2015-1} and Remark \ref{remark-6-10}),
\begin{eqnarray*}
\mathbf{S}(o,y)&=&\mathrm{Tr}_\mathbb{R}(N(y,\cdot)+
\mathrm{ad}_\mathfrak{m}(y)),\\
\mathbf{L}_y(v,v,v)&=&3\mathbf{C}_y(v,v,[v,y]_\mathfrak{m}
-N(y,v))-\mathbf{C}_y(v,v,v,\eta(y)),\\
\mathbf{R}_y(v)&=&[y,[v,y]_\mathfrak{h}]_\mathfrak{m}
+DN(\eta,y,v)-N(y,N(y,v))+N(y,[y,v]_\mathfrak{m})-
[y,N(y,v)]_\mathfrak{m}.
\end{eqnarray*}

At the end of this subsection, we remark that a left invariant
spray structure $\mathbf{G}$ or a left invariant Finsler metric $F$ on a Lie group $G$ is automatically homogeneous, with
$G$ identified with $G/H=G/\{e\}$, which is equipped with the unique decomposition $\mathfrak{g}=\mathfrak{h}+\mathfrak{m}=0+\mathfrak{g}$.
This decomposition is obviously reductive.
See \cite{Hu2015-2} for the curvature properties of a left invariant Finsler metric on a Lie group.
\subsection{Spray vector field and connection operator}
\label{subsection-5-2}

The following lemma is crucial for later discussion.

\begin{lemma} \label{lemma-5-1}
Let $G/H$ be a smooth coset space with a linear decomposition  $\mathfrak{g}=\mathfrak{h}+\mathfrak{m}$. Suppose that $c(t)$ is a smooth curve on $G/H$
which is defined for $t\in(a,b)$ with $a<0<b$, satisfying $c(0)=g\cdot o$ for some $g\in G$ and $\dot{c}(t)\neq0$ everywhere. Then
there exists a unique smooth curve
$\overline{c}(t)$ on $G$, which is defined around  $t=0$, and satisfies
\begin{equation}\label{054}
\overline{c}(0)=g,\quad
c(t)=\overline{c}(t)\cdot o \quad\mbox{and}\quad
(L_{\overline{c}(t)^{-1}})_*
(\dot{\overline{c}}(t))=(\overline{c}(t)^{-1})_*(\dot{c}(t))
\in\mathfrak{m}\backslash0\quad\mbox{everywhere.}
\end{equation}
Further more, $\overline{c}(t)$
can also be defined for $t\in(a,b)$ if one of the following is satisfied:
\begin{enumerate}
\item the decomposition $\mathfrak{g}=\mathfrak{h}+\mathfrak{m}$ is reductive;
\item the subgroup $H$ is compact.
\end{enumerate}
\end{lemma}

The equality $(L_{\overline{c}(t)^{-1}})_*
(\dot{\overline{c}}(t))=(\overline{c}(t)^{-1})_*(\dot{c}(t))$
in (\ref{054}) is because our convention identifies $T_o(G/H)$ with $\mathfrak{m}$, and then identifies $T_{c(t)}(G/H)=
(\overline{c}(t))_*(T_{o}(G/H))$ for
$c(t)=\overline{c}(t)\cdot o$ with $(L_{\overline{c}(t)})_*(\mathfrak{m})\subset T_{\overline{c}(t)}G$ naturally. See also Fact 2 in Section \ref{subsection-6-1}.\medskip

\begin{proof}
Firstly, we prove the unique existence of $\overline{c}(t)$ for $t$ close to $0$.

There exists a smooth curve $g(t)$ with $t\in(a,b)$ on $G$, satisfying $g(0)=g$ and
$c(t)=g(t)\cdot o$ everywhere. Then the curve $\overline{c}(t)$ indicated in Lemma \ref{lemma-5-1}, if it exists, must be of the form $\overline{c}=g(t)h(t)$,
in which $h(t)$ is a smooth curve in $H$ with $h(0)=e$, for the first two requirements in (\ref{054}), i.e., $\overline{c}(0)=g$ and $c(t)=\overline{c}(t)\cdot o$, to be satisfied.
By the calculation
\begin{eqnarray*}
(L_{\overline{c}(t)^{-1}})_*(\dot{\overline{c}}(t))&=&
(L_{\overline{c}(t)^{-1}})_*
((L_{g(t)})_*(\dot{h}(t))
+(R_{h(t)})_*(\dot{g}(t)))
\\
&=&(L_{\overline{c}(t)^{-1}})_*
((L_{g(t)})_*(\dot{h}(t))
+(L_{\overline{c}(t)^{-1}})_*((R_{h(t)})_*(\dot{g}(t)))\\
&=&(L_{h(t)^{-1}})_*(\dot{h}(t))+\mathrm{Ad}(h(t)^{-1})
((L_{g(t)^{-1}})_*(\dot{g}(t)),
\end{eqnarray*}
we see that the third condition in (\ref{054}), i.e., $(L_{\overline{c}(t)^{-1}})_*
(\dot{\overline{c}}(t))\in\mathfrak{m}\backslash\{0\}$,
is satisfied if and only if $h(t)\in H$ is a solution of the ODE
\begin{equation}\label{053}
\dot{h}(t)+
(L_{h(t)})_*((\mathrm{Ad}(h(t)^{-1})
((L_{g(t)^{-1}})_*(\dot{g}(t))))_\mathfrak{h})
=0.
\end{equation}
The solution $h(t)$ of (\ref{053}) satisfying $h(0)=e$ exists uniquely for $t$ sufficiently close to 0. So the smooth curve $\overline{c}(t)$ on $G$ which satisfies (\ref{054}) exists uniquely around $t=0$.

Nextly, we prove the existence of $\overline{c}(t)$ for $t\in(a,b)$ when the decomposition $\mathfrak{g}=\mathfrak{h}+\mathfrak{m}$ is reductive.
In this case, (\ref{053}) can be simplified as
$\dot{h}(t)=(R_{h(t)})_*(y(t))$, in which
$y(t)=((L_{g(t)^{-1}})_*(\dot{g}(t)))_\mathfrak{h}$ for $t\in(a,b)$ is a given smooth curve in $\mathfrak{h}$. The existence of the solution $h(t)$ for $t\in(a,b)$, satisfying $h(0)=e$, can be similarly proved by the argument proving  Theorem D in \cite{Xu2021-1}. Then the existence of $\overline{c}(t)$ for $t\in(a,b)$ follows immediately.

Finally, we prove the the existence of $\overline{c}(t)$ for $t\in(a,b)$ when $H$ is compact. Let $\epsilon$ be a positive number, which is arbitrarily close to 0. The compactness implies the existence of $\delta>0$, such that for any $t_0\in[a+\epsilon,b-\epsilon]$ and any $h\in H$, there exists a unique solution $h(t)$ of (\ref{053}) for $t\in(t_0-\delta,t_0+\delta)$ satisfying $h(t_0)=h$. Then the solution $h(t)$ of (\ref{053}) with $h(0)=e$ can be extended to $[a+\epsilon,b-\epsilon]$ for any sufficiently small  $\epsilon>0$, by gluing it with other solutions of (\ref{053}).
So we have the existence of $h(t)$, as well as the existence of $\overline{c}(t)$, for $t\in(a,b)$.
\end{proof}\medskip

Let $(G/H,\mathbf{G})$ be a homogeneous spray manifold
with a linear decomposition $\mathfrak{g}=\mathfrak{h}+\mathfrak{m}$.
Suppose that $c_y(t)$ is the geodesic on $(G/H,\mathbf{G})$ with $c(0)=o$ and $\dot{c}(0)=y\in T_{o}(G/H)\backslash\{0\}=\mathfrak{m}\backslash\{0\}$. Denote $\overline{c}_y(t)$ the smooth curve on $G$ provided by Lemma \ref{lemma-5-1}, satisfying $\overline{c}_y(0)=e$,
$c_y(t)=\overline{c}_y(t)\cdot o$ and $y(t)=(L_{\overline{c}_y(t)^{-1}})_*
(\overline{c}_y(t))\in\mathfrak{m}\backslash\{0\}$
for $t$ close to 0. Since $y(t)$ is a smooth curve in $\mathfrak{m}\backslash\{0\}$,
we can define the {\it spray vector field} $\eta:\mathfrak{m}\backslash\{0\}\rightarrow\mathfrak{m}$
by $\eta(y)=-\tfrac{{\rm d}}{{\rm d}t}|_{t=0}y(t)$.

\begin{lemma} \label{lemma-5-2}
The spray vector field $\eta$ is a positively 2-homogeneous smooth map.
\end{lemma}
\begin{proof}
By the property of exponential map, wherever $t$ is close to $0$ and $y\neq0$, $c(t,y)=c_y(t)$ is smooth for both $t$ and $y$. By the theory of ODE, $\overline{c}(t,y)=\overline{c}_y(t)$ is also smooth for both $t$ and $y$. So $\eta(y)=-\tfrac{{\rm d}}{{\rm d}t}|_{t=0}(L_{\overline{c}_y(t)^{-1}})_*
(\dot{\overline{c}}_y(t))$ depends smoothly on $y\in\mathfrak{m}\backslash\{0\}$.

For any constant $\lambda>0$, $c_{\lambda y}(t)=c(\lambda t)$ and $\dot{c}_{\lambda y}(t)=\lambda
\dot{c}(\lambda t)$ when $t$ is close to $0$. So we have
\begin{eqnarray*}
\eta(\lambda y)=-\tfrac{{\rm d}}{{\rm d}t}|_{t=0}(L_{\overline{c}_{\lambda y}(t)^{-1}})_*
(\dot{\overline{c}}_{\lambda y}(t))=-\lambda\tfrac{{\rm d}}{{\rm d}t}|_{t=0}
(L_{\overline{c}_y(\lambda t)})_*(\dot{\overline{c}}_y(\lambda t))
=\lambda^2\eta(y),
\end{eqnarray*}
which proves the positive 2-homogeneity of $\eta$.
\end{proof}\medskip

The smoothness of the spray vector field $\eta$
enable us to define the {\it connection operator} $N:\mathfrak{m}\backslash\{0\}\times \mathfrak{m}\rightarrow \mathfrak{m}$ by
$N(y,v)=\tfrac12D\eta(y,v)-\tfrac12[y,v]_\mathfrak{m}$, in which $D\eta(y,v)=\tfrac{{\rm d}}{{\rm d}t}|_{t=0}\eta(y+tv)$.

When we view $\eta$ as a smooth tangent vector field on the manifold $\mathfrak{m}\backslash\{0\}$, we can get the following correspondence between geodesics on $(G/H,\mathbf{G})$ and integral curves of $-\eta$ on $\mathfrak{m}\backslash\{0\}$.

\begin{theorem}\label{theorem-5-3}
Let $(G/H,\mathbf{G})$ be a homogeneous spray manifold with a linear decomposition $\mathfrak{g}=
\mathfrak{h}+\mathfrak{m}$ and $\eta$ the spray vector field. Then we have a one-to-one correspondence between the following two sets:
\begin{enumerate}
\item the set of all geodesics $c(t)$ on $(G/H,\mathbf{G})$, which is defined for $t$ sufficiently close to 0 and satisfies $c(0)=o$;
\item the set of all integral curves $y(t)$ of $-\eta$ on $\mathfrak{m}\backslash\{0\}$, which is defined for $t$ sufficiently close 0.
\end{enumerate}
The correspondence is from $c(t)$ to $y(t)=
(L_{\overline{c}(t)^{-1}})_*(\dot{\overline{c}}(t))$, in which $\overline{c}(t)$ is the smooth curve on $G$ satisfying $\overline{c}(0)=e$, $c(t)=\overline{c}(t)\cdot o$ and $(L_{\overline{c}(t)^{-1}})_*(\dot{\overline{c}}(t))
\in\mathfrak{m}\backslash\{0\}$ for all possible values of $t$.
The range for the parameter $t$ in this correspondence can be changed to an arbitrary interval $(a,b)$ with $a<0<b$ when $\mathfrak{g}=\mathfrak{h}+\mathfrak{m}$ is reductive
or $H$ is compact.
\end{theorem}

\begin{proof}
Firstly, we prove the correspondence from (1) to (2) around $t=0$.

We consider a geodesic
$c(t)$ with $c(0)=o$ on $(G/H,\mathbf{G})$. Let $\overline{c}(t)$ be the smooth curve on $G$ provided by Lemma \ref{lemma-5-1} for $c(t)$, satisfying $\overline{c}(0)=e$,
$c(t)=\overline{c}(t)\cdot o$ and $(L_{\overline{c}(t)^{-1}})_*(\dot{\overline{c}}(t))
\in\mathfrak{m}\backslash\{0\}$, for all possible values of $t$. Then
$y(t)=(L_{\overline{c}(t)^{-1}})_*(\dot{\overline{c}}(t))$
is a smooth curve in $\mathfrak{m}\backslash\{0\}$ defined around $t=0$.

Denote $c_{t_0}(t)=\overline{c}(t_0)^{-1}\cdot c(t+t_0)$ for any $t_0$ where $\overline{c}(t)$ is defined. By the left invariancy of $\mathbf{G}$, $c_{t_0}(t)$ is a geodesic on $(G/H,\mathbf{G})$ satisfying $c_{t_0}(0)=o$. Meanwhile,
$\overline{c}_{t_0}(t)=\overline{c}(t_0)^{-1}\cdot
\overline{c}(t+t_0)$
satisfies $$\overline{c}_{t_0}(0)=e,
\quad{c}_{t_0}(t)=\overline{c}_{t_0}(t)
\cdot o\quad\mbox{and}\quad
(L_{\overline{c}_{t_0}(t)^{-1}})_*(\dot{\overline{c}}_{t_0}(t))
\in\mathfrak{m}$$
for $t$ close to 0, i.e., it is the smooth curve on $G$ provided by Lemma \ref{lemma-5-1} for $c_{t_0}(t)$. It is easy to see that $$y_{t_0}(t)=(L_{\overline{c}_{t_0}(t)^{-1}})_*(\dot{\overline{c}}_{t_0}(t))
=(L_{\overline{c}(t+t_0)^{-1}})_*(\dot{\overline{c}}(t+t_0))
=y(t+t_0),
$$
and in particular $y_{t_0}(0)=y(t_0)$.
So by definition, we get
\begin{eqnarray*}
\eta(y(t_0))=-\tfrac{{\rm d}}{{\rm d}t}|_{t=0}y_{t_0}(t)
=-\tfrac{{\rm d}}{{\rm d}t}|_{t=0}y(t+t_0)=-\tfrac{{\rm d}}{{\rm d}t}{y}(t_0)
\end{eqnarray*}
for each $t_0$, i.e., $y(t)$ is an integral curve of $-\eta$.
The correspondence from (1) to (2) is proved.

Nextly, we prove the backward correspondence around $t=0$.

We consider an integral curve $y(t)$ of $-\eta$, which is defined for $t$ close to 0. Then we have a smooth curve $\overline{c}'(t)$ on $G$ satisfying
\begin{equation}\label{056}
\overline{c}'(0)=e\quad\mbox{and}\quad
\dot{\overline{c}'}(t)=(L_{\overline{c}'(t)})_*(y(t))
\end{equation}
which is defined
around $t=0$.
Let $c(t)$ be the geodesic on $(G/H,\mathbf{G})$ satisfying
$c(0)=o$ and
$\dot{c}(0)=y(0)\in \mathfrak{m}\backslash\{0\}$, and
$\overline{c}(t)$ the smooth curve on $G$ provided by
Lemma \ref{lemma-5-1} for $c(t)$, satisfying
$\overline{c}(0)=e$, $c(t)=\overline{c}(t)\cdot o$ and
$(L_{\overline{c}(t)^{-1}})_*(\dot{\overline{c}}(t))\in\mathfrak{m}$
around $t=0$. We have proved the correspondence from (1) to (2), which guarantees that $\overline{c}(t)$ is also a solution of (\ref{056}).
By the uniqueness for the solution of (\ref{056}), $\overline{c}(t)$ coincides with $\overline{c}'(t)$ for $t$ around $0$, i.e., $c(t)=\overline{c}(t)\cdot o=\overline{c}'(t)\cdot o=c'(t)$ is a geodesic for $t$ around $0$.
The correspondence from (2) to (1) is proved.

Finally, we prove for any interval $(a,b)$ with $a<0<b$, the one-to-one correspondence between
\begin{enumerate}
\item the set of geodesics $c(t)$ with $t\in(a,b)$ on $(G/H,\mathbf{G})$, satisfying $c(0)=o$, and
\item the set of integral curve $y(t)$ with $t\in(a,b)$ of $-\eta$ on $\mathfrak{m}\backslash\{0\}$,
\end{enumerate}
when the decomposition $\mathfrak{g}=\mathfrak{h}+\mathfrak{m}$ is reductive
or $H$ is compact.

From (1) to (2), when the geodesic $c(t)$ is defined for $t\in(a,b)$, the smooth curve $\overline{c}(t)$ on $G$, provided by
Lemma \ref{lemma-5-1} for $c(t)$, satisfying $\overline{c}(0)=e$, $c(t)=\overline{c}(t)\cdot o$ and
$y(t)=(L_{\overline{c}(t)^{-1}})_*(\dot{\overline{c}}(t))
\in\mathfrak{m}$, is also defined for $t\in(a,b)$. So the previous argument shows that $y(t)$ is an integral curve of $-\eta$ on $\mathfrak{m}\backslash\{0\}$, which is defined for $t\in(a,b)$.

From (2) to (1), for the integral curve $y(t)$ of $-\eta$, which is defined for $t\in(a,b)$,
the smooth curve $\overline{c}'(t)$ satisfying (\ref{056}) is also defined for $t\in(a,b)$. The argument is similar to that proving Theorem D in \cite{Xu2021-1}.
Denote $c(t)=\overline{c}'(t)\cdot o$. For each $t_0\in(a,b)$,
similar argument as above can show that
$\overline{c}(t_0)^{-1}\cdot c(t+t_0)
=\overline{c}(t_0)^{-1}\overline{c}(t)\cdot o$ is a geodesic around $t=0$. Using the $G$-invariancy of $\mathbf{G}$, $c(t)$ for $t\in(a,b)$ is a geodesic on $(G/H,\mathbf{G})$.

To summarize, the range of $t$ in the correspondence in Theorem \ref{theorem-5-3} can be an arbitrary interval $(a,b)$ with $a<0<b$. This ends the proof.
\end{proof}\medskip

Comparing Theorem \ref{theorem-5-3} with Theorem \ref{theorem-4-3}, we see immediately that the definitions of spray vector field and connection operator in Section \ref{subsection-5-2} are compatible with
those in Section \ref{subsection-4-2} or \cite{Xu2021-1}, for a left invariant spray structure.
They are also compatible with those defined by L. Huang \cite{Hu2015-1}
(see (\ref{301}) and (\ref{302})), which will be proved  in Section \ref{subsection-5-3} (see Theorem \ref{theorem-5-7}).



\subsection{Construction of the submersion for a homogeneous spray structure}
\label{subsection-5-3}

Let $\eta:\mathfrak{m}\backslash\{0\}\rightarrow\mathfrak{m}$ be the spray vector field of the homogeneous spray manifold $(G/H,\mathbf{G})$ with a linear decomposition $\mathfrak{g}=\mathfrak{h}+\mathfrak{m}$. Using some cut-off function technique, we can extend $\eta$
to a positively 2-homogeneous  smooth map $\overline{\eta}:\mathfrak{g}\backslash\{0\}
\rightarrow\mathfrak{g}$. Then there exists a
left invariant spray structure $\overline{\mathbf{G}}$ on $G$, such that
$\overline{\eta}$ is its spray vector field.

The following theorem claims a submersion between $(G,\overline{\mathbf{G}})$ and $(G/H,\mathbf{G})$.

\begin{theorem}\label{theorem-5-6}
Let $(G/H,\mathbf{G})$ be a homogeneous spray manifold with a linear decomposition $\mathfrak{g}=\mathfrak{h}+\mathfrak{m}$ and $\eta:\mathfrak{m}\backslash\{0\}\rightarrow\mathfrak{m}$ the  spray vector field.
Let $\overline{\mathbf{G}}$ be a left invariant spray structure on $G$ such that its spray vector field $\overline{\eta}:\mathfrak{g}\backslash\{0\}
\rightarrow\mathfrak{g}$ satisfies $\eta=\overline{\eta}|_{\mathfrak{m}\backslash\{0\}}$. Denote $\pi: G\rightarrow G/H$ the smooth map $\pi(g)=g\cdot o$ for all $ g\in G$, and $\mathcal{L}=\cup_{g\in G}(L_g)_*(\mathfrak{m})$ a distribution on $G$.
Then
$(\pi,\mathcal{L})$ is a submersion from $(G,\overline{\mathbf{G}})$ to $(G/H,\mathbf{G})$.
\end{theorem}

\begin{proof}
Let $\overline{c}(t)$ be any geodesic on
$(G,\overline{\mathbf{G}})$, satisfying
 $\dot{\overline{c}}(0)\in\mathcal{L}_{\overline{c}(0)}=
(L_{\overline{c}(0)})_*(\mathfrak{m})$. By the left invariancy of $\mathbf{G}$, Theorem \ref{theorem-4-3} implies that $y(t)=(L_{\overline{c}(t)^{-1}})_*(\dot{\overline{c}}(t))$
is an integral curve of $-\overline{\eta}$. Since the smooth tangent vector field $\overline{\eta}$ is tangent to $\mathfrak{m}\backslash\{0\}$ and $y(0)=(L_{c(0)^{-1}})_*(\dot{\overline{c}}(0))
\in\mathfrak{m}\backslash\{0\}$,
we have $y(t)\in\mathfrak{m}\backslash\{0\}$, i.e.,
$\dot{\overline{c}}(t)\in (L_{\overline{c}(t)})_*(\mathfrak{m}\backslash\{0\})$, $\forall t$.
This argument proves that $\overline{\mathbf{G}}$
is tangent to $\mathcal{L}\subset T(TG\backslash0)$.

Meanwhile, we see that $y(t)$ is also an integral curve of $-\eta$. While proving Theorem \ref{theorem-5-3}, we have showed that  $c(t)=\overline{c}(t)\cdot o=\pi(\overline{c}(t))$ is a geodesic
on $(G/H,\mathbf{G})$.

To summarize, we see by definition that $(\pi,\mathcal{L})$ is a submersion between the spray structures $\overline{\mathbf{G}}$ and $\mathbf{G}$. This ends the proof.
\end{proof}\medskip

Theorem \ref{theorem-5-6} permits the submersion technique in homogeneous spray geometry.
Practically,
we  use a special $\overline{\mathbf{G}}$ in Theorem \ref{theorem-5-6} for the convenience in calculation.
For example, in the proof of Theorem \ref{theorem-5-7} below,
we choose the left invariant spray structure $\overline{\mathbf{G}}$ induced by a canonical construction for Minkowski norms.
In Section 6, we choose another $\overline{\mathbf{G}}$  such that
its spray vector field $\overline{\eta}$ satisfies $\overline{\eta}(y)=\eta(y_\mathfrak{m})$ in a conic open neighborhood of $\mathfrak{m}\backslash\{0\}$. Then a lot of calculation can be simplified.

In the rest of this section, we use the submersion technique to prove two theorems.

One theorem is for an explicit correspondence between a homogenous spray structure and its spray vector field, with respect to a reductive
decomposition $\mathfrak{g}=\mathfrak{h}+\mathfrak{m}$ for $G/H$.


\begin{theorem}\label{theorem-5-5}
For a smooth coset space $G/H$ with
a reductive decomposition $\mathfrak{g}=\mathfrak{h}+\mathfrak{m}$, there is a one-to-one correspondence between the following two sets:
\begin{enumerate}
\item the set of all $G$-invariant spray structures $\mathbf{G}$ on $G/H$;
\item the set of all $\mathrm{Ad}(H)$-invariant smooth maps
$\eta:\mathfrak{m}\backslash\{0\}\rightarrow\mathfrak{m}$.
\end{enumerate}
The correspondence is from $\mathbf{G}$ to its spray vector field with respect to the given decomposition.
\end{theorem}

\begin{proof}
Firstly, we consider the correspondence from (1) to (2).  We only need to prove for any homogeneous spray structure $\mathbf{G}$ on $G/H$, its spray vector field $\eta:\mathfrak{m}\backslash\{0\}\rightarrow\mathfrak{m}$ is $\mathrm{Ad}(H)$-invariancy, i.e., $\eta(\mathrm{Ad}(g)y)=\mathrm{Ad}(g)\eta(y)$, $\forall g\in H$, $y\in\mathfrak{m}\backslash\{0\}$.

Let $c(t)$ be the geodesic on $(G/H,\mathbf{G})$ with $c(0)=o$ and $\dot{c}(0)=y\in\mathfrak{m}\backslash\{0\}$.
Denote $\overline{c}(t)$ the smooth curve on $G$ provided by Lemma \ref{lemma-5-1} for $c(t)$, satisfying $\overline{c}(0)=e$, $c(t)=\overline{c}(t)\cdot o$ and
$(L_{\overline{c}(t)^{-1}})_*
(\dot{\overline{c}}(t))\in\mathfrak{m}\backslash\{0\}$ for all $t$ close to 0. Then for any $g\in H$, $c'(t)=g\cdot c(t)$ is also a geodesic with $c(0)=o$ and it satisfies $\dot{c}'(0)=\mathrm{Ad}(g)(y)$. More over, $\overline{c}'(t)=
g\overline{c}(t)g^{-1}$ is the smooth curve satisfying
$\overline{c}'(0)=e$, $c'(t)=\overline{c}'(t)\cdot o$ and
$(L_{c'(t)^{-1}})_*(\dot{c}'(t))=
\mathrm{Ad}(g)((L_{c(t)^{-1}})_*(\dot{c}(t)))\in
\mathfrak{m}\backslash\{0\}$
(the reductive property is needed here) for $t$ close to 0. So we have by definition
\begin{eqnarray*}
\eta(\mathrm{Ad}(g)y)=-\tfrac{{\rm d}}{{\rm d}t}|_{t=0}
((L_{c'(t)^{-1}})_*(\dot{c}'(t)))=
-\tfrac{{\rm d}}{{\rm d}t}|_{t=0}(\mathrm{Ad}(g)
((L_{c(t)^{-1}})_*(\dot{c}(t))))=\mathrm{Ad}(g)(\eta(y)),
\end{eqnarray*}
which proves the $\mathrm{Ad}(H)$-invariancy for $\eta$.

Nextly, we sketch the correspondence from (2) to (1) with some tedious chores skipped.

Let $\eta$ be any $\mathrm{Ad}(H)$-invariant smooth map from $\mathfrak{m}\backslash\{0\}\rightarrow\mathfrak{m}$.
We extend it to a positively 2-homogeneous smooth map $\overline{\eta}:\mathfrak{g}\backslash\{0\}
\rightarrow\mathfrak{g}$. Then there exits a left invariant
spray structure $\overline{\mathbf{G}}$ on $G$, such that
$\overline{\eta}$ is the spray structure of $\overline{\mathbf{G}}$. Theorem \ref{theorem-4-3} implies that $\overline{\mathbf{G}}$ is tangent to $\mathcal{L}=\cup_{g\in G}(L_g)_*(\mathfrak{m})$.
Since the decomposition $\mathfrak{g}=\mathfrak{h}+\mathfrak{m}$ is reductive and $\eta$ is $\mathrm{Ad}(H)$-invariant, both $\mathcal{L}$ and
$\overline{\mathbf{G}}|_{\mathcal{L}\backslash0}$ are right $H$-invariant. Then
$\mathbf{G}=(\pi_*)_*(\overline{\mathbf{G}}
|_{\mathcal{L}\backslash0})$ is a well defined smooth tangent vector field on $T(G/H)\backslash0$. More over, $\mathbf{G}$
is a spray structure on $G/H$, and the pair $(\pi,\mathcal{L})$ is a submersion from
$(G,\overline{\mathbf{G}})$ to $(G/H,\mathbf{G})$. Applying Theorem \ref{theorem-4-3} and Theorem \ref{theorem-5-3} to any geodesic $\overline{c}(t)$ on $(G,\overline{\mathbf{G}})$ tangent to $\mathcal{L}$ and the geodesic
$c(t)=\overline{c}(t)\cdot o$ on $(G/H,\mathbf{G})$,
we see immediately that the spray vector field of $\mathbf{G}$ coincides with $\overline{\eta}|_{\mathfrak{m}\backslash\{0\}}=\eta$.

To summarize, this argument provides the correspondence from (2) to (1). Then the proof ends.
%
\end{proof}\medskip

Denote $\mathbf{G}_0$ the homogeneous spray structure
on $G/H$ with a reductive decomposition $\mathfrak{g}=\mathfrak{h}+\mathfrak{m}$, such that its spray vector field $\eta$ is constantly 0. Then the pair $\pi(g)=g\cdot o$ and $\mathcal{L}=\cup_{g\in G}(L_g)_*(\mathfrak{m})$ is a submersion from the canonical
bi-invariant spray structure $\overline{\mathbf{G}}_0$ (see Theorem \ref{theorem-4-2}) on $G$
to $\mathbf{G}_0=(\pi_*)_*(\overline{\mathbf{G}}_0
|_{\mathcal{L}\backslash0})$. Notice that the geodesics for $\mathbf{G}_0$ coincide with those for the Nomizu connection (see Corollary 2.5 in Chapter 10 of \cite{KN1969}), so we may call $\mathbf{G}_0$ the {\it canonical spray structure} or the {\it spray structure for the Nomizu connection}. However, it should be notified that the covariant derivative of $\mathbf{G}_0$ corresponds to a torsion free connection (for example, the Levi-Civita connection in Riemannian geometry), which is not the Nomizu connection (see Theorem 2.6 in Chapter 10 of \cite{KN1969}).

For any homogeneous spray structure $\mathbf{G}$
on $G/H$ with a reductive decomposition,
we can present the corresponding $\overline{\mathbf{G}}$ in the proof of Theorem \ref{theorem-5-5} as $\overline{\mathbf{G}}=\overline{\mathbf{G}}_0-
\overline{\mathbf{H}}$. Then $\mathbf{H}=(\pi_*)_*(\overline{\mathbf{H}})$ is a well defined $G$-invariant smooth tangent vector fields on $T(G/H)\backslash0$, which is tangent to each $T_x(G/H)\backslash\{0\}$ and $\mathbf{H}|_{T_o(G/H)}=\eta$.
To summarize, we have $$\mathbf{G}=(\pi_*)_*(\overline{\mathbf{G}})=\mathbf{G}_0-
\mathbf{H},$$
which generalizes the representation for a left invariant spray structure in \cite{Xu2021-1,Xu2021-2}.

It should be notified that when the decomposition $\mathfrak{g}=\mathfrak{h}+\mathfrak{m}$ is not reductive, the relation between $\mathbf{G}$ and $\eta$, and the calculation for $\eta$ may be much more complicated.

The other theorem is for the compatibility between the definitions of spray vector field and connection operator in homogeneous spray geometry and those in homogeneous Finsler geometry (see \cite{Hu2015-1} or (\ref{301}) and (\ref{302}) in Section \ref{subsection-5-1}).

\begin{theorem}\label{theorem-5-7}
Let $F$ be a homogeneous Finsler metric on $M=G/H$,
in which $G$ is a closed subgroup of $I(M,F)$,
and $\mathfrak{g}=\mathfrak{h}+\mathfrak{m}$ a reductive decomposition. Then the spray vector field $\eta$ and the connection operator $N$ for the geodesic spray $\mathbf{G}$ of $F$ satisfies
\begin{eqnarray}
g_y(\eta(y),u)&=&g_y(y,[u,y]_\mathfrak{m}),\quad\forall y\in\mathfrak{m}\backslash\{0\},\ u\in\mathfrak{m},\quad\mbox{and}\label{301}\\
2g_y(N(y,v),u)&=& g_y([u,v]_\mathfrak{m},y)+
g_y([u,y]_\mathfrak{m},v)+g_y([v,y]_\mathfrak{m},y)\nonumber
\\ & &
-2\mathbf{C}_y(u,v,\eta(y)),\quad\forall y\in\mathfrak{m}\backslash\{0\},\ u,v\in\mathfrak{m},
\label{302}
\end{eqnarray}
in which the inner product $g_y(\cdot,\cdot)$ and the Cartan tensor $\mathbf{C}_y(\cdot,\cdot,\cdot)$ are for the Minkowski norm $F=F(o,\cdot)$ on $\mathfrak{m}$.
\end{theorem}


\begin{proof} Firstly, we prove the statement for $\eta$.

Since $G$ is a closed subgroup of $I(M,F)$, its isotropy subgroup $H$ is compact. So we can find an $\mathrm{Ad}(H)$-invariant inner product $\langle\cdot,\cdot\rangle_{\mathrm{bi}}$ on $\mathfrak{h}$.
Using similar argument as in the proof of Lemma 3.4 in \cite{XDHH}, we can construct an $\mathrm{Ad}(H)$-invariant
Minkowski norm $\overline{F}$ on $\mathfrak{g}$, such that
\begin{equation}\label{065}
\overline{F}(y)=\sqrt{\langle y_\mathfrak{h},y_\mathfrak{h}
\rangle_{\mathrm{bi}}+F(y_\mathfrak{m})^2}
\end{equation}
in a conic open neighborhood of $\mathfrak{m}\backslash\{0\}$ in $\mathfrak{g}\backslash\{0\}$. Using left translations, $\overline{F}$ induces a Finsler metric on $G$, which is both left $G$-invariant and right $H$-invariant. For simplicity, we denote this metric as the same $\overline{F}$.
Then Lemma 3.3 and Lemma 3.4 in \cite{XDHH} implies that $\pi:(G,\overline{F})\rightarrow(G/H,F)$ is a Finsler submersion, which horizonal bundle coincides with  $\mathcal{L}=\cup_{g\in G} (L_{g})_*(\mathfrak{m})$.
By the geodesic property of Finsler submersion (see Theorem 3.1 in \cite{PD2001}), $(\pi,\mathcal{L})$ is a submersion between the geodesic spray $\overline{\mathbf{G}}$ of $\overline{F}$ and the geodesic spray $\mathbf{G}$ of $F$.

Theorem 3.1 in \cite{XD2015} provides the following representation for $\overline{\mathbf{G}}$,
\begin{equation}\label{062}
\overline{\mathbf{G}}=v^i\widetilde{V}_i-\tfrac12 g^{il}c^k_{lj}
[F^2]_{v^k}v^j\partial_{v^i}.
\end{equation}
Notice that the right invariant frame $\{V_i,\forall i\}$ on $(G,\overline{F})$ is a Killing frame \cite{XD2014}, and $[V_i,V_j]=-c_{ij}^k V_k$ (see the first line in (\ref{063})) results in a sign difference. By Theorem \ref{theorem-4-2},
\begin{equation}\label{064}
\tfrac12 g^{il}c^k_{lj}
[F^2]_{v^k}v^j\partial_{v^i}=\overline{\mathbf{G}}_0-
\overline{\mathbf{G}}
\end{equation}
is left invariant, in which
$\overline{\mathbf{G}}_0=u^i\widetilde{U}_i=v^i\widetilde{V}_i$ is the canonical affine bi-invariant spray structure on $G$.
When restricted to $T_eG\backslash\{0\}$, $\{u^i,\partial_{u^i},\forall i\}$ coincides with $\{v^i,\partial_{v^i},\forall i\}$, so
the restriction of (\ref{064}) to $T_eG\backslash\{0\}$, i.e.,
the spray vector field $\overline{\eta}:\mathfrak{g}
\backslash\{0\}\rightarrow\mathfrak{g}$ for $\overline{\mathbf{G}}$
satisfies $$\overline{g}_{y}(\overline{\eta}(y),u)=
\overline{g}_{y}(y,[u,y]_\mathfrak{g}),\quad\forall y\in\mathfrak{g}\backslash\{0\},\ u\in\mathfrak{g},$$
in which $\overline{g}_y(\cdot,\cdot)$ is for the Minkowski norm $\overline{F}$ on $\mathfrak{g}$.
Using (\ref{065}), we see that for each $y$ in a conic open neighborhood of $\mathfrak{m}\backslash\{0\}$ in $\mathfrak{g}\backslash\{0\}$ that
$$\overline{g}_y(u,v)=g_{y_\mathfrak{m}}
(u_\mathfrak{m},v_\mathfrak{m})+
\langle u_\mathfrak{h},v_\mathfrak{h}
\rangle_{\mathrm{bi}}.$$
By the reductive property of $\mathfrak{g}=\mathfrak{h}+\mathfrak{m}$ and the $\mathrm{Ad}(H)$-invariancy of the Minkowski norm $F$ on $\mathfrak{m}$, we have
$g_y(y,[u,y])=0$, $\forall y\in\mathfrak{m}\backslash\{0\}$, $u\in\mathfrak{h}$. Then
it is easy to check that
$\overline{\eta}$ maps $\mathfrak{m}\backslash\{0\}$ to $\mathfrak{m}$ and
\begin{equation}\label{066}
g_{y}(\overline{\eta}(y),u)=
g_{y}(y,[u,
y]_\mathfrak{m}),\quad\forall y\in\mathfrak{m}\backslash\{0\}, u\in\mathfrak{m}.
\end{equation}

Let $\eta:\mathfrak{m}\backslash\{0\}\rightarrow\mathfrak{m}$ the spray vector field for $\mathbf{G}$.
For any $y\in\mathfrak{m}\backslash\{0\}$, let $c(t)$ be the geodesic on $(G/H,F)$ which is defined around $t=0$ and satisfies $c(0)=o$ and $\dot{c}(0)=y$.
Let $\overline{c}(t)$ be the smooth curve on $G$ provided by Lemma \ref{lemma-5-1} for $c(t)$, satisfying $\overline{c}(0)=e$, $c(t)=\overline{c}(t)\cdot o$ and
$y(t)=(L_{\overline{c}(t)^{-1}})_*(\dot{\overline{c}}(t))\in \mathfrak{m}\backslash0$ everywhere. Since $\overline{c}(t)$ is a lifting of the geodesic $c(t)$ which is tangent to $\mathcal{L}$, we see that $\overline{c}(t)$ is a geodesic on $(G,\overline{\mathbf{G}})$.
Then by Theorem \ref{theorem-5-3} and Theorem \ref{theorem-4-3}, $y(t)$
is an integral curve of $-\eta$, as well as that of $-\overline{\eta}$. So we have $\eta(y(t))=\overline{\eta}(y(t))$, $\forall t$. In particular, when $t=0$, $\eta(y)=\overline{\eta}(y)$ for any $y\in\mathfrak{m}\backslash0$. By (\ref{066}), the statement for $\eta$ in Theorem \ref{theorem-5-7} is proved.

Nextly, we prove the statement for $N$. L. Huang pointed out
(see (4) in \cite{Hu2017}) that his definition (\ref{302}) for $N$ satisfies
$N(y,v)=\tfrac12D\eta(y,v)-\tfrac12[y,v]_\mathfrak{m}$, in
which $\eta$ is the one in (\ref{301}). So the statement for $N$
in Theorem \ref{theorem-5-7} follows after that for $\eta$ immediately.
\end{proof}

\begin{remark}\label{remark-5-8}
The assumption that the subgroup $G$ is closed in $I(M,F)$ is not necessary for Theorem \ref{theorem-5-7}. We add it to avoid some minor
chores.
\end{remark}

\section{Parallel translation and curvature for a homogeneous spray structure}

In this section, we use the submersion theory in Section 3 to study the geometry of a homogeneous spray structure.

\subsection{Some notations and facts}
\label{subsection-6-1}

The following
assumptions and notations are automatically applied through out this section.

Let $\overline{M}=G$ be an $m$-dimensional Lie group, $H\subset G$ an $(m-n)$-dimensional closed subgroup, and $M=G/H$ the $n$-dimensional smooth coset space. We choose
a linear decomposition $\mathfrak{g}=\mathfrak{h}+\mathfrak{m}$ for $G/H$,
and fix a basis $\{e_1,\cdots,e_m\}$ of $\mathfrak{g}$, such that $e_i\in\mathfrak{m}$ for $1\leq i\leq n$ and $e_\alpha\in\mathfrak{h}$ for $n+1\leq \alpha\leq m$. We apply the convention (\ref{067}) for indices, i.e. $1\leq i,j,k,l,p,q,r\leq n$ and $n+1\leq \alpha,\beta,\gamma\leq m$. So we have the Lie bracket coefficients in
$$[e_i,e_j]_\mathfrak{g}=c^k_{ij}e_k+c^\alpha_{ij}e_\alpha,\cdots,
[e_\alpha,e_\beta]_\mathfrak{g}=c_{\alpha\beta}^ie_i
+c_{\alpha\beta}^\gamma e_\gamma.$$
Notice that every $c_{\alpha\beta}^i$ vanishes because $\mathfrak{h}$ is a Lie subalgebra, and every $c_{\alpha i}^\beta=-c_{i\alpha}^\beta$ vanishes when the
chosen decomposition is reductive.
The corresponding left and right invariant frames are
$\{U_i,\forall i; U_\alpha,\forall\alpha\}$, $\{V_i,\forall i ;V_\alpha,\forall\alpha\}$ on $G$, in which $U_i(e)=V_i(e)=e_i$, $U_\alpha(e)=V_\alpha(e)=e_\alpha$, for each $i$ and $\alpha$, and
$\{\widetilde{U}_i,\partial_{u^i},\forall i;\widetilde{U}_\alpha,
,\partial_{u^\alpha},\forall \alpha\}$,
$\{\widetilde{V}_i,\partial_{u^i},\forall i;\widetilde{V}_\alpha,
\partial_{u^\alpha},\forall \alpha\}$ on $TG$, respectively.

Let $\mathbf{G}$ be a homogeneous spray structure on $G/H$,
with the spray vector field $\eta:\mathfrak{m}\backslash\{0\}
\rightarrow\mathfrak{m}$ and the connection operator
$N(y,w)=\tfrac12D\eta(y,w)-\tfrac12[y,w]_\mathfrak{m}$, $\forall y\in\mathfrak{m}\backslash\{0\}$.
Let $\overline{\mathbf{G}}$ be a left invariant spray structure on $G$, such that its spray vector field
$\overline{\eta}:\mathfrak{g}\backslash\{0\}
\rightarrow\mathfrak{g}$ satisfies
$\overline{\eta}(y)=\eta(y_\mathfrak{m})$  in a conic open neighborhood of $\mathfrak{m}\backslash\{0\}$ in
$\mathfrak{g}\backslash\{0\}$. Then Theorem \ref{theorem-5-6} indicates that the pair $(\pi,\mathcal{L})$, in which the smooth map $\pi:G\rightarrow G/H$ is  $\pi(g)=g\cdot o$ and the distribution $\mathcal{L}$ is
$\mathcal{L}=\cup_{g\in G}(L_g)_*(\mathfrak{m})$, is a submersion between $(G,\overline{\mathbf{G}})$ and $(G/H,\mathbf{G})$.

Here we summarize some obvious facts.\medskip

{\bf Fact 1}\quad{
The distribution $\mathcal{L}$ can be characterized by $u^\alpha=0$, $\forall\alpha$, i.e., any vector in $T_gG$ belongs to $\mathcal{L}_{g}$ if and only if it is of the form $u^iU_i(g)$}.
At each $g\in G$, the kernel of $\pi_*:T_gG\rightarrow T_{g\cdot o}(G/H)$ is $(L_g)_*(\mathfrak{h})$. So for
$\Phi^{T\overline{M}}_{\mathcal{L}}$  (see Section \ref{subsection-3-3} for its definition, same below for other bundle maps), we have
$$\Phi^{T\overline{M}}_{\mathcal{L}}((L_g)_*(y))=
(L_g)_*(y_\mathfrak{m}),\quad  \forall g\in G, y\in \mathfrak{g}=T_gG.$$

{\bf Fact 2}\quad Any vector $y\in \mathfrak{m}\subset \mathfrak{g}=T_eG$ is identified as a tangent vector in $T_o(G/H)$, which is its $\Phi^{\mathcal{L}}_{TM}$-image.
By the left invariancy of $\Phi^{\mathcal{L}}_{TM}$, { we  have $$\Phi^{\mathcal{L}}_{TM}((L_g)_*(v))=g_*(v),\quad\forall g\in G, v\in\mathfrak{m}.$$
Here $v$ in the left side is viewed as a vector in $\mathfrak{m}\subset\mathfrak{g}=T_eG$, and $v$ in the right side is viewed as a vector in $\mathfrak{m}=T_o(G/H)$}.\medskip

{\bf Fact 3}\quad Obviously each $\partial_{u^i}$ is tangent to $\mathcal{L}$. By the observation $\widetilde{V}_i u^j=\widetilde{V}_\alpha u^j=0$ for all $i,j,\alpha$ (see the third line in (\ref{063})), all $\widetilde{V}_i$'s and $\widetilde{V}_\alpha$'s are tangent to $\mathcal{L}$ as well. Counting the dimension, we see that
at each $\overline{p}\in\mathcal{L}\backslash0$,  $$\partial_{u^i}, \widetilde{V}_i,\ \forall i, \mbox{ and }
\widetilde{V}_\alpha, \ \forall \alpha\mbox{, provide a basis for }T_{\overline{p}}(T\mathcal{L}\backslash0).$$

The tangent map $\pi_*:T_gG\rightarrow T_{g\cdot o}(G/H)$ maps to $(L_g)_*(\mathfrak{h})$ to 0, so
$(\pi_*)_*$ maps each $\partial_{u^\alpha}$ to 0. To summarize, at any $\overline{p}\in\mathcal{L}$, $$\Phi^{T(T\overline{M})}_{T\mathcal{L}} \mbox{ maps every }\partial_{u^\alpha} \mbox{ to }0\mbox{, and fixes each } \partial_{u^i},
\widetilde{V}_i\mbox{ and }\widetilde{V}_\alpha.$$

{\bf Fact 4} \quad For the linear sub-bundle $\mathcal{H}'$ of $T(\mathcal{L}\backslash0)$ (see Section \ref{subsection-3-5}, same below for $\mathcal{K}$), its fiber $\mathcal{H}'_{\overline{p}}$ at each
$\overline{p}\in\mathcal{L}_g\backslash0$ consists $\Phi^{T(T\overline{M})}_{T\mathcal{L}}
(\widetilde{v}^{\overline{\mathcal{H}}})$ for all
$v\in \mathcal{L}_g$. So Fact 1 implies that at each $\overline{p}\in\mathcal{L}\backslash0$,
$$\Phi^{T(T\overline{M})}_{T\mathcal{L}}
(\widetilde{U}_i^{\overline{\mathcal{H}}})\mbox{ for all }i\mbox{
provides a basis for }\mathcal{H}'_{\overline{p}}.$$

{\bf Fact 5} \quad For the linear sub-bundle $\mathcal{K}$
of $T(TG)$, its fiber $\mathcal{K}_{\overline{p}}$ at each
$\overline{p}\in TG$ is $\ker (\pi_*)_*$, which can be presented as $\mathrm{span}\{\overline{\partial}_{z^\alpha},\overline{\partial}_{w^\alpha},\forall \alpha\}$ in a standard local coordinate $(x^i,z^\alpha,y^j,w^\beta)$ for $\pi$ (see Section \ref{subsection-3-2}).

From Fact 3,
we have already seen that  $\mathrm{span}\{\partial_{u^\alpha}|_{\overline{p}},\forall\alpha\} =\mathrm{span}\{\overline{\partial}_{w^\beta},\forall \beta\}\subset\mathcal{K}_{\overline{p}}$. Since $z=(z^\alpha)$ is the local coordinate on every $gH$, each $U_\alpha$ is a $C^\infty$-linear combination of $\overline{\partial}_{z^\alpha}$'s.  Then its complete lifting $\widetilde{U}_\alpha$ is a $C^\infty$-linear combination of
$\overline{\partial}_{z^\alpha}$'s and $\overline{\partial}_{w^\beta}$'s. Counting the dimension,
we see that at each $\overline{p}\in\mathcal{L}\backslash0$,
$$\partial_{u^\alpha}\mbox{ and }\widetilde{U}_\alpha\mbox{ for all }\alpha\mbox{ provide a basis of } \mathcal{K}_{\overline{p}}.$$

{\bf Fact 6}\quad
$\overline{\mathbf{G}}$ can be presented as $\overline{\mathbf{G}}=\overline{\mathbf{G}}_0-
\overline{\mathbf{H}}$,
    in which
$\overline{\mathbf{G}}_0=u^i\widetilde{U}_i
+u^\alpha\widetilde{U}_\alpha=v^i\widetilde{V}_i
+v^\alpha\widetilde{V}_\alpha$ is the canonical bi-invariant affine spray structure on $G$, and
$\overline{\mathbf{H}}=\overline{\mathbf{H}}^i\partial_{u^i}
+\overline{\mathbf{H}}^\alpha\partial_{u^\alpha}$ with $\overline{\eta}=\overline{\mathbf{H}}
|_{T_eG\backslash\{0\}}$ and
$\eta=\overline{\mathbf{H}}|_{\mathcal{L}_e\backslash\{0\}}$.
Since $\eta(y)=\eta(y_\mathfrak{m})$ in a conic open neighborhood of $\mathfrak{m}\backslash0$ in $\mathfrak{g}\backslash0$, we see that, there exists an open neighborhood of $\mathcal{L}\backslash0$ in $TG\backslash0$, where
$$
\overline{\mathbf{H}}^\alpha\equiv0,\ \forall\alpha,\mbox{ and }
\tfrac{\partial}{\partial u^\alpha}\overline{\mathbf{H}}^i\equiv0,\ \forall i,\alpha.$$

\subsection{Description for parallel translations}

Let $(G/H,\mathbf{G})$ be a homogeneous spray manifold with the linear decomposition $\mathfrak{g}=\mathfrak{h}+\mathfrak{m}$, $\eta$ the spray vector field and $N$ the connection operator.

Firstly, we consider the linearly parallel translation along a smooth curve $c(t)$ on $G/H$ with $c(0)=o$ and nowhere-vanishing $\dot{c}(t)$. Let $\overline{c}(t)$ be the smooth curve on $G$ provided by Lemma \ref{lemma-5-1} for $c(t)$, satisfying $\overline{c}(0)=e$, $c(t)=\overline{c}(t)\cdot o$ and
$(L_{\overline{c}(t)^{-1}})_*(\dot{\overline{c}}(t))\in\mathfrak{m}$
everywhere. The range of the parameter $t$ is a suitable open interval $(a,b)$ with $a<0<b$, such that both $c(t)$ and $\overline{c}(t)$ are defined.

The following theorem provide the ODE on $\mathfrak{m}$
describing the linearly parallel translation along $c(t)$.

\begin{theorem}\label{theorem-6-1}
Let $(G/H,\mathbf{G})$ be a homogeneous spray manifold with a linear decomposition $\mathfrak{g}=\mathfrak{h}+\mathfrak{m}$,
$c(t)$ any smooth curve on $G/H$ with $c(0)=o$ and nowhere-vanishing $\dot{c}(t)$, and $\overline{c}(t)$ the smooth curve on $G$ satisfying $\overline{c}(0)=e$, $c(t)=\overline{c}(t)\cdot o$ and $y(t)=(L_{\overline{c}(t)^{-1}})_*
(\overline{c}(t))\in\mathfrak{m}\backslash\{0\}$ for each $t$.
Then for any smooth vector field $W(t)=(\overline{c}(t))_*(w(t))$ along $c(t)$,
we have
\begin{equation*}\label{069}
D_{\dot{c}(t)}W(t)=(\overline{c}(t))_*
(\tfrac{{\rm d}}{{\rm d}t}w(t)+N(y(t),w(t))+[y(t),w(t)]_\mathfrak{m}).
\end{equation*}
In particular, $W(t)$ is linearly parallel along $c(t)$ if and only if
$w(t)$ satisfies
$$\tfrac{{\rm d}}{{\rm d}t}w(t)+N(y(t),w(t))+[y(t),w(t)]_\mathfrak{m}=0$$
everywhere.
\end{theorem}
\begin{proof} By Fact 2, the lifting of $W(t)=(\overline{c}(t))_*(w(t))$ along $\overline{c}(t)$ in $\mathcal{L}$ is
$$\overline{W}(t)=(L_{\overline{c}(t)})_*(w(t))=
w^i(t)U_i(\overline{c}(t)).$$
For the same reason, we can present $\dot{\overline{c}}(t)$
as
\begin{equation*}
\dot{\overline{c}}(t)=(L_{\overline{c}(t)})_*(y(t))
=u^i(t)U_i(\overline{c}(t)).
\end{equation*}
Denote $\overline{D}$ the linearly covariant derivative on $(G,\overline{\mathbf{G}})$. Then using the first statement in
Theorem \ref{theorem-4-4}, we get
\begin{eqnarray}\label{068}
\overline{D}_{\dot{\overline{c}}(t)}\overline{W}(t)&=&
(\tfrac{{\rm d}}{{\rm d}t}w^l(t)+\tfrac12 w^j(t)
\tfrac{\partial}{\partial u^j}\overline{\mathbf{H}}^l(\overline{c}(t),\dot{\overline{c}}(t))+
\tfrac12w^j(t)u^k(t)c_{kj}^l)U_l(\overline{c}(t))\nonumber\\
& &+(\tfrac12 w^j(t)\tfrac{\partial}{\partial u^j}
\overline{\mathbf{H}}^\alpha(\overline{c}(t),\dot{\overline{c}}(t))+
\tfrac12 w^j(t)u^k(t)c_{kj}^\alpha) U_\alpha(\overline{c}(t))\nonumber\\
&=& (\tfrac{{\rm d}}{{\rm d}t}w^l(t)+\tfrac12 w^j(t)
\tfrac{\partial}{\partial u^j}\overline{\mathbf{H}}^l(\overline{c}(t),\dot{\overline{c}}(t))+
\tfrac12w^j(t)u^k(t)c_{kj}^l)U_l(\overline{c}(t))\nonumber\\
& &+\tfrac12 w^j(t)u^k(t)c_{kj}^\alpha U_\alpha(\overline{c}(t)),
\end{eqnarray}
in which we have used Fact 6 for the second equality, i.e., $\overline{\mathbf{H}}^\alpha$ vanishes around $(\overline{c}(t),\dot{\overline{c}}(t))
\in\mathcal{L}\backslash\{0\}$.
Using the connection operator and Lie bracket of $\mathfrak{g}$, (\ref{068}) can be translated to
$$\overline{D}_{\dot{\overline{c}}(t)}
\overline{W}(t)=(L_{\overline{c}(t)})_*(\tfrac{{\rm d}}{{\rm d}t}w(t)+
N(y(t),w(t))+[y(t),w(t)]_\mathfrak{m}+
\tfrac12[y(t),w(t)]_\mathfrak{h}),$$
Using Fact 1, we get
\begin{equation}\label{303}
\Phi^{T\overline{M}}_\mathcal{L}
(\overline{D}_{\dot{\overline{c}}(t)}
\overline{W}(t))=(L_{\overline{c}(t)})_*(\tfrac{{\rm d}}{{\rm d}t}w(t)+
N(y(t),w(t))+[y(t),w(t)]_\mathfrak{m}).
\end{equation}
By Lemma \ref{lemma-3-5}, $\Phi^{\mathcal{L}}_{TM}(\Phi^{T\overline{M}}_\mathcal{L}
(\overline{D}_{\dot{\overline{c}}(t)}
\overline{W}(t)))=
D_{\dot{c}(t)}W(t)$. Using Fact 2,
we see from (\ref{303}) that
$$D_{\dot{c}(t)}W(t)=(\overline{c}(t))_*
(\tfrac{{\rm d}}{{\rm d}t}w(t)+N(y(t),w(t))+[y(t),w(t)]_\mathfrak{m}),$$
which proves the first statement in Theorem \ref{theorem-6-1}.

The second statement of Theorem \ref{theorem-6-1} follows after the first immediately.
\end{proof}\medskip

Some applications of Theorem \ref{theorem-6-1} will be discussed in Section \ref{subsection-6-2} and Section \ref{subsection-6-3} (see Theorem \ref{theorem-6-3}, Theorem \ref{theorem-6-4} and Corollary \ref{corollary-6-7}).

Nextly, we keep all assumptions and notations for $c(t)$ and $\overline{c}(t)$, and consider the nonlinearly parallel translation. The following theorem provide the wanted ODE on $\mathfrak{m}\backslash0$.

\begin{theorem}\label{theorem-6-2}
Let $(G/H,\mathbf{G})$ be a homogeneous spray manifold with a linear decomposition $\mathfrak{g}=\mathfrak{h}+\mathfrak{m}$,
$c(t)$
any smooth curve on $G/H$ with $c(0)=o$ and nowhere-vanishing $\dot{c}(t)$, and $\overline{c}(t)$ the smooth curve on $G$ satisfying $\overline{c}(0)=e$, $c(t)=\overline{c}(t)\cdot o$ and $w(t)=(L_{\overline{c}(t)^{-1}})_*
(\overline{c}(t))\in\mathfrak{m}\backslash\{0\}$ for each $t$.
Suppose $Y(t)=(\overline{c}(t))_*(y(t))$ is a nowhere-vanishing smooth vector field along $c(t)$. Then $Y(t)$ is nonlinearly parallel
along $c(t)$ if and only if $y(t)$ satisfies
\begin{equation*}
\tfrac{{\rm d}}{{\rm d}t}y(t)+N(y(t),w(t))=0
\end{equation*}
everywhere.
\end{theorem}
\begin{proof} 
On the submanifold $\overline{N}=\cup_{t}
(T_{\overline{c}(t)}G\backslash\{0\})$, we have the global coordinate $(t,u^1,\cdots,u^m)$ for $y=u^i U_i(\overline{c}(t))+u^\alpha U_\alpha(\overline{c}(t))\in T_{\overline{c}(t)}G\backslash\{0\}$.
The submanifold $N_\mathcal{L}=\cup_t (\mathcal{L}_{\overline{c}(t)}\backslash\{0\})
\subset \overline{N}$ corresponds to  $u^\alpha=0$, $\forall \alpha$.

On $\overline{N}$, we have
the global frame $\{\partial_{t},\partial_{u^1},\cdots,\partial_{u^m}\}$. At each point $y\in N_\mathcal{L}\cap T_{c(t)}G\backslash\{0\}$, $\partial_t$ and all $\partial_{u^i}$'s are tangent to $N_\mathcal{L}$, and each $\partial_{u^\alpha}$ is contained
in $\ker(\pi_*)\cap T_{\overline{c}(t)}G$. So $\Phi^{T(T\overline{M})|_{\mathcal{L}}}_{T\mathcal{L}}$ fixes $\partial_t$ and each $\partial_{u^i}$, and by Fact 3, it maps each $\partial_{u^\alpha}$ to 0.

Denote $w(t)=w^i(t)e_i(t)$, then we have
$\dot{\overline{c}}(t)=(L_{\overline{c}(t)})_*(w(t))=
(L_{\overline{c}(t)})_*(w^i(t)e_i)$.
Using (2) in Lemma \ref{lemma-4-8}, at each $(\overline{c}(t),y)\in N_\mathcal{L}$ with $y=u^i U_i(\overline{c}(t))$, we can get
\begin{eqnarray}\label{070}
\Phi^{T(T\overline{M})|_{\mathcal{L}}}_{T\mathcal{L}}
(\widetilde{\dot{\overline{c}}(t)}^{\overline{\mathcal{H}}}
|_{\mathcal{L}})
&=&\Phi^{T(T\overline{M})|_{\mathcal{L}}}_{T\mathcal{L}}
(\partial_t+(-\tfrac12 w^i(t)\tfrac{\partial}{\partial u^i}\overline{\mathbf{H}}^j(\overline{c}(t),y)+\tfrac12 w^i(t)u^p(t)c^j_{pi})\partial_{u^j}\nonumber\\
& &+
(-\tfrac12 w^i(t)\tfrac{\partial}{\partial u^i}\overline{\mathbf{H}}^\alpha(\overline{c}(t),y)+\tfrac12 w^i(t)u^p c^\alpha_{pi})\partial_{u^\alpha})\nonumber\\
&=&\partial_t+(-\tfrac12 w^i(t)\tfrac{\partial}{\partial u^i}\overline{\mathbf{H}}^j(\overline{c}(t),y)+\tfrac12 w^i(t)u^p(t)c^j_{pi})\partial_{u^j}.
\end{eqnarray}

Now we prove the equivalence in Theorem \ref{theorem-6-2}.

Suppose that $Y(t)=(\overline{c}(t))_*(y(t))$ with $y(t)=u^i(t)e_i$ is a nonlinearly parallel vector field along $c(t)$. Denote $\overline{Y}(t)=(\Phi^{N_\mathcal{L}}_N)^{-1}(Y(t))$. Since $\Phi^{N_\mathcal{L}}_\mathcal{L}=
\Phi^{\mathcal{L}\backslash0}_{TM\backslash0}
|_{N_\mathcal{L}}$,  
$\overline{Y}(t)$ is the
lifting of $Y(t)$ in $\mathcal{L}$ along $\overline{c}(t)$. Then Fact 2 provides
$\overline{Y}(t)=(L_{\overline{c}(t)})_*(y(t))
=(L_{\overline{c}(t)})_*(u^i(t)e_i)$.
If $Y(t)$ and $\overline{Y}(t)$ are viewed as curves on
$N=\cup_{t}(T_{c(t)}(G/H)\backslash\{0\})$ and
$N_\mathcal{L}=\cup_t(\mathcal{L}_{\overline{c}(t)}\backslash\{0\})$ respectively, $Y(t)$ is an integral curve of $\widetilde{\dot{c}(t)}^\mathcal{H}$, and by Lemma \ref{lemma-3-6}, $\overline{Y}(t)=(t,u^1(t),\cdots,u^n(t),0,\cdots,0)$ is an integral curve of the smooth vector field in (\ref{070}).
Then we see the ODE
$$\tfrac{{\rm d}}{{\rm d}t}u^i(t)=-\tfrac12 w^i(t)\tfrac{\partial}{\partial u^i}\overline{\mathbf{H}}^j(\overline{c}(t),y)+\tfrac12 w^i(t)u^p(t)c^j_{pi},\quad\forall i,$$
which can be translated to $$\tfrac{{\rm d}}{{\rm d}t}y(t)=-\tfrac12 D\eta(y(t),w(t))+\tfrac12
[y(t),w(t)]_\mathfrak{m}=-N(y(t),w(t)),$$
using the connection operator and Lie bracket of $\mathfrak{g}$.

To summarize, above argument proves one side of the equivalence, and
the other side can be proved similarly.
\end{proof}

In the special case that $w(t)\equiv w\in\mathfrak{m}\backslash\{0\}$, the smooth tangent vector field $-N(\cdot,w)$ on $\mathfrak{m}\backslash0$ generates a one-parameter subgroup of diffeomorphisms $\rho_t$. Applying Theorem \ref{theorem-6-2}, $\rho_t$ can be explained by
the nonlinearly parallel translation along $c(t)=(\exp tw) \cdot o$ (notice that the corresponding $\overline{c}(t)=\exp tw$ is defined for $t\in\mathbb{R}$), i.e.,

\begin{corollary} \label{corollary-6-6}
Let $(G/H,\mathbf{G})$ be a homogeneous spray manifold with a linear decomposition $\mathfrak{g}=\mathfrak{h}+\mathfrak{m}$, and $N$ the connection operator. Then for any fixed $w\in\mathfrak{m}\backslash\{0\}$, the one-parameter subgroup $\rho_t$ generated by the smooth tangent vector field $N(\cdot,w)$ on $\mathfrak{m}\backslash0$ can be presented
as $\rho_t=(\exp(-tw))_*\circ\mathrm{P}^{nl}_{c;0,t}$, in which
$\mathrm{P}^{nl}_{c;0,t}$ is the nonlinearly parallel translation along $c(t)=(\exp tw)\cdot o$ from $c(0)=o$ to $c(t)$.
\end{corollary}

\begin{remark}\label{remark-6-5}
Theorem \ref{theorem-6-2} and Corollary \ref{corollary-6-6} suggest we study the Lie algebra $\mathfrak{H}$ generated by $N(\cdot,w)$ for all $w\in\mathfrak{m}$, using the Lie bracket between
smooth tangent vector fields. It seems interesting to
explore the relation between $\mathfrak{H}$ and the  restricted holonomy of $(G/H,\mathbf{G})$ and look for homogeneous spray structures $\mathbf{G}$ which are not affine, and have finite dimensional $\mathfrak{H}$. That may shed light on the Landsberg problem for homogeneous Finsler manifolds \cite{Ma1996,XD2021}. See also Section 4.2 in \cite{Xu2021-2}, where similar observations are made
for a left invariant spray structure.
\end{remark}
\subsection{Homogeneous S-curvature and Landsberg curvature formulae}
\label{subsection-6-2}

Firstly,
we use linearly parallel translation to generalize L. Huang's homogeneous S-curvature formula (see Proposition 4.6 in \cite{Hu2015-1}).

Let $(G/H,\mathbf{G})$ be a homogeneous spray structure with a linear decomposition $\mathfrak{g}=\mathfrak{h}+\mathfrak{m}$. We further assume that the
$\mathrm{Ad}(H)$-action on $\mathfrak{g}/\mathfrak{h}$
is unimodular, i.e., for each $g\in H$, the determinant of $\mathrm{Ad}(g):\mathfrak{g}/\mathfrak{h}\rightarrow
\mathfrak{g}/\mathfrak{h}$ is $\pm1$. Then there exists a
$G$-invariant smooth measure ${\rm d}\mu$ on $G/H$ which is unique
up to a scalar. The S-curvature for $\mathbf{G}$ and ${\rm d}\mu$ is given by the following theorem.

\begin{theorem}\label{theorem-6-3}
Let $(G/H,\mathbf{G})$ be a homogeneous spray structure with a linear decomposition $\mathfrak{g}=\mathfrak{h}+\mathfrak{m}$. Suppose that the $\mathrm{Ad}(H)$-action on $\mathfrak{g}/\mathfrak{h}$ is unimodular. Then for $\mathbf{G}$ and any $G$-invariant smooth measure ${\rm d}\mu$ on $G/H$, the S-curvature satisfies
\begin{equation*}
\mathbf{S}(o,y)=\mathrm{Tr}_\mathbb{R}(N(y,\cdot)+
\mathrm{ad}_\mathfrak{m}(y)),
\end{equation*}
for any $y\in\mathfrak{m}\backslash\{0\}=
T_o(G/H)\backslash\{0\}$. Here $\mathrm{ad}_\mathfrak{m}(y):\mathfrak{m}\rightarrow\mathfrak{m}$ is the linear map $w\mapsto[y,w]_\mathfrak{m}$.
\end{theorem}

\begin{proof}Let $c(t)$ be the geodesic on $(G/H,\mathbf{G})$ with $c(0)=o$ and $\dot{c}(0)=y\in\mathfrak{m}\backslash\{0\}$, and
$\overline{c}(t)$ the smooth curve on $G$ provided by Lemma \ref{lemma-5-1}, satisfying $\overline{c}(0)=e$, $c(t)=\overline{c}(t)\cdot o$ and
$(L_{\overline{c}(t)^{-1}})_*
(\dot{\overline{c}}(t))\in\mathfrak{m}\backslash\{0\}$
around $t=0$. Denote $E_i(t)=({\overline{c}(t)})_*(e_i)$, $\forall 1\leq i\leq n$. Then $\{E_1(t),\cdots,E_n(t)\}$ is a smooth frame along $c(t)$. By the $G$-invariancy of ${\rm d}\mu$,
${\rm d}\mu(E_1(t),\cdots,E_n(t))$ is a nonzero constant function.
By Lemma \ref{lemma-2-2},
$\mathbf{S}(o,y)$ is the real trace for linear map $A(w)=-\tfrac{{\rm d}}{{\rm d}t}|_{t=0}w(t)$, in which the smooth curve $w(t)$ in $\mathfrak{m}$ is determined by the linearly parallel vector field $W(t)=(\overline{c}(t))_*(w(t))$ along $c(t)$ with
$W(0)=w$. By the second statement in Theorem \ref{theorem-6-1}, we have
$-\tfrac{{\rm d}}{{\rm d}t}|_{t=0}w(t)=N(y,w)+[y,w]_\mathfrak{m}$,
so $\mathbf{S}(o,y)=\mathrm{Tr}_\mathbb{R}(N(y,\cdot)+
\mathrm{ad}_\mathfrak{m}(y))$, which ends the proof.
\end{proof}\medskip

Nextly, we use the linearly parallel translation to re-prove L. Huang's Landsberg curvature formula curvature formula for a homogeneous Finsler manifold.
\begin{theorem}\label{theorem-6-4}
Let $(G/H,F)$ be a homogeneous Finsler manifold with a linear decomposition $\mathfrak{g}=\mathfrak{h}+\mathfrak{m}$. Then for any $y\in\mathfrak{m}\backslash\{0\}=T_{o}(G/H)\backslash\{0\}$
and $w\in \mathfrak{m}=T_o(G/H)$, the Landsberg curvature of $(G/H,F)$ satisfies
\begin{equation}\label{073}
\mathbf{L}_y(w,w,w)=3\mathbf{C}_y(w,w,[w,y]_\mathfrak{m}-
N(y,w))-\mathbf{C}_y(w,w,w,\eta(y)).
\end{equation}
\end{theorem}

Here the Cartan tensor $\mathbf{C}_y(\cdot,\cdot,\cdot)$ and the Cartan tensor $\mathbf{C}_y(\cdot,\cdot,\cdot,\cdot)$ with a higher order are defined for the Minkowski norm $F=F(o,\cdot)$ on $\mathfrak{m}=T_o(G/H)$. In particular, we have
$$
\mathbf{C}_y(w_1,w_2,w_3,w_4)=\tfrac{{\rm d}}{{\rm d}t}|_{t=0}C_{y+tw_4}
(w_1,w_2,w_3),\quad\forall w_1,w_2,w_3,w_4\in\mathfrak{m}.$$

\begin{proof}
Let $\eta$ and $N$ be the spray vector field and connection operator for the geodesic spray of
$(G/H,F)$. Notice that the given decomposition $\mathfrak{g}=\mathfrak{h}+\mathfrak{m}$ may not be reductive, so $\eta$ and $N$ may not be those in (\ref{301}) and (\ref{302}) respectively.

Let $c(t)$ be a geodesic on $(G/H,F)$ with $c(0)=o$ and $\dot{c}(0)=y\in\mathfrak{m}\backslash\{0\}$, and $\overline{c}(t)$  the smooth curve on $G$ provided by Lemma \ref{lemma-5-1}, satisfying $\overline{c}(0)=e$, $c(t)=\overline{c}(t)\cdot o$ and
$y(t)=
(L_{\overline{c}(t)^{-1}})_*(\dot{\overline{c}}(t))
=(\overline{c}(t)^{-1})_*(\dot{c}(t))
\in\mathfrak{m}
\backslash\{0\}$ around $t=0$.
Let $W(t)=({\overline{c}(t)})_*(w(t))$ be the linearly parallel vector field along $c(t)$ with $W(0)=w(0)=w$.
Using the formula (\ref{071}) for Landsberg curvature and the $G$-invariancy of $F$, we get
\begin{eqnarray}\label{072}
\mathbf{L}_y(w,w,w)&=&\tfrac{{\rm d}}{{\rm d}t}|_{t=0}
\mathbf{C}_{\dot{c}(t)}(W(t),W(t),W(t))\nonumber\\
&=&\tfrac{{\rm d}}{{\rm d}t}|_{t=0}
\mathbf{C}_{(\overline{c}(t)^{-1})_*(\dot{c}(t))}
((\overline{c}(t)^{-1})_*(W(t)), (\overline{c}(t)^{-1})_*(W(t)), (\overline{c}(t)^{-1})_*(W(t)))\nonumber\\
&=&\tfrac{{\rm d}}{{\rm d}t}|_{t=0}\mathbf{C}_{y(t)}(w(t),w(t),w(t)).
\end{eqnarray}
By Theorem \ref{theorem-5-3} and the second statement in Theorem \ref{theorem-6-1}, we have $\tfrac{{\rm d}}{{\rm d}t}y(t)=-\eta(y(t))$
and $\tfrac{{\rm d}}{{\rm d}t}w(t)=-N(y(t),w(t))-[y(t),w(t)]_\mathfrak{m}$. So
the calculation (\ref{072}) can be continued as
\begin{eqnarray*}
\mathbf{L}_y(w,w,w)&=&\tfrac{{\rm d}}{{\rm d}t}|_{t=0}
\mathbf{C}_{y(t)}(w(t),w(t),w(t))\\
&=&3\mathbf{C}_{y}
(w,w,\tfrac{{\rm d}}{{\rm d}t}|_{t=0}w(t))+
\mathbf{C}_y(w,w,w,\tfrac{{\rm d}}{{\rm d}t}|_{t=0}y(t))\\
&=&3\mathbf{C}_y(w,w,w,[w,y]_\mathfrak{m}-N(y,w))
-\mathbf{C}_y(w,w,w,\eta(y)),
\end{eqnarray*}
which ends the proof.
\end{proof}\medskip

Though the formula (\ref{073}) in Theorem \ref{theorem-6-4} looks the same as the one in Proposition 4.6 of \cite{Hu2015-1},
there is still a slight refinement because we do not need the reductive property here.

\subsection{Homogeneous Riemann curvature formula}
\label{subsection-6-3}
Now we use the submersion technique to calculate the Riemann curvature of $(G/H,\mathbf{G})$. As a preparation, we need the following lemma.

\begin{lemma} \label{lemma-6-5}
At each point of $\mathcal{L}\backslash0$, we have
$$\Phi^{T(T\overline{M})|_{\mathcal{L}}}_{T\mathcal{L}}
(\widetilde{U}^{\overline{\mathcal{H}}}_q)=
\widetilde{U}_q- (\tfrac12\tfrac{\partial}{\partial u^q}
\overline{\mathbf{H}}^j-\tfrac12 u^kc^j_{qk})\partial_{u^j}
-c^\alpha_{iq}u^i\partial_{u^\alpha}.$$
\end{lemma}

\begin{proof} 
%
%
%
%
Firstly, we use (1) in Lemma \ref{lemma-4-8} to get
\begin{eqnarray*}
\widetilde{U}_q^{\overline{\mathcal{H}}}
=\widetilde{U}_q- (\tfrac12\tfrac{\partial}{\partial u^q}
\overline{\mathbf{H}}^j-
\tfrac12 u^kc^j_{qk})\partial_{u^j}
+\tfrac12u^jc^{\alpha}_{qj}\partial_{u^\alpha},
\end{eqnarray*}
at any $\overline{p}=(g,y)\in\mathcal{L}\backslash0$ with $y=u^iU_i(g)$. Here
we have used $u^\alpha=0$ for all $\alpha$ at $\overline{p}$, and Fact 6, i.e, $\overline{\mathbf{H}}^\alpha\equiv0$ around $\overline{p}$.
Nextly, we use (2) in Lemma \ref{lemma-4-1} and get
\begin{equation*}
\widetilde{U}_q=\phi_q^i\widetilde{V}_i+
\phi_q^\alpha\widetilde{V}_\alpha+c^j_{iq}u^i\partial_{u^j}
+c^\alpha_{iq}u^i\partial_{u^\alpha}
\end{equation*}
at $\overline{p}$. Finally, Fact 3 provides
\begin{eqnarray*}
\Phi^{T(T\overline{M})|_{T\mathcal{L}}}_\mathcal{L}
(\widetilde{U}_q^{\overline{\mathcal{\mathcal{H}}}})
&=&\Phi^{T(T\overline{M})|_{T\mathcal{L}}}_\mathcal{L}
(\widetilde{U}_q)- (\tfrac12\tfrac{\partial}{\partial u^q}
\overline{\mathbf{H}}^j-\tfrac12 u^kc^j_{qk})\partial_{u^j}
\nonumber\\
&=&\widetilde{U}_q-c^\alpha_{iq}u^i\partial_{u^\alpha}- (\tfrac12\tfrac{\partial}{\partial u^q}
\overline{\mathbf{H}}^j-\tfrac12 u^kc^j_{qk})\partial_{u^j},\nonumber
\end{eqnarray*}
which ends the proof.
\end{proof}\medskip

%
%
%

The following theorem provides the Riemann curvature formula for $(G/H,\mathbf{G})$.

\begin{theorem}\label{theorem-6-6}
Let $(G/H,\mathbf{G})$ be a homogeneous spray manifold with a linear decomposition $\mathfrak{g}=\mathfrak{h}+\mathfrak{m}$. Then for any $y\in\mathfrak{m}\backslash\{0\}=T_o(G/H)\backslash\{0\}$, the Riemann curvature $\mathbf{R}_y:\mathfrak{m}\rightarrow\mathfrak{m}$ satisfies
\begin{eqnarray}
\mathbf{R}_y(w)&=&[y,[w,y]_\mathfrak{h}]_\mathfrak{m}
+DN(\eta(y),y,w)-N(y,N(y,w))\nonumber\\
& &+N(y,[y,w]_\mathfrak{m})
-[y,N(y,w)]_\mathfrak{m},\label{100}
\end{eqnarray}
in which $DN(\eta(y),y,w)=\tfrac{{\rm d}}{{\rm d}t}|_{t=0}N(y+t\eta(y),w)$.
\end{theorem}

\begin{proof} When the spray vector field
$\eta:\mathfrak{m}\backslash\{0\}\rightarrow\mathfrak{m}$ is viewed as the smooth vector field $\overline{\mathbf{H}}|_{\mathfrak{m}\backslash\{0\}}=
\overline{\mathbf{H}}^i(e,\cdot)\partial_{u^i}$, the connection operator $N(y,w)$ with $y=u^ie_i\in\mathfrak{m}\backslash\{0\}$ and $w=w^ie_i\in\mathfrak{m}$ is then identified with
$$\tfrac12w^i\tfrac{\partial}{\partial u^i}
\overline{\mathbf{H}}^j(e,y)\partial_{u^j}-
\tfrac12 w^i u^j c_{ji}^k\partial_{u^k}.$$
So we can use the $G$-invariancy of $\mathbf{G}$ to translate (\ref{100}) to
\begin{eqnarray}
\mathbf{R}_y(w)&=&w^q(c^\alpha_{qj}u^iu^jc^k_{i\alpha}
+\tfrac34 c^r_{pq}u^p\tfrac{\partial}{\partial u^r}
\overline{\mathbf{H}}^i
+\tfrac12 c^i_{qj}
\overline{\mathbf{H}}^j
+\tfrac12\overline{\mathbf{H}}^p
\tfrac{\partial^2}{\partial u^p\partial u^q}
\overline{\mathbf{H}}^i\nonumber\\
& &
-\tfrac14\tfrac{\partial}{\partial u^q}
\overline{\mathbf{H}}^p
\tfrac{\partial}{\partial u^p}\overline{\mathbf{H}}^i
+\tfrac14c^i_{pr}u^r\tfrac{\partial}{\partial u^q}
\overline{\mathbf{H}}^p-\tfrac14 c^p_{qj}c^i_{pr}u^j u^r) g_* (e_i)\label{105}
\end{eqnarray}
where $y=u^ig_* (e_i), w=w^i g_*(e_i)\in T_{g\cdot o}(G/H) $ and $y\neq0$.

By (\ref{050}), Lemma \ref{lemma-3-7}, Fact 4 and Fact 5, (\ref{105}) is equivalent to
\begin{eqnarray}\label{101}
[\overline{\mathbf{G}}|_{\mathcal{L}\backslash0},
\Phi^{T(T\overline{M})|_{\mathcal{L}}}_{\mathcal{L}}
(\widetilde{U}_q^{\overline{\mathbf{H}}})]
&=&(c^\alpha_{qj}u^iu^jc^k_{i\alpha}
+\tfrac34 c^r_{pq}u^p\tfrac{\partial}{\partial u^r}
\overline{\mathbf{H}}^i
+\tfrac12 c^i_{qj}
\overline{\mathbf{H}}^j
+\tfrac12\overline{\mathbf{H}}^p
\tfrac{\partial^2}{\partial u^p\partial u^q}
\overline{\mathbf{H}}^i\nonumber\\
& &
-\tfrac14\tfrac{\partial}{\partial u^q}
\overline{\mathbf{H}}^p
\tfrac{\partial}{\partial u^p}\overline{\mathbf{H}}^i
+\tfrac14c^i_{pr}u^r\tfrac{\partial}{\partial u^q}
\overline{\mathbf{H}}^p-\tfrac14 c^p_{qj}c^i_{pr}u^j u^r)
\partial_{u^i} \nonumber\\& &\mbox{(mod }
\widetilde{U}^{\overline{\mathcal{H}}}_i, \widetilde{U}_\alpha,\partial_{u^\beta},\forall i,\alpha,\beta\mbox{)}.
\end{eqnarray}
for each $q$. By Lemma \ref{lemma-6-5},  $\Phi^{T(T\overline{M})|_{\mathcal{L}}}_{\mathcal{L}}
(\widetilde{U}_q^{\overline{\mathbf{H}}})$ can be smoothly extended to $TG\backslash0$. So when we calculate the left side of (\ref{101}), we may first do it in
the neighborhood of $\mathcal{L}\backslash0$ where Fact 6 is valid, i.e., $\overline{\mathbf{H}}^\alpha$ and $\tfrac{\partial}{\partial u^\alpha}\overline{\mathbf{H}}^i$ vanish,
and then restrict the result to $\mathcal{L}\backslash0$, where $u^\alpha$ vanishes.

According to Lemma \ref{lemma-6-5}, we denote
$\Phi^{T(T\overline{M})|_{\mathcal{L}}}_{\mathcal{L}}
(\widetilde{U}_q^{\overline{\mathcal{H}}})=\overline{X}_1+\overline{X}_2$,
in which $\overline{X}_1=-c^\alpha_{iq}u^i\partial_{u^\alpha}$
and
$\overline{X}_2=\widetilde{U}_q-
(\tfrac12\tfrac{\partial}{\partial u^q}\overline{\mathbf{H}}^j-\tfrac12u^ic^j_{qi})
\partial_{u^j}$, at each $(g,y)\in\mathcal{L}\backslash0$ with
$y=u^iU_i(g)$.

Firstly, we calculate $[\overline{\mathbf{G}}|_{\mathcal{L}\backslash0},
\overline{X}_1]$
and get
\begin{eqnarray}
[\overline{\mathbf{G}}|_{\mathcal{L}\backslash0},
\overline{X}_1]&=&[u^i\widetilde{U}_i+
u^\alpha\widetilde{U}_\alpha-
\overline{\mathbf{H}}^i\partial_{u^i},
-c^\alpha_{jq}u^j\partial_{u^\alpha}]\nonumber\\
&=&-c^\alpha_{jq}u^iu^j[\widetilde{U}_i,\partial_{u^\alpha}]
\quad \mbox{(mod }\widetilde{U}_\alpha,\partial_{u^\beta},
\forall\alpha,\beta\mbox{)}
\nonumber\\
&=& c^\alpha_{qj}u^iu^j c^k_{i\alpha}\partial_{u^k}
\quad\mbox{(mod }\widetilde{U}_\alpha,\partial_{u^\beta},\forall \alpha,\beta\mbox{)},\label{102}
\end{eqnarray}
in which the last equality needs (4) of Lemma \ref{lemma-4-1}.

Nextly, we calculate $[\overline{\mathbf{G}},\overline{X}_2]$. The calculation is almost the same as that proving Theorem C in \cite{Xu2021-1}.

To be precise, we have
\begin{eqnarray}
& &[-\overline{\mathbf{H}}^i\partial_{u^i},
\widetilde{U}_q-(\tfrac12\tfrac{\partial}{\partial u^q}
\overline{\mathbf{H}}^j\partial_{u^j}-\tfrac12 u^kc^j_{qk}\partial_{u^j})]\nonumber\\
&=&\widetilde{U}_q\overline{\mathbf{H}}^i\partial_{u^i}
+\overline{\mathbf{H}}^i[\widetilde{U}_q,\partial_{u^i}]
+\tfrac12[\overline{\mathbf{H}}^i\partial_{u^i},
\tfrac{\partial}{\partial u^q}\overline{\mathbf{H}}^j\partial_{u^j}]
-\tfrac12c^j_{qk}[\overline{\mathbf{H}}^i\partial_{u^i},
u^k\partial_{u^j}]\nonumber\\
&=&c^r_{pq}u^p\tfrac{\partial}{\partial u^r}
\overline{\mathbf{H}}^i\partial_{u^i}+
c_{qp}^i\overline{\mathbf{H}}^p\partial_{u^i}
+\tfrac12\overline{\mathbf{H}}^p
\tfrac{\partial^2}{\partial u^p\partial u^q}\overline{\mathbf{H}}^i\partial_{u^i}-
\tfrac12\tfrac{\partial}{\partial u^q}\overline{\mathbf{H}}^p
\tfrac{\partial}{\partial u^p}\overline{\mathbf{H}^i}\partial_{u^i}\nonumber\\
& &-\tfrac12c^i_{qj}\overline{\mathbf{H}}^j\partial_{u^i}
+\tfrac12c^p_{qj}u^j\tfrac{\partial}{\partial
u^p}\overline{\mathbf{H}}^i\partial_{u^i}\quad\mbox{(mod }\partial_{u^\alpha},\forall \alpha\mbox{)}\nonumber\\
&=&\tfrac12 c^r_{pq}u^p\tfrac{\partial}{\partial u^r}
\overline{\mathbf{H}}^i\partial_{u^i}
+\tfrac12 c^i_{qj}\overline{\mathbf{H}}^j\partial_{u^i}
+\tfrac12\overline{\mathbf{H}}^p
\tfrac{\partial^2}{\partial u^p\partial u^q}
\overline{\mathbf{H}}^i\partial_{u^i}
-\tfrac12\tfrac{\partial}{\partial u^q}\overline{\mathbf{H}}^p
\tfrac{\partial}{\partial u^p}\overline{\mathbf{H}}^i
\partial_{u^i}\nonumber\\
& &\mbox{(mod }\partial_{u^\alpha},\forall \alpha\mbox{)}
,
\label{103}
\end{eqnarray}
where Fact 6, the left invariancy of $\overline{\mathbf{H}}^i$ and Lemma \ref{lemma-4-1} have been applied.

Using Lemma \ref{lemma-4-1},
we also have
\begin{eqnarray}
& &[\overline{\mathbf{G}}_0,
\widetilde{U}_q-(\tfrac12\tfrac{\partial}{\partial u^q}
\overline{\mathbf{H}}^i-\tfrac12 c^i_{qj})\partial_{u^i}]
\nonumber\\
&=&[\overline{\mathbf{G}}_0,\widetilde{U}_q]
-\tfrac12(\tfrac{\partial}{\partial u^q}
\overline{\mathbf{H}}^i-\tfrac12 c^i_{qj}u^j)[\overline{\mathbf{G}}_0,\partial_{u^i}]
-\tfrac12\overline{\mathbf{G}}_0(\tfrac{\partial}{\partial u^q}\overline{\mathbf{H}}^i-c^i_{qj}u^j)\partial_{u^i}
\nonumber\\
&=&-\tfrac12(\tfrac{\partial}{\partial u^q}
\overline{\mathbf{H}}^i-\tfrac12 c^i_{qj}u^j)[\overline{\mathbf{G}}_0,\partial_{u^i}]
=-\tfrac12(\tfrac{\partial}{\partial u^q}
\overline{\mathbf{H}}^i-\tfrac12 c^i_{qj}u^j)[u^p\widetilde{U}_p,\partial_{u^i}]
\nonumber\\
&=&\tfrac12(\tfrac{\partial}{\partial u^q}
\overline{\mathbf{H}}^p-c^p_{qj}u^j)\widetilde{U}_p
-\tfrac12u^p(\tfrac{\partial}{\partial u^q}
\overline{\mathbf{H}}^i-c^i_{qj}u^j)[\widetilde{U}_p,
\partial_{u^i}]\nonumber\\
&=&\tfrac12(\tfrac{\partial}{\partial u^q}
\overline{\mathbf{H}}^p-c^p_{qj}u^j)(
\widetilde{U}^{\overline{\mathcal{H}}}+\tfrac12
(\tfrac{\partial}{\partial u^p}\overline{\mathbf{H}}^i
-c^i_{pr}u^r)\partial_{u^i})-\tfrac12 u^p(
\tfrac{\partial}{\partial u^q}\overline{\mathbf{H}}^r-c^r_{qj}u^j)c^i_{pr}\partial_{u^i}
\nonumber\\& &\mbox{(mod }\partial_{u^\alpha},\forall \alpha
\mbox{)}\nonumber\\
&=&\tfrac14(\tfrac{\partial}{\partial u^q}
\overline{\mathbf{H}}^p-c^p_{qj}u^j)
(\tfrac{\partial}{\partial u^p}\overline{\mathbf{H}}^i-c^i_{pr}u^r)\partial_{u^i}
-\tfrac12u^p(\tfrac{\partial}{\partial u^q}
\overline{\mathbf{H}}^r-c^r_{qj}u^j)c^i_{pr}\partial_{u^i}
\quad
\mbox{(mod }\widetilde{U}_i^{\overline{\mathcal{H}}},
\partial_{u^\alpha},
\forall i,\alpha\mbox{)}\nonumber\\
&=&\tfrac14\tfrac{\partial}{\partial u^q}
\overline{\mathbf{H}}^p
\tfrac{\partial}{\partial u^p}\overline{\mathbf{H}}^i
\partial_{u^i}
-\tfrac14 c^p_{qj}u^j\tfrac{\partial}{\partial u^p}
\overline{\mathbf{H}}^i\partial_{u^i}+
\tfrac14 c^i_{pr}u^r\tfrac{\partial}{\partial u^q}
\overline{\mathbf{H}}^p\partial_{u^i}-\tfrac14
c^p_{qj}c^i_{pr}u^ju^r\partial_{u^i}\nonumber\\
& &\mbox{(mod }\widetilde{U}^{\overline{\mathcal{H}}}_i,
\partial_{u^\alpha},\forall i,\alpha\mbox{)},\label{104}
\end{eqnarray}
where the first summand in the second line vanishes because $\overline{\mathbf{G}}_0=v^i\widetilde{V}_i+
v^\alpha\widetilde{V}_\alpha$ is right invariant, and the third summand in the second line  vanishes because $\overline{\mathbf{H}}^i$'s, $u^i$'s and $\partial_{u^i}$'s are left invariant, i.e.,
$\widetilde{V}_j\overline{\mathbf{H}}^i=
\widetilde{V}_\alpha\overline{\mathbf{H}}^i=0$, $\widetilde{V}_ju^i=\widetilde{V}_\alpha u^i=0$, and $[\widetilde{V}_j,\partial_{u^i}]=[\widetilde{V}_\alpha,\partial_{u^i}]
=0$, $\forall i,j,\alpha$.

Adding (\ref{102}), (\ref{103}) and (\ref{104}), we get (\ref{101}), which prove Theorem \ref{theorem-6-6}.
\end{proof}\medskip

\begin{remark}Though (\ref{100}) is not explicit given in \cite{Hu2015-1} for a homogeneous Finsler metric and a reductive decomposition, L. Huang showed in a private communication that it can be easily deduced from Proposition 3.3, (26) and (30) in \cite{Hu2015-1}.
\end{remark}

The thought in Theorem 1.2 of \cite{Xu2021-2} provides a new interpretation for the homogeneous Riemann curvature formula
(\ref{100}).

\begin{corollary}\label{corollary-6-7}
Let $(G/H,\mathbf{G})$ be a homogeneous spray manifold with a linear decomposition $\mathfrak{g}=\mathfrak{h}+\mathfrak{m}$.
Let $c(t)$ be a geodesic on $(G/H,\mathbf{G})$,
which is defined around $t=0$ and satisfies $c(0)=o$,
$\overline{c}(t)$ the smooth curve on $G$ provided by Lemma \ref{lemma-5-1} for $c(t)$, satisfying $\overline{c}(0)=e$, $c(t)=\overline{c}(t)\cdot o$ and
$y(t)=(L_{c(t)^{-1}})_*(\dot{c}(t))\in\mathfrak{m}\backslash\{0\}$
everywhere, and
$W(t)=(\overline{c}(t))_*(w(t))$
a linearly parallel vector field along $c(t)$, where $w(t)=w^i(t)e_i$ is viewed as a smooth vector field along
the smooth curve $y(t)$ in $\mathfrak{m}\backslash\{0\}$. Then $N(t)=N(y(t),w(t))$
 and $R(t)=\mathbf{R}_{y(t)}(w(t))$ satisfy
 \begin{eqnarray*}
 N(t)=[w(t),\eta]\quad\mbox{and}\quad R(t)=[y(t),[w(t),y(t)]_\mathfrak{h}]_\mathfrak{m}+
 [\eta,N(t)]
 \end{eqnarray*}
 when they are viewed as smooth vector fields along $y(t)$.
\end{corollary}

\begin{remark}\label{remark-6-10}
Here we need to slightly generalize the Lie bracket between two smooth vector fields as in \cite{Hu2015-2}. Let $X$ be a smooth vector field on $M$ and $Y(t)$ a smooth vector field along an integral curve $c(t)$ for $X$, then $[X,Y(t)]=-[Y(t),X]$ is a well defined smooth vector field along $c(t)$. When $c(t)$ is not constant, we may locally extend $Y(t)$ to a smooth vector field $Z$ on $M$, then $[X,Y(t)]=[X,Z]|_{c(t)}$ is independent of the extension. Using local coordinate, the bracket between $X=X^i\partial_{x^i}$ and
$Y(t)=Y^i(t)\partial_{x^i}|_{c(t)}$ can be presented as
\begin{equation}\label{305}
[X,Y(t)]=(\tfrac{{\rm d}}{{\rm d}t} Y^i(t)\partial_{x^i}-Y^i(t)\tfrac{\partial}{\partial x^i}X^j\partial_{x^j})|_{c(t)}.
\end{equation}
Notice that when $c(t)$ is constant, (\ref{305}) can still be used to
defined $[X,Y(t)]$ along $c(t)$, which is independent of the choice of
local coordinate.
\end{remark}

The proof of Corollary \ref{corollary-6-7} is very similar to that for Theorem 1.2 in \cite{Xu2021-2}.\medskip

\begin{proof}[Proof of Corollary \ref{corollary-6-7}]
By Theorem \ref{theorem-5-3}, $y(t)=(L_{c(t)^{-1}})_*(\dot{c}(t))$ for the geodesic $c(t)$ is an integral curve of $-\eta$. We view
$w(t)=(\overline{c}(t)^{-1})_*(W(t))$ as a smooth vector field along $y(t)$,
so the bracket $[-\eta,w(t)]$ is a well defined smooth vector field along $y(t)$ by Remark \ref{remark-6-10}. Using (\ref{305}), we get
\begin{eqnarray*}
[-\eta,w(t)]=\tfrac{{\rm d}}{{\rm d}t}w(t)+D\eta(y(t),w(t))
=\tfrac{{\rm d}}{{\rm d}t}w(t)+2N(y(t),w(t))+[y(t),w(t)]_\mathfrak{m}.
\end{eqnarray*}
By Theorem \ref{theorem-6-1}, we have
$\tfrac{{\rm d}}{{\rm d}t}w(t)+N(y(t),w(t))+[y(t),w(t)]_\mathfrak{m}=0$ for the linearly parallel vector field $W(t)=(\overline{c}(t))_*(w(t))$ along $c(t)$,
so
$N(t)=N(y(t),w(t))=[-\eta,w(t)]=[w(t),\eta]$.

Using the fact
$\tfrac{{\rm d}}{{\rm d}t}y(t)=-\eta(y(t))$ from Theorem \ref{theorem-5-3} and the linearity of $N(\cdot,\cdot)$ for the second entry,
we have
\begin{eqnarray}\label{306}
\tfrac{{\rm d}}{{\rm d}t}N(y(t),w(t))&=&
DN(-\eta,y(t),w(t))+N(y(t),\tfrac{{\rm d}}{{\rm d}t}w(t))
\nonumber\\
&=&-DN(\eta,y(t),w(t))+
N(y(t),-N(y(t),w(t))-[y(t),w(t)]_\mathfrak{m})
\nonumber\\
&=&-DN(\eta,y(t),w(t))-N(y(t),N(y(t),w(t)))\nonumber\\& &-
N(y(t),[y(t),w(t)]_\mathfrak{m}),
\end{eqnarray}
in which $DN(\eta,y(t),w(t))=\tfrac{{\rm d}}{{\rm d}s}|_{s=0}N(y(t)+s\eta(y(t)),w(t))$.
So
(\ref{305}) and (\ref{306}) imply
\begin{eqnarray*}
[\eta,N(t)]&=&-\tfrac{{\rm d}}{{\rm d}t}N(y(t),w(t))-
D\eta(y(t),N(y(t),w(t)))\\
&=&DN(\eta,y(t),w(t))+
N(y(t),N(y(t),w(t)))+N(y(t),[y(t),w(t)]_\mathfrak{m})\\
& &-(2N(y(t),N(y(t),w(t)))+[y(t),N(y(t),w(t))]_\mathfrak{m})\\
&=&DN(\eta,y(t),w(t))-N(y(t),N(y(t),w(t)))+
N(y(t),[y(t),w(t)]_\mathfrak{m})
\\& &-[y(t),N(y(t),w(t))]_\mathfrak{m}\\
&=&\mathbf{R}_{y(t)}(w(t))-
[y(t),[w(t),y(t)]_\mathfrak{h}]_\mathfrak{m}
=R(t)-[y(t),[w(t),y(t)]_\mathfrak{h}]_\mathfrak{m}.
\end{eqnarray*}
This ends the proof.
\end{proof}\medskip

\noindent
{\bf Acknowledgement}.
This paper is supported by Beijing Natural Science Foundation
(Z180004), National Natural Science
Foundation of China (12131012, 11821101),
The author sincerely thanks Libing Huang, Yuri G. Nikonorov, Ju Tan, Wolfgang Ziller and Joseph A. Wolf for helpful discussions.

\end{document}